\documentclass[12pt]{amsart}

\usepackage{amsxtra,amssymb,amsmath,amscd,url,listings}
\usepackage{amsrefs}
\usepackage{enumitem}
\newlist{enumth}{enumerate}{1}
\setlist[enumth]{label=\emph{(\arabic*)}, ref=(\arabic*)}

\usepackage{eucal}
\usepackage{fullpage}
\usepackage{scrtime}
\usepackage[colorlinks]{hyperref}

\usepackage{mathrsfs}
\usepackage[scr=rsfs,cal=boondox]{mathalfa}

\usepackage{datetime}
\usepackage[all]{xy}
\usepackage{graphicx}
\usepackage{tikz-cd}
\tikzset{commutative diagrams/arrow style=math font}

\usepackage{mathrsfs}
\usepackage{xspace}

\usepackage{tensor}
\usepackage{braket}

\usepackage{todonotes}
\usepackage{tensor}

\shortdate

\newcommand{\main}{\mathrm{main}}
\newcommand{\err}{\mathrm{err}}

\DeclareMathSymbol{A}{\mathalpha}{operators}{`A}%
\DeclareMathSymbol{B}{\mathalpha}{operators}{`B}%
\DeclareMathSymbol{C}{\mathalpha}{operators}{`C}%
\DeclareMathSymbol{D}{\mathalpha}{operators}{`D}%
\DeclareMathSymbol{E}{\mathalpha}{operators}{`E}%
\DeclareMathSymbol{F}{\mathalpha}{operators}{`F}%
\DeclareMathSymbol{G}{\mathalpha}{operators}{`G}%
\DeclareMathSymbol{H}{\mathalpha}{operators}{`H}%
\DeclareMathSymbol{I}{\mathalpha}{operators}{`I}%
\DeclareMathSymbol{J}{\mathalpha}{operators}{`J}%
\DeclareMathSymbol{K}{\mathalpha}{operators}{`K}%
\DeclareMathSymbol{L}{\mathalpha}{operators}{`L}%
\DeclareMathSymbol{M}{\mathalpha}{operators}{`M}%
\DeclareMathSymbol{N}{\mathalpha}{operators}{`N}%
\DeclareMathSymbol{O}{\mathalpha}{operators}{`O}%
\DeclareMathSymbol{P}{\mathalpha}{operators}{`P}%
\DeclareMathSymbol{Q}{\mathalpha}{operators}{`Q}%
\DeclareMathSymbol{R}{\mathalpha}{operators}{`R}%
\DeclareMathSymbol{S}{\mathalpha}{operators}{`S}%
\DeclareMathSymbol{T}{\mathalpha}{operators}{`T}%
\DeclareMathSymbol{U}{\mathalpha}{operators}{`U}%
\DeclareMathSymbol{V}{\mathalpha}{operators}{`V}%
\DeclareMathSymbol{W}{\mathalpha}{operators}{`W}%
\DeclareMathSymbol{X}{\mathalpha}{operators}{`X}%
\DeclareMathSymbol{Y}{\mathalpha}{operators}{`Y}%
\DeclareMathSymbol{Z}{\mathalpha}{operators}{`Z}%

\renewcommand{\mathcal}{\mathscr}

\renewcommand{\leq}{\leqslant}
\renewcommand{\geq}{\geqslant}
 
\numberwithin{equation}{section}

\newcommand{\uple}[1]{\text{\boldmath${#1}$}}

\def\stacksum#1#2{{\stackrel{{\scriptstyle #1}}
{{\scriptstyle #2}}}}

\def\setminus{\mathchoice
    {\mathbin{\vrule height .72ex width 1.61ex depth -.38ex}}
    {\mathbin{\vrule height .72ex width 1.61ex depth -.38ex}}
    {\mathbin{\vrule height .50ex width 0.85ex depth -.28ex}}
    {\mathbin{\vrule height .20ex width 0.570ex depth -.24ex}}
}

\newcommand{\bfalpha}{\uple{\alpha}}
\newcommand{\bfgamma}{\uple{\gamma}}

\newcommand{\bfrho}{\uple{\rho}}

\newcommand{\bftheta}{\uple{\theta}}
\newcommand{\bfbeta}{\uple{\beta}}
\newcommand{\bfeta}{\uple{\eta}}

\newcommand{\Cc}{\mathbf{C}}

\newcommand{\Aa}{\mathbf{A}}

\newcommand{\Zz}{\mathbf{Z}}
\newcommand{\Pp}{\mathbf{P}}
\newcommand{\Rr}{\mathbf{R}}
\newcommand{\Gg}{\mathbf{G}}

\newcommand{\Qq}{\mathbf{Q}}

\newcommand{\Fq}{\Ff_q}
\newcommand{\Fqt}{\Ff^\times_q}

\newcommand{\kt}{{k^\times}}

\newcommand{\Ff}{\mathbf{F}}

\newcommand{\mmu}{\boldsymbol{\mu}}

\newcommand{\mcN}{\mathscr{N}}
\newcommand{\mcH}{\mathscr{H}}

\def\loccit{loc.\kern3pt cit.{}\xspace}
\def\cf{see\kern.3em}
\def\Cf{See\kern.3em}
\def\eg{e.g.\kern.3em}

\def\resp{\text{resp.}\kern.3em}

\newcommand{\mods}[1]{\,(\mathrm{mod}\,{#1})}

\newcommand{\what}{\widehat}
\newcommand{\whFqt}{\widehat{\Fqt}}

\newcommand{\ra}{\rightarrow}

\DeclareMathOperator{\res}{Res}

\DeclareMathOperator{\Kl}{\mathrm{Kl}}

\DeclareMathOperator{\supp}{supp}

\DeclareMathOperator{\Tr}{Tr}

\newcommand{\eps}{\varepsilon}
\renewcommand{\rho}{\varrho}

\DeclareMathOperator{\SL}{\mathbf{SL}}

\DeclareMathOperator{\SO}{\mathbf{SO}}

\DeclareMathOperator{\Ort}{\mathbf{O}}

\newcommand{\demi}{{\textstyle{\frac{1}{2}}}}

\DeclareMathSymbol{\gena}{\mathord}{letters}{"3C}
\DeclareMathSymbol{\genb}{\mathord}{letters}{"3E}

\def\intc{\frac{1}{2i\pi}\mathop{\int}\limits}

\theoremstyle{plain}
\newtheorem{theorem}{Theorem}[section]
\newtheorem*{theorem*}{Theorem}
\newtheorem{lemma}[theorem]{Lemma}
\newtheorem{corollary}[theorem]{Corollary}

\newtheorem*{conjectureP}{Conjecture $\msP(a,b,c,d)$}

\newtheorem{proposition}[theorem]{Proposition}

\theoremstyle{remark}

\theoremstyle{definition}

\newtheorem{definition}[theorem]{Definition}

\newtheorem{remark}{Remark}[section]

\newcommand{\mcM}{\mathscr{M}}
\newcommand{\mcL}{\mathscr{L}}

\newcommand{\mcF}{\mathscr{F}}
\newcommand{\mcK}{\mathscr{K}}

\newcommand{\mcE}{\mathscr{E}}

\newcommand{\mfS}{\mathfrak{S}}

\newcommand{\msP}{\mathscr{P}}

\renewcommand{\geq}{\geqslant}
\renewcommand{\leq}{\leqslant}
\renewcommand{\Re}{\mathfrak{Re}\,}

\newcommand{\ov}[1]{\overline{#1}}

\newcommand\sumsum{\mathop{\sum\sum}\limits}
\newcommand\sumdsum{\mathop{\sum\cdots\sum}\limits}
\newcommand\sumsumsum{\mathop{\sum\sum\sum}\limits}

\newcommand{\Qlb}{{\ov{\Qq}_{\ell}}}

\newcommand{\Good}{Galant }
\newcommand{\good}{galant}
\newcommand{\goodp}{galant}

\newcommand{\Hyp}{\mathrm{Hyp}}
\newcommand{\HYP}{\mathscr{H}}

\begin{document}

\title{Toroidal families and averages of $L$-functions, II: cubic moments}
 
\author{\'Etienne Fouvry}
\address{Universit\'e Paris--Saclay,   CNRS \\
Laboratoire de Math\'ematiques d'Orsay\\
  91405 Orsay  \\France}
\email{etienne.fouvry@universite-paris-saclay.fr}

\author{Emmanuel Kowalski}
\address{ETH Z\"urich -- D-MATH\\
  R\"amistrasse 101\\
  CH-8092 Z\"urich\\
  Switzerland} \email{kowalski@math.ethz.ch}

\author{Philippe Michel}
\address{EPFL/SB/TAN, Station 8, CH-1015 Lausanne, Switzerland }
\email{philippe.michel@epfl.ch}

\author{Will Sawin}
\address{Princeton University, Department of Mathematics, Fine Hall, Washington Road, Princeton, NJ 08540, USA}
\email{wsawin@math.princeton.edu}


\begin{abstract}
  Generalizing our previous work on ``toroidal averages'', we study the
  average of special values of $L$-functions of the form
  $L(\demi,\chi^a)L(\demi,\chi^b)L(\demi,\chi^c)$ for integers~$a$, $b$
  and~$c$, where $\chi$ varies over Dirichlet characters of a given
  prime modulus. We highlight connections with estimates for bilinear
  forms of trace functions and with bounds for the number of solutions
  of monomial equations in three variables in small boxes over finite
  fields.
\end{abstract}
\date{\today,\ \thistime} 
\maketitle 
\setcounter{tocdepth}{1}
\begin{flushright}
  \textit{Dedicated to Roger Heath-Brown\\
    with the deepest admiration}
\end{flushright}
\tableofcontents

\section{Introduction}

In this paper we continue the study initiated in \cite{FKMAA} of
evaluating moments of central values of $L$-functions in toroidal
families. In that paper, the first three named authors obtained an
asymptotic formula for second mixed moments of Dirichlet $L$-functions
at the central point $1/2$,
\begin{equation}\label{doublemoment}M_{a,b}(q):=\frac{1}{q-1}\sum_{\chi\mods q}L(1/2,\chi^a)L(1/2,\chi^b).
\end{equation}
for $a,b\in\Zz-\{0\}$ and $q\ra\infty$ along the primes. In this
paper, we consider the cubic case and aim for any asymptotic formula
as $q\rightarrow +\infty$ along the primes of 
\begin{equation}
	\label{triplemoment}
	M_{a,b,c}(q):=\frac{1}{q-1}\sum_{\chi\mods q}L(1/2,\chi^a)L(1/2,\chi^b)L(1/2,\chi^c),
\end{equation}
for $a,b,c\in\Zz-\{0\}$ some non-zero fixed integers. 

Before stating our main results, we make some basic reductions:  first we observe that it is sufficient (up to permuting $(a,b,c)$ and applying complex conjugation) to evaluate $M_{a,b,c}(q)$ and $M_{a,b,-c}(q)$ for $a,b,c\geq 1$. Next, as we will see, the value of $M_{a,b,\pm c}(q)$ depends on $d=\gcd(a,b,c)$. Changing variables, our problem is equivalent to evaluating moments of the shape $M_{ad,bd,\pm cd}(q)$ for $a,b,c$ setwise coprime integers and $d\geq 1$ another fixed integer\footnote{We may further assume that $d|q-1$; see below.}. 

Finally, we need a definition:
\begin{definition}[\Good triples]\label{defgoodtriple} Let $a,b,c\geq 1$ be setwise coprime integers. 

The triple $(a,b,c)$  is called {\em \good} unless $a+b+c=4$. 

The triples $(a,b,c)$ such that $a+b+c=4$ are called {\em oxozonic}; such triples are permutations of $(1,1,2)$. 

The triple $(a,b,-c)$ is {\em \good} unless one of the following holds: 
\begin{itemize}
	\item $(a,b,-c)$ is of the form  $(c,b,-c)$ or $(a,c,-c)$; such a triple is called {\em induced}.
	\item $(a,b,-c)=(a,b,-(a+b))$ and $a+b=c\in\{2,3,4\}$; such a triple is called  {\em solvable}. Up to the substitution $a\leftrightarrow b$, solvable triples are equal to one of the following $(1,1,-2),\ (1,2,-3),\ (1,3,-4)$.
	\item $(a,b,-c)$ is neither induced nor solvable, $a+b+c\equiv 0\mods 2$ and $$\max(a+b,c)-(a,c)-(b,c)+1=4.$$
	Up to the substitution $a\leftrightarrow b$, such triples are equal to one of the following  $$(1,2,-5),(1,4,-3),(1,6,-3),\hbox{ or }(2,3,-1)$$
	
	The triples $(1,4,-3),(1,6,-3),(2,3,-1)$ are called {\em oxozonic} and the triple $(1,2,-5)$ is called {\em sulfatic}.
\end{itemize}	

\end{definition}

\begin{remark}
  The terminology is borrowed directly from our companion
  paper~\cite{FKMSBilinear} and is connected with the shape of the
  geometric monodromy group $G_{a,b,\pm c}$ of a certain $\ell$-adic
  sheaf $\mcK_{a,b,\pm c}$ associated with the triple $(a,b,\pm
  c)$. Likewise the {\em sulfatic} case is when $G_{a,b,\pm c}=\SO_4$
  and the {\em oxozonic} case is for $G_{a,b,\pm c}=\Ort_4$.
\end{remark}

We prove the following (see Theorem \ref{mainthm} for the general version):
\begin{theorem}\label{mainthmintro} Let $a,b,c\geq 1$ be setwise coprime integers, $d\geq 1$ another integer and $q\geq 3$ be a prime. 
	
	Assume that $(a,b,\pm c)$ is \good\ or {oxozonic}. 
	
	\begin{enumth}
		\item There exists $D_{a,b,\pm c}\geq 1$ and $\eta=\eta(a,b,c,d)>0$ such that
	\begin{equation}
	    M_{ad,bd,\pm cd}(q)\geq D_{a,b,\pm c}+O_{a,b,c,d}(q^{-\eta}).
     \label{tripleRelowerbound}
	\end{equation}
\item In the special case  $a=b$, \eqref{tripleRelowerbound} is an equality:
\begin{equation}\label{asymptotica=b}
	M_{ad,ad,\pm cd}(q)=D_{a,a,\pm c}+O_{a,c,d}(q^{-\eta}).
\end{equation}

\end{enumth}

Furthermore we  have
$D_{a,b,c}=1$
and
$D_{a,b,-c}=D_{a,b,-c}(1/2)$ is the value ($>1$) at $s=1/2$ of the converging Dirichlet series $D_{a,b,-c}(s)$ defined in \eqref{defDabc}.	
\end{theorem}
Note that the main term of the asymptotic expansion \eqref{asymptotica=b} is independent of the value of $d$.	The same 
observation applies to the formulas \eqref{asymptoticconjectural}, \eqref{friday} and \eqref{Mxilowerbound***} below. This phenomenon is similar in the context of 
the asymptotic formula of $M_{a,b} (q)$ defined by \eqref{doublemoment} (see the definition of the constant $\alpha (a,b)$ in \cite{FKMAA}*{p.110}).

We  expect that \eqref{tripleRelowerbound} is always an equality and
\eqref{asymptotica=b} shows that this is the case when $a=b$. However,
 the general case rests currently on a conjectural statement
 $\msP(a,b,c,d)$ (which we  will prove for $a=b$). Given $L,M,N\geq 1$, we define
\begin{equation*}
	N_{a,b,c,d}(L,M,N;q)=\\|\{(l,m,n)\sim L\times M\times N,\ (l^am^bn^c)^d\equiv 1\mods q,\ l^am^bn^c\not=1\}|
\end{equation*}

\begin{conjectureP}\label{conjpiercetriple} Let $a,b,c$ be non-zero integers with $a,b,c$ setwise coprime and $d\geq 1$ another integer.
 
  For any prime $q\geq 3$, there exists $\eta_0>0$ (depending on $a,b,c,d$ but not on $q$) such that for $L,M,N\geq 1$  and  $\eps>0$, one has 
  \begin{equation}
	\label{Nabcwishedboundintro}\frac{N_{a,b,c,d}(L,M,N;q)}{\sqrt{LMN}}\ll_{\eps,a,b,c,d} q^\eps\Bigl( \frac{\sqrt{LMN}}{q}+(LMN)^{-\eta_0}\Bigr).
	\end{equation}
 \end{conjectureP}
 We refer to \S \ref{secoverview} for a description of how this conjecture arises and to \S \ref{subsecerrorterm} for an additional discussion and proofs of special cases.

Assuming this conjecture we have
\begin{theorem}\label{mainthmconj} Let $a,b,c\geq 1$ be setwise coprime integers, $d\geq 1$ be another integer and $q\geq 3$ be a prime. Let 
\begin{equation}
	\label{dqdef}
	d_q:=\begin{cases}
	(d,q-1)&\hbox{ if $d$ is even,}\\
	2(d,q-1)&\hbox{ if $d$ is odd.}
\end{cases}
\end{equation}	
	If $(a,b,\pm c)$ is \good\ or {oxozonic} and if  Conjecture $\msP(a,b,\pm c,d_q)$ holds we have	
	\begin{equation}\label{asymptoticconjectural}
	M_{ad,bd,\pm cd}(q)=D_{a,b,\pm c}+O_{a,b,c,d}(q^{-\eta}).
\end{equation}
\end{theorem}

In any case, the inequality \eqref{tripleRelowerbound} is sufficient for the following general triple non-vanishing result (see \S \ref{16h50} for its proof):
\begin{corollary}\label{Ldonotvanish} Let $a,b, c\neq 0$  be three  non-zero integers.  Then there exists $Q>2$ such that, uniformly for $q\geq Q
$ prime,  we  have
    $$|\{\chi\mods q,\
    L(1/2,\chi^{a})L(1/2,\chi^b)L(1/2,\chi^c)\not=0\}|\gg_{a,b,c}
    q/(\log q)^{12}.$$
\end{corollary}
\begin{remark}\label{proportion}
	As we will see in the proof given in  \S \ref{16h50}, the
        non-vanishing of the product
        $L(1/2,\chi^{a})L(1/2,\chi^b)L(1/2,\chi^c)$ only concerns
        even characters.  Moreover, assuming a slight variant of
        Conjecture $\msP(a,b,\pm c,d)$, one could use mollifiers to
        exhibit a positive proportion of $\chi\mods q$ for which
        $L(1/2,\chi^{a})L(1/2,\chi^b)L(1/2,\chi^c)\not=0$.
        	Since this variant is a bit technical to state, we won't write it down here, but we would like to mention that this
        variant is known in the case $a=b$.
\end{remark}

In proposition \ref{propfouvry} we prove using a divisor switching argument that Conjecture $\msP(a,b,\pm c,d)$ holds on average over the set of primes $q\in\msP\cap[Q,2Q)$. Consequently we obtain an asymptotic formula for the cubic mixed moments on average over such moduli.
 \begin{theorem}\label{mainthmaverage}  Let $a,b, c\geq 1$  be three, setwise coprime, integers and let $d\geq 1$ be another integer.
	
	Assume that $(a,b,\pm c)$ is {\good}  or oxozonic. There exists $\eta=\eta(a,b,c)>0$ such that for $Q\geq 3$
	\begin{equation}\label{friday}\frac{1}{|\msP\cap[Q,2Q)|}\sum_{q\sim Q}M_{ad,bd,\pm cd}(q)=D_{a,b,\pm c}+O_{a,b,c,d}(Q^{-\eta}).\end{equation}
	
\end{theorem}

\begin{remark} Using Hoheisel's Theorem \cite{IK}*{Thm. 10.5} we could have replaced $[Q,2Q)$ in \eqref{friday} by the shorter interval $[Q,Q+Q^{1-\delta})$ for $\delta=\delta(a,b,c,d)>0$ sufficiently small (see the proof of Proposition \ref{propfouvry}) . 	
\end{remark}

\begin{remark} In the case $a=b$, the product $L(s,\chi^a)L(s,\chi^a)$
  can be interpreted as the $L$-function $L(s,\chi^a\times E)$ where
  $E$ is a suitable Eisenstein series. More generally given $f$ an
  holomorphic Hecke-cuspform\footnote{ So that $f$ satisfies the
    Ramanujan-Petersson conjecture.} (with trivial nebentypus) our
  methods allow to evaluate the following mixed moment
$$M_{ad,\pm cd}(f;q):=\frac{1}{q-1}\sum_{\chi\mods q}L(1/2,f\times \chi^{ad})L(1/2,\chi^{\pm {cd}}).$$
More precisely, if $(a,a,\pm c)$ is {\good}  or oxozonic, we have
$$M_{ad,\pm cd}(f;q) =D_{f,a,\pm c}+O_{f,a,c,d}(q^{-\eta})
$$ with 
 $D_{f,a,\pm c}>0$ and $\eta=\eta(a,c,d)>0$. This can be further refined to establishing the existence of a positive proportion of $\chi$'s for which $L(1/2,f\times \chi^{ad})L(1/2,\chi^{cd})\not=0$ (cf. Remark \ref{proportion}).
\end{remark}

\begin{remark} There remain a few moments which are not treated by Theorem \ref{mainthmintro} and Theorem \ref{mainthmaverage}. The following table summarizes the various kind of triples, provides examples, and explains their respective status.
\begin{table}[htbp]
\centering
\begin{tabular}{c|c|c|c}
\hline
Type & Typical examples & Monodromy & Treated in this paper? \\
\hline
galant infinite  & (1,1,1)                        & $\SL_3$  & yes          \\
galant finite  & (1,4,-5)                        & $\mfS_5$  & yes          \\
oxozonic & $(1,4,-3)$                        & $\Ort_4$     & yes          \\
sulfatic & $(1,2,-5)$                      & $\SO_4$    & no \\
induced  & $(a,c,-c)$                      & $  (\Fq(\mu_a),+)\rtimes \mu_a$& no \\
solvable & $(1,c-1,-c),\ c=2,3,4$ & $\mfS_c,\ c=2,3,4$  & no \\
\hline
\end{tabular}
\end{table}

After this paper was submitted the authors have been able to generalize Theorems \ref{thmTypeII} and \ref{thmtriplesum} to any $\ell$-adic sheaf, pure of weight $0$ which is geometrically irreducible and whose geometric monodromy group is not solvable (\cite{FKMSinf}); in particular this covers the case of the sulfatic triple $(1,2-5)$.

The case of {\em induced} moments
$$M_{ad,cd,-cd}(q)=\frac{1}{q-1}\sum_{\chi\mods q}L(1/2,\chi^{ad})|L(1/2,\chi^{cd})|^2,$$
and {\em solvable} moments
$$M_{ad,bd,-(a+b)d}(q)=\frac{1}{q-1}\sum_{\chi\mods q}L(1/2,\chi^{ad})L(1/2,\chi^{bd})L(1/2,\ov \chi^{(a+b)d})$$
for
$(a,b)=(1,1),\ (1,2),\ (1,3)$ are now covered in \cite{FKMSinduced} .
\end{remark}

\subsection{Related works}
The ``ordinary'' cubic moment
\begin{equation*}
	M_{1,1,1}(q)=\frac{1}{q-1}\sum_{\chi\mods q}L(1/2,\chi)^3
\end{equation*}
was evaluated asymptotically with a power saving error term by Petrow (unpublished); it turns out that this question is closely related to the problem of distribution of the ternary divisor function $d_3(n)$ in large arithmetic progressions modulo $q$ which was solved by Friedlander-Iwaniec in \cite{FI} and subsequently improved by Heath-Brown \cite{HB}, using bounds for higher dimensional algebraic exponential sums (furnished by Birch and Bombieri) which were consequences of  Deligne's resolution of the Weil's conjecture (see also \cites{FKM3,Sharma} for further progress). 

The moment $M_{1,1,-1}(q)$ was essentially evaluated by Das and Khan \cite{DK}: more precisely, they evaluated asymptotically the closely related mixed moment
$$M_{1,-1}(f;q)=\frac{1}{q-1}\sum_{\chi\mods q}L(1/2,f\times\chi)L(1/2,\chi^{-1})$$
where $f$ is a fixed Hecke eigenform. A bit surprisingly, this case is {\em simpler} than the ordinary cubic moment and our general approach will explain why this is the case (the triple $(1,1,-1)$ is {\em induced} according to Definition \ref{defgoodtriple}). A bit later, Zacharias \cite{zacharias} evaluated more general versions of this cubic moment along with the more difficult 
$$M_{1,1}(f;q)=\frac{1}{q-1}\sum_{\chi\mods q}L(1/2,f\times\chi)L(1/2,\chi).$$

Also closely related is the work of J. Pliego \cites{Pliego1,Pliego2} in the archimedean aspect, who  evaluated certain cubic mixed moments of the Riemann $\zeta$ function along the critical line: for $a,b,c\in\Rr-\{0\}$
$$\int_{0}^T\zeta(\frac{1}2+iat)\zeta(\frac{1}2+ibt)\zeta(\frac{1}2+ict)dt,\hbox{ as }T\ra\infty.$$

Finally, very recently, F. Berta and S. zur Verth \cite{BV}, using methods  related to this paper and to the companion paper \cite{FKMSBilinear} have established non-vanishing results for $L(1/2,\chi^a)L(1/2,\chi^b)$ for a positive proportion of $\chi$'s under some  constraint on the angle of the normalized Gauss sum of $\chi$ (see also  \cite{MAMS}*{Thm. 1.8} and \cite{Hough}).
\begin{theorem*} Given $a,b\in \Zz-\{0\}$, and $I\subset \mathbb S^1$ an interval of positive measure in the unit circle, there exists $P=P(a,b,I)>0$ such that as $q$ varies over large enough primes, one has
  $$
  \frac{1}{q-2}|\{\chi\mods q,\ \chi\not=1,\
  L(1/2,\chi^a)L(1/2,\chi^b)\not=0,\ \eps(\chi)\in I\}|\geq P.
  $$

  Here $\eps(\chi)$ denote the unitarily normalized Gauss sum.
\end{theorem*}
Closely related is the recent work of A. Earnst \cite{Earnst} who
established the non-vanishing of $L(\chi,1/2)$ with $\eps(\chi)\in I$
for a proportion $\gg \mathrm{length}(I)$ of characters $\chi$ as long
as $\mathrm{length}(I)$ is larger than a small but fixed positive
power of $q$. Earnst's result does not use \cite{FKMSBilinear} but
instead the previous work \cite{KMSAnn} which concerns bilinear sums with
hyper-Kloosterman sums $\Kl_k(\bullet)$, for which polynomial
dependency on~$k$ is known.

\subsection{Overview of the proof of Theorem \ref{mainthmintro}} \label{secoverview}

Given $a,b,c\geq 1$ setwise coprime, $d\geq 1$ an integer, let  $d'=\gcd(d,q-1)$. By Fourier analysis (see \S \ref{secreduction}) on the group $\what{\Fqt}$ we find that
$$M_{ad,bd,\pm cd}(q)=\sum_{\xi\in\mu_{d'}(\Fq)}M_{a,b,\pm c}(\xi;q)$$
where
$$\mu_{d'}(\Fq)=\{\xi\in\Fqt,\ \xi^{d'}=1\}$$
is the group of $d'$-th roots of unity and $M_{a,b,\pm c}(\xi;q)$ denotes  the twisted moment
$$M_{a,b,\pm c}(\xi;q)=\frac{1}{q-1}\sum_{\chi\mods q}\ov\chi(\xi)L(1/2,\chi^{a})L(1/2,\chi^{b})L(1/2,\chi^{\pm c}).$$

We are going to compute each term
$L(1/2,\chi^{a})L(1/2,\chi^{b})L(1/2,\chi^{\pm c})$ using the
approximate functional equation (abbreviated AFE below) for the
product $L$-functions
\[
  L(s,\chi^{a})L(s,\chi^{b})L(s,\chi^{\pm c}).
\]

The form of the AFE depends on the possible trivial entries of the
triple $(\chi^{a},\chi^{b},\chi^{c})$, and on their parity. It is
therefore natural to first decompose the sums $M_{a,b,c}(\xi;q)$ along
the {\em even} and {\em odd} characters and to write
$$
	M_{a,b,c}(\xi;q)=\frac 12 M^e_{a,b,c}(\xi;q)+\frac 12 M^o_{a,b,c}(\xi;q)
$$
where
\begin{equation*}
	M^e_{a,b,c}(\xi;q):=\frac{2}{q-1}\sum_\stacksum{\chi\mods q}{\chi(-1)=1}\ov\chi(\xi)L(1/2,\chi^a)L(1/2,\chi^b)L(1/2,\chi^c),
\end{equation*}
and the odd part $M^o_{a,b,c}(\xi;q)$ is defined similarly via the equation $\chi(-1)=-1$. 
 Observe that $M^e_{a,b,c}(\xi;q)$ and $M^o_{a,b,c}(\xi;q)$ are
 invariant under complex conjugation and hence are real numbers.

 Applying the AFEs, we are led to evaluating (up to a negligible error term) two types of sums (see \eqref{approxgeneric} below)
\begin{equation}
	\frac{1}{q-1}\sum_{\chi\mods q}\sum_{l,m,n\geq 1}\frac{\chi(l^am^bn^{\pm c}\xi^{-1})}{(lmn)^{1/2}}V(\frac{lmn}{X})
	\label{firstsum}
\end{equation}
and
\begin{equation}
\frac{1}{q-1}\sum_{\chi\mods q}\eps(\chi^a)\eps(\chi^b)\eps(\chi^{\pm c})\sum_{l,m,n}\frac{\ov\chi(l^am^bn^{\pm c}\xi)}{(lmn)^{1/2}}V(\frac{lmn}{Y}),
	\label{second-sum}
\end{equation}
 for $\xi$ ranging over $\pm \mu_{d'}(\Fq)=\mu_{d_q}(\Fq)$ where $d_q$ is defined in  \eqref{dqdef}; here $\eps(\chi^a),\eps(\chi^b),\eps(\chi^{\pm c})$
are the normalized Gauss sums, $V$ denotes a smooth function on $\Rr_{>0}$ (depending on the parities of $a,b,c$) which is  rapidly decreasing at $\infty$ and which converges to $1$ near $0$ and $X,Y\geq 1$ are parameters satisfying $XY\asymp q^3.$

\subsubsection*{Bounding solutions to congruence equations} Performing the $\chi$ summation in \eqref{firstsum} gives
\begin{equation}\label{partialsum1}
	\sum_\stacksum{l,m,n\geq 1}{l^am^bn^{\pm c}\equiv\xi\mods q}\frac{1}{(lmn)^{1/2}}V(\frac{lmn}{X}).
\end{equation}
When $\xi=1$ the contribution of the triples $(l,m,n)$ satisfying the {\em exact} equality
$l^a m^b n^{\pm c}=1$ yields a main term; we expect all the remaining terms to be error terms. Performing a dyadic partition of unity on the variables  $l,m,n$ we are led to count the number of solutions to  the congruence equation $l^am^b n^{\pm c}\equiv \xi\mods q$ for $l,m,n$ varying in dyadic boxes and satisfying $l^am^b n^{\pm c}\not=1$; this leads us to Conjecture $\msP(a,b,c,d_q)$ mentioned above.

Regarding Conjecture $\msP(a,b,c,d)$ we can offer the following observations. 

-- In the case $a=b=c=1$ (the usual cubic moment) we see that the
change of variable $k:=lmn$ reduces to the easy problem of counting, for $\xi\in \mu_{d}(\Fq)$, the number of solutions to the equation $$k\equiv \xi\mods q,\ 1\not= k\simeq LMN$$ weighted by the ternary divisor function $d_3(k)=k^{o(1)}$. We obtain
$$\frac{N_{1,1,1,d}(L,M,N;q)}{\sqrt{LMN}}\ll_{\eps} q^\eps\Bigl( \frac{\sqrt{LMN}}{q}+(LMN)^{-1/2}\Bigr).$$

-- If only $a=b$,  the same kind of trick applies: we write $k:=lm$ and are reduced to proving that
\begin{equation*}
	\frac{|\{(k,n)\sim LM\times N,\ k^an^{\pm c}\equiv\xi\mods q,\
          k^an^{\pm c}\not=1\}|}{\sqrt{LMN}}{\ll}q^\eps
        \Bigl( \frac{\sqrt{LMN}}{q}+(LMN)^{-\eta}\Bigr);
\end{equation*}
this essentially follows from the work of Pierce
\cite{Pierce}*{Theorem 4} and establishes Conjecture $\msP(a,a,c,d)$. Once we have bounded  the sums \eqref{second-sum}  adequately, the asymptotic formula \eqref{asymptotica=b} will follows.

-- Another simple observation is that Conjecture $\msP(a,a,c,d)$ holds on average over the primes $q$ in the interval $[Q,2Q)$;  this will eventually lead to the asymptotic formula \eqref{friday}.

-- Our final (trivial) observation is that $N_{a,b,c,d}(L,M,N;q)$ is a
sum of non-negative terms~(!); since we can arrange for
$V(\frac{lmn}{X})$ to also be non-negative, we see that
\eqref{partialsum1} is a sum of non-negative terms and therefore at
least as large as the main term; this will eventually lead to the
inequality \eqref{tripleRelowerbound}.

\subsubsection*{Bounding trilinear sums of exponential sums} We next turn to the second type of sums \eqref{second-sum}. Performing the $\chi$ summation and a dyadic partition reduce to obtaining the following bound
\begin{equation}\label{trilinearwish}
	\Sigma_{a,b,\pm c}(L,M,N):=\sum_{(l,m,n)\sim L\times M\times N}K_{a,b,\pm c}(\xi l^am^bn^{\pm c};q)\overset{?}{\ll}(LMN)^{1/2}q^{1/2-\eta}
\end{equation}
for some $\eta>0$ and
 $$LMN\ll Y=q^{1+\delta};$$
here we have set
\begin{equation}
	\label{Kabcdefintro}
	K_{a,b,\pm c}(u;q)=\frac{1}{q}\sum_{x^ay^bz^{\pm c}\equiv u\mods q}e_q(x+y+z),
\end{equation}
where $e_q$ is the usual non--trivial additive character of $\Fq$ (see \eqref{defeq}).  We call this sum the {\em monomial Kloosterman-like sum} of type $(a,b,\pm c)$. 

We will see in \S \ref{sec:hyperKloos} that $u\mapsto K_{a,b,\pm
  c}(u;q)$ is the trace function of an $\ell$-adic sheaf\footnote{
  More precisely, a {\em complex} of sheaves in the $-c$ case.}, denoted $\mcK_{a,b,\pm c}$, which is mixed of weights $\leq 0$ and has complexity bounded in terms of $a,b,c$ only; in particular, one has
\begin{equation}
	\|K_{a,b,\pm c}\|_\infty=O_{a,b,c}(1).
	\label{supnomrboundintro}
      \end{equation}
      

As in previous works, our strategies for obtaining \eqref{trilinearwish} will depend on the relative ranges of $L,M,N$. The simplest case is when $LMN\leq q^{1-\delta}$ for some $\delta>0$: the bound \eqref{supnomrboundintro} immediately gives \eqref{trilinearwish}; we are then essentially reduced to treating the range 
$LMN\asymp q$. 

To go further we will show in  Theorem \ref{th-galant-sheaf} that if $(a,b,\pm c)$ is {\good} in the sense of Definition \ref{defgoodtriple} then $\mcK_{a,b,\pm c}(\bullet)$ is {\good} in the sense of Definition \ref{defgoodsheafintro} below.

This allows us to use techniques from {\em applied $\ell$-adic cohomology} (see \cite{FKMS} for a general presentation). Specifically, if one of the three parameters $L,M,N$ exceeds somewhat the P\'olya-Vinogradov range,  for instance $$N\geq q^{1/2+\delta},$$
for some small $\delta>0$, we can use the P\'olya-Vinogradov (completion) method on the $n$-sum; if the product of two of three parameters $L,M,N$ exceeds somewhat the P\'olya-Vinogradov range while the remaining variable is not too small, for instance $$MN\geq q^{1/2+\delta}\hbox{ and }L\geq q^\delta$$
we use Theorem \ref{thmtriplesum} below; finally in the remaining case where two parameters are within the P\'olya-Vinogradov range while the remaining one is small,  for instance  
$$L\leq q^{\delta},\ q^{1/2-\delta}\leq M, N\leq q^{1/2+\delta}$$
we use Theorem \ref{thmTypeII}.

%

\subsection{Structure of the paper}
\begin{itemize}

\item In \S \ref{secbilinearsums} we review some basics about $\ell$-adic sheaves and their trace functions and list the various general bounds for bilinear and trilinear sums of trace functions from \cite{FKMSBilinear} that we will use in this paper.
	
\item In \S \ref{sec:hyperKloos} we discuss the exponential sums
  $K_{a,b,\pm c}(\bullet;q)$ which appear in the evaluation of
  cubic mixed moments and their associated sheaves $\mcK_{a,b,\pm
    c}$. We show using the Hasse-Davenport relations that
  these sheaves are closely connected (are ``forms'' of) certain
  Kloosterman and hypergeometric sheaves. In particular, they have the
  same geometric monodromy groups. Using this and Katz's work we
  conclude that if $(a,b,\pm c)$ is {\good} in the sense of Definition
  \ref{defgoodtriple}, then the bounds given in \S \ref{secbilinearsums} apply
  to $K_{a,b,\pm c}(\bullet;q)$.
	\item In \S \ref{seccubicmoment} we start the analysis of the mixed cubic moment $M_{a,b,\pm c}(q)$ and in \S \ref{secAFE} we reduce the problem to estimating two kinds of rapidly converging sums,  $M_{1,a,b,\pm c}(\xi;q)$ and $M_{2,a,b,\pm c}(\xi;q)$, when $a,b,c$ are setwise coprime.
	\item The first sums  are handled in \S \ref{secM1} and produce the main term. This involves counting solutions of congruence equations and is related to Conjecture $\msP(a,b,c,2d)$. 
	\item The  second sums, which involve trilinear sums in the
          exponential sums $K_{a,b,\pm c}(\bullet;q)$, are treated in \S \ref{secM2}. Under the assumption that $(a,b,\pm c)$ is  {\good}  we show using the results of \S \ref{secbilinearsums} and \S \ref{sec:hyperKloos} that the contribution of the sums $M_{2,a,b,\pm c}(\xi;q)$ is an error term.
	\item In the final \S \ref{secproofmainthm} we put all these estimates together and prove Theorem \ref{mainthmintro} and its averaged version Theorem \ref{mainthmaverage}. 
        \end{itemize}

\subsection*{Notation}

For $z\in\Cc$, we denote $e(z)=e^{2i\pi z}$.

We use the notation $f\ll g$ and $f=O(g)$, for $f$ and $g$ defined on a
set~$X$, synonymously: either notation means that there exists a
constant~$C\geq 0$ such that $|f(x)|\leq Cg(x)$ for all $x\in X$.

An exponent $o(1)$ in an estimate, such as $f(x)\ll x^{o(1)}g(x)$ for
all $x\geq 1$, means that, for any $\eps>0$, the estimate
$f(x)\ll x^{\eps}g(x)$ holds for $x\geq 1$. It is then implied that the
constant $C$ in the $\ll$ notation depends on $\eps$ (in addition to
other parameters which may be specified).

\subsection*{Acknowledgements}

This research was partially supported by the SNF grants $197045$,
$219220$, the SNF-ANR ``ETIENE'' grant $10003145$, the NSF grant
DMS-2502029, and a Sloan Research Fellowship.

We are grateful to Filippo Berta, Henryk Iwaniec, Peter Sarnak, Svenja zur Verth and
Ping Xi for encouragements and useful comments.

EF thanks EPFL and ETHZ for their hospitality and for providing an
environment conductive to a fruitful collaboration.

\section{Bounds for multilinear sums of trace functions}\label{secbilinearsums}
In this section, we recall some of the main results of
\cite{FKMSBilinear} which provide non trivial bounds for multilinear
sums of a certain class of trace functions (attached to suitable
$\ell$-adic sheaves). We refer to \cite{FKMSBilinear} for the precise
description of the terms used here. 

\begin{definition}\label{defgoodsheafintro} Let $q\not=\ell$ be primes, $k/\Fq$ be a finite extension and $\mcF$ an $\ell$-adic sheaf on $\Aa^1_k$ pure of weight $0$ with underlying vector space $V_\mcF$. We say that $\mcF$ is  {\em \good}  if  its geometric monodromy group $G_\mcF=G^{\mathrm{geom}}_\mcF$ has one of the following properties
		 \begin{enumerate}
				\item\label{goodinfinite} its identity component $N=G_\mcF^0$ is a simple algebraic group acting  irreducibly on $V_\mcF$,
				
				\item\label{goodfinite} or is finite and contains, as a normal subgroup, a finite perfect group $N$ which is a central extension of a nonabelian finite simple group $S$ by an abelian group $A$, $N$ acting irreducibly  on $V_\mcF$.
			\end{enumerate} 		
			In particular $\mcF$ is geometrically irreducible.
\end{definition}

Note that we have incorporated in the definition the assumption that
the sheaf is of weight~$0$, so the notion here corresponds more
precisely to that of a \emph{gallant and light} sheaf
in~\cite{FKMSBilinear}.

\begin{remark}\label{remgoodmorphism} Suppose that $\mcF$ is {\good};
  given $f:\Pp_k^1\ra \Pp^1_{k}$, a non-constant morphism of degree
  coprime with $q$, the pull-back sheaf $f^*\mcF$ remains {\good} if
  we are in Case \eqref{goodinfinite} above or, if we are in Case
  \eqref{goodfinite}, as long as $f$ induces a cyclic covering.
\end{remark}
 
 In fact it is useful to slightly extend the above definition
 \begin{definition}\label{defgoodcomplexsheafintro} Let $q\not=\ell$ be primes, $k/\Fq$ be a finite extension and $\mcK$ be complex of $\ell$-adic sheaves on $\Aa^1_k$ which is mixed of weights $\leq 0$. We say that $\mcK$ is  {\em \good}  if its weight-$0$ part $\mcK_0$ is {\good}  in the sense of Definition \ref{defgoodsheafintro}.
\end{definition}

\subsection{Bounds for bilinear sums of trace functions}
Given $b,c\in\Zz-\{0\}$, $K:\Zz/q\Zz\to \Cc$ a complex valued, $q$-periodic arithmetic function, $M,N\geq 1$ and $\bfalpha=(\alpha_m)_{m\sim M}$, $\bfbeta=(\beta_n)_{n\sim N}$ be sequences of complex numbers, we define the ``type-II'' sum
$$B_{b,c}(\bfalpha,\bfbeta)=\sumsum_{m\sim M,n\sim N}\alpha_m\beta_n K(m^bn^c).$$
The following theorem is \cite[Th. 1.3 \& Cor. 1.4]{FKMSBilinear}.

\begin{theorem}\label{thmTypeII} Let $b,c$ be non-zero integers. Let
  $q$ be a prime and $K(\bullet)$ be the trace function of an
  $\ell$-adic sheaf $\mcF$ modulo~$q$. We assume that $\mcF$ is {\good}  in the sense of Definition \ref{defgoodsheafintro}.

Let $k\geq 3$ be an integer and $M,N\geq 1$ such that
	\begin{equation}
	\label{assumptypeII}M\leq q,\ q^{3/(2k)}\leq N\leq q^{1/2+3/(4k)}.
	\end{equation}

 Then, for any sequences of complex numbers $\bfalpha=(\alpha_m)_{m\sim M},\ \bfbeta=(\beta_n)_{n\sim N}$, we have
\begin{equation}
	\label{BtypeII}B_{b,c}(\bfalpha,\bfbeta)\ll q^{o(1)}\|\bfalpha\|_2\|\bfbeta\|_2(MN)^{1/2}(\frac{1}{M}+(\frac{q^{\frac{3}4(1+1/k)}}{MN})^{1/k})^{1/2};	\end{equation}
		here the implicit constant depend at most on $b$, $c$, $k$ and on  the complexity  $c(\mcF)$ of $\mcF$.
\end{theorem}
\begin{remark} Taking $k$ large enough the bound \eqref{BtypeII} is non trivial (ie. $\ll\|\bfalpha\|_2\|\bfbeta\|_2(MN)^{1/2(1-\eta)}$ for some constant $\eta>0$) as long as $M\geq q^\delta  $ and $MN\geq q^{3/4+\delta}$ for some $\delta>0$. In particular when $M,N$ are within the {\em P\'olya-Vinogradov range} $$M\asymp N\asymp q^{1/2}$$
and in fact  $M\asymp N$ could be taken as small as $q^{\frac{3}{8}(1+\delta)}$ for some $\delta>0$. 
\end{remark}

We will also need to bound trilinear sums in $K$: given $\bfalpha= (\alpha_l)_{l\sim L}, \bfbeta=(\beta_m)_{m\sim M}, \bfgamma=(\gamma_n)_{n\sim N}$ sequences of complex numbers, we set
$$S(\bfalpha,\bfbeta,\bfgamma):=\sumsumsum_{l\sim L,m\sim M,n\sim N}\alpha_l\beta_m\gamma_n K(l^am^bn^c).$$

Next, we have \cite[Theorem 1.5]{FKMSBilinear}.

\begin{theorem}\label{thmtriplesum} Let $q$ be a prime, $a,b,c$ be non
  zero integers, and $K$ the trace function of an $\ell$-adic sheaf
  $\mcF$ modulo~$q$. We assume that $\mcF$ is {\good} in the sense of
  Definition \ref{defgoodsheafintro}.

Let $L,M,N\geq 1$ be integers satisfying
\begin{equation*}
	L,MN\leq 4q
\end{equation*}
and consider sequences of complex numbers
$$\bfalpha= (\alpha_l)_{l\sim L}, \bfbeta=(\beta_m)_{m\sim M}, \bfgamma=(\gamma_n)_{n\sim N}$$ satisfying
$$|\alpha_l|,|\beta_m|,|\gamma_n|\leq 1.$$ 
Then, for any $k\geq 2$ we have
\begin{equation}\label{trilinearbound}
S(\bfalpha,\bfbeta,\bfgamma)\ll q^{o(1)}
 LMN(\frac{q^{1/2}}{MN}+\frac{q}{L^kMN})^{1/(2k)}	
\end{equation}
 where the implicit constant depends on $c(\mcF)$, $a,b,c$ and $k$.
 \end{theorem}
\begin{remark}The bound \eqref{trilinearbound} is non trivial as long as $$L\geq q^{\delta},\ MN\geq q^{1/2+\delta}$$ and $k$ is chosen sufficiently large so that $L^kMN\geq q^{1+\delta}$.
\end{remark}

We will also deal with a situation of sheaves whose geometric
monodromy group is the orthogonal group $\Ort_4$ (such sheaves are
called {\em oxozonic}): since $\SO_4$ is not simple, such  sheaves are
not {\good}. The substitute is~\cite[Th.\,1.6]{FKMSBilinear}.

\begin{theorem}\label{thmO4} Let $q$ be a prime, $a,b,c$ be non zero integers,	and $K$ be the trace function on $\Aa^1_{\Fq}$  of an $\ell$-adic sheaf $\mcF$ whose geometric monodromy group satisfies $G_\mcF\simeq \Ort_4$ (in some faithful representation $\rho_\mcF$) and for which $G^0_\mcF=\SO_4$ acts irreducibly.

Upon making the relevant assumptions on the sizes of the parameters $L,M,N$ \begin{itemize}
\item the conclusion of Theorem \ref{thmTypeII}, ie.
\begin{equation}
	\label{BtypeIIort}
	B_{b,c}(\bfalpha,\bfbeta)\ll q^{o(1)}\|\bfalpha\|_2\|\bfbeta\|_2(MN)^{1/2}(\frac{1}{M}+(\frac{q^{\frac{3}4(1+1/k)}}{MN})^{1/k})^{1/2};	\end{equation}	
  holds if $c$ is odd,\item  while the conclusion of Theorem \ref{thmtriplesum} always holds:
\begin{equation}\label{tripleort}
	S(\bfalpha,\bfbeta,\bfgamma)\ll q^{o(1)}
 LMN(\frac{q^{1/2}}{MN}+\frac{q}{L^kMN})^{1/(2k)}.
\end{equation}	
\end{itemize}
\end{theorem}

\section{The hypergeometric sheaves for the cubic
  moment}\label{sec:hyperKloos}

\subsection{Statement of the result}

The main result of this section concerns the exponential sums which
arise in the analysis of the cubic toroidal averages. These are the
sums
\begin{equation}\label{defK}
  K_{a,b,
    c}(u;q,\psi)=\frac{1}{q}\sum_{\substack{x,y,z\in\Fqt\\x^ay^bz^c=u}}
  \psi(x+y+z),
\end{equation}
where
\begin{itemize}
\item $a$, $b$, $c$ are non-zero integers,
\item $q$ is a prime number,
\item $\psi$ is a non-trivial additive character modulo~$q$.
\end{itemize}

As in the main part of the paper, we assume throughout that $a$ and $b$
are positive integers, and we write the third parameter in the form
$\pm c$ for some positive integer~$c$ and some sign~$\pm\in\{-1,1\}$. We
also assume that $a$, $b$, $c$ are globally coprime.
If~$\psi(x)=e(x/q)$, then we sometimes omit~$\psi$ from the
notation. Similarly, we also sometimes omit~$q$ if the modulus is clear
from context.

\begin{theorem}\label{th-galant-sheaf}
  Let $a$, $b$, $c$ be coprime positive integers and let $\pm$ be
  either~$1$ or~$-1$. Let~$\ell$ be a prime number.

  For any prime~$q\not=\ell$, and any non-trivial additive
  character~$\psi$ of~$\Ff_q$, there exists a hypergeometric
  $\ell$-adic sheaf~$\mcK_{a,b,\pm c}(\psi)$ on $\Aa^1$ over~$\Ff_q$
  with the following properties:
  \begin{enumth}
  \item The sheaf~$\mcK_{a,b,\pm c}(\psi)$ is lisse and pure of
    weight~$0$ on~$\Gg_m$;
  \item The complexity of~$\mcK_{a,b,\pm c}(\psi)$, in the sense
    of~\cite{FKM1} or equivalently of~\cite{qst}, is bounded in terms
    of~$(a,b,c)$ only;
  \item The trace function of~$\mcK_{a,b,\pm c}(\psi)$ on~$\Fqt$ is
    equal to $-K_{a,b,c}(u;q,\psi)$ if the sign~$\pm$ is~$1$ and to
    \[
      -K_{a,b,-c}(u;q,\psi)+O(q^{-1/2})
    \]
    if the sign is~$-1$, where the implied constant depends only
    on~$(a,b,c)$.
  \item If $q$ is large enough, depending only on $a$, $b$, $c$, then
    the direct image to~$\Aa^1$ of the sheaf~$\mcK_{a,b,\pm c}(\psi)$ is
    {\good} (resp. oxozonic, sulfatic, induced) if and only if the triple
    $(a,b,\pm c)$ is {\good} (resp. oxozonic, sulfatic, induced) in the
    sense of \textup{Definition~\ref{defgoodtriple}}.
  \end{enumth}
\end{theorem}

\begin{remark}
  As the statement suggests, the situation is much simpler in the case
  of a sign $\pm=1$, and the reader may consider this case first.
\end{remark}

The proof of this theorem will involve the following steps: we first use
the formal properties of hypergeometric sheaves, as defined by Katz, to
construct~$\mcK_{a,b,\pm c}(\psi)$ and to obtain its properties (1), (2)
and~(3).  We then analyze the monodromy group of the sheaf (again,
relying on the work of Katz, and also on his recent work with Tiep \cite{katz-tiep}) to
deduce~(4).  When the sign is~$+1$, this follows immediately from our
discussion in~\cite[\S\,9]{FKMSBilinear}, but the case of the sign $-1$
requires a more delicate analysis.

\subsection{Construction of a complex}

Let~$q$ be a prime number. Let $u\colon \Gg_m \to \Aa^1$ be the open
immersion. We denote by~$*_!$ the $\ell$-adic (compactly supported)
multiplicative convolution operation on complexes of $\ell$-adic sheaves
on~$\Gg_m$ over the base field~$\Ff_q$ (see the basic surveys
in~\cite[\S\,5.1,\,5.2]{GKM} or in~\cite[\S\,8.1]{ESDE}).  Define
\begin{equation}\label{eq-conv-ta}
  \mathcal{T}_{a,b,c}(\psi)=\Bigl([x\mapsto x^{a}]_*u^*\mcL_{\psi}[1]
  \Bigr) *_!  \Bigl([x\mapsto x^{b}]_*u^*\mcL_{\psi}[1] \Bigr)*_!
  \Bigl([x\mapsto x^{c}]_*u^*\mcL_{\psi}[1] \Bigr)(2).
\end{equation}

This is a complex of $\ell$-adic sheaves on~$\Gg_m$ over the base
field~$\Ff_q$. By the Grothendieck--Lefschetz trace formula and the
formalism of convolution, the trace function of
$\mathcal{T}_{a,b,c}(\psi)$ is
\[
  u\mapsto -\frac{1}{q}\sum_{xyz=u} \sum_{x_1^{a}=x}\psi(x_1)
  \sum_{y_1^{b}=y}\psi(y_1) \sum_{z_1^c=z}\psi(z_1)=
  -\sum_{x^{a}y^bz^c=u}\psi(x+y+z)=-K_{a,b,c}(u;q,\psi).
\]

\subsection{Hypergeometric representation}

Let~$k$ be a finite extension of~$\Ff_q$ and~$\ell$ a
prime~$\not=q$. We recall that we assume fixed an isomorphism
of~$\Qlb$ and~$\Cc$, so that we can use equivalently $\ell$-adic or
complex-valued characters.

For any additive $\ell$-adic character~$\psi$ of~$\Ff_q$, we denote by
$\psi_{/k}$ the corresponding character
\[
  \psi_{/k}(x)=\psi(\Tr_{k/\Ff_q}(x))
\]
of~$k$. We then extend the definition of the exponential sums to~$k$
by
\[
  K_{a,b, c}(u;k)=\frac{1}{|k|}\sum_{\substack{x,y,z\in
      k^{\times}\\x^ay^bz^c=u}} \psi_{/k}(x+y+z).
\]

Let~$\psi$ be any additive character of~$k$. For~$a\in k$, we denote
by~$\psi_{a}$ the character defined by
\begin{equation*}
  \psi_{a}(x)=\psi(ax).
\end{equation*}

For any $\ell$-adic multiplicative character $\chi$ of~$k^{\times}$,
we denote
\[
  G(\psi,\chi)=\sum_{x\in \kt}\psi(x)\chi(x),\quad\quad
  \eps(\psi,\chi)=\frac{1}{|k|^{1/2}}\sum_{x\in \kt}\psi(x)\chi(x)
\]
the un-normalized and normalized Gauss sums.

Let~$a\geq 1$ be an integer. We denote
\[
  \bfrho[a]=\{\rho\in\what \kt,\ \rho^a=1\}
\]
the group of characters of~$k^{\times}$ of order dividing $a$.

We now assume that $a$, $b$, $c$ are coprime positive integers and
that $k/\Fq$ is such that $a,b,c$ all divide $|\kt|$.  We denote
\begin{gather*}
  \bfrho[a,b]=\bfrho[a]\sqcup\bfrho[b]
  \\
  \bfrho[a,b,c]=\bfrho[a]\sqcup\bfrho[b]\sqcup \bfrho[c],
\end{gather*}
viewed as \emph{multisets}, so that for instance
the trivial character occurs with multiplicity $2$ in~$\bfrho[a,b]$
and $3$ in~$\bfrho[a,b,c]$.

To these multisets, and to a non-trivial additive character~$\psi$, are
associated the hypergeometric complexes denoted
$\HYP_{a,b,c}(\psi)=\HYP(\bfrho[a,b,c],\emptyset,\psi)$ and
$\HYP_{a,b,-c}(\psi)=\HYP(\bfrho[a,b],\bfrho[c],\psi)$
in~\cite{ESDE}*{Chap.\,8}. If we enumerate arbitrarily
\[
  \bfrho[a,b,c]=\{\rho_1,\ldots,\rho_{a+b+c}\},\quad\quad
  \bfrho[a,b]=\{\rho_1,\ldots,\rho_{a+b}\},\quad\quad 
  \bfrho[c]=\{\eta_1,\ldots,\eta_c\},
\]
then the trace functions of these complexes are given, respectively, by
the exponential sums
\[
  \Hyp(u;\bfrho[a,b,c],\emptyset,\psi)=
  (-1)^{a+b+c}\sumdsum_{x_1\cdots x_{a+b+c}=u}
  \psi(x_1+\cdots+x_{a+b+c})\prod_{1\leq i\leq a+b+c}\rho_i(x_i)
\]
and
\begin{multline*}
  \Hyp(u;\bfrho[a,b],\bfrho[c],\psi)=
  \\
  (-1)^{a+b-c} \sumdsum_{x_1\cdots x_{a+b}=uy_{1}\cdots y_{c}}
  \psi((x_1+\cdots+x_{a+b})-(y_{1}+\cdots+y_{c})) \prod_{1\leq i\leq
    a+b}\rho_i(x_i)\prod_{1\leq j\leq c}\ov{\eta_j(y_j)}
\end{multline*}
(see~\cite{ESDE}*{8.2.7}).

We will denote these sums by $\Hyp_{a,b,c}(u;\psi)$ and
$\Hyp_{a,b,-c}(u;\psi)$ respectively.



\begin{proposition}\label{prop:hyperkloossums}
  Assume again that $a,b,c\ |\ |\kt|$. Let~$f_+$ and~$f_-$ be the
  elements of~$k^{\times}$ defined by
  \[
    f_+ =a^ab^bc^c,\quad\quad f_- =a^ab^b(- c)^{-c}.
  \]

  For any~$u\in k^{\times}$, the formula
  \[
    K_{a,b,\pm c}(f_\pm
    u;\psi)=(-1)^v\frac{\eps_{a,b,c}(\psi)}{q}\Hyp_{a,b,\pm c}(u;\psi)
  \]
  holds, where
  \begin{align*}
    \eps_{a,b,c}(\psi)&=-\eps_{a}(\psi)\eps_{b}(\psi)\eps_{c}(\psi)
    \\
    v&=a+b+|c|.
  \end{align*}
\end{proposition}

The proof will use the Hasse--Davenport relations in the following
form (see~\cite{GKM}*{\S\,5.6,p.\, 84}, taking into account the fact
that Katz writes the formula in terms of un-normalized Gauss sums).

\begin{proposition}\label{prop:HD}
  Let $a\geq 1$ be an integer. Let $k$ be an extension of $\Fq$ such
  that $a$ divides $|k^\times|$. Let $\psi$ be a non-trivial
  $\ell$-adic additive character of~$k$.  We have
  \begin{equation}\label{eqHD}
    -\eps(\psi_a,\chi^a)=\eps_a(\psi)\prod_{\rho\in\bfrho[a]}
    {\eps(\psi,\chi\rho)},
  \end{equation}
  where
  \[
    \eps_a(\psi)=|k|^{-1/2}\prod_{\rho\in\bfrho[a]}{
      \eps(\psi,\rho)^{-1}}.
  \]
\end{proposition}

\begin{proof}[Proof of Proposition~\ref{prop:hyperkloossums}]
  We claim first that for all~$u\in k^{\times}$, the formula
  \[
    -\sum_{x^a=u}\psi(ax)=\eps_a(\psi)\Kl(u;\bfrho[a];\psi)
  \]
  holds, where $\Kl(u;\bfrho[a];\psi)=\Hyp(u;\bfrho[a];\psi)$ (with
  obvious conventions).

  Indeed, we check that both sides of this equality have the same
  discrete Mellin transform as functions of~$u$, i.e., that for any
  character~$\chi$ of~$k^{\times}$, the equality
  \[
    \sum_{u\in \kt} \chi(u)\Bigl(
    -\sum_{x^a=u}\psi(ax)\Bigr)
    =\sum_{u\in\kt}\chi(u)\eps_a(\psi)\Kl(u;\bfrho[a];\psi)
  \]
  holds. The left-hand side is equal to
  $-|k|^{1/2}\eps(\psi_a,\chi^a)$ by exchanging the sums and using the
  definition of the Gauss sum. On the right-hand side, enumerating
  \[
    \bfrho[a]=\{\rho_1,\ldots,\rho_a\}
  \]
  arbitrarily, we obtain by definition
  \[
    \sum_{u\in\kt}\chi(u)\eps_a(\psi)\Kl(u;\bfrho[a];\psi)
    =\frac{\eps_a(\psi)}{|k|^{(a-1)/2}} \sum_{u\in\kt}\chi(u)
    \sumsum_{x_1\cdots
      x_a=u}\psi(x_1+\cdots+x_a)\prod_{i=1}^a\rho_i(x_i),
  \]
  which is equal to
  \[
    \eps_a(\psi)|k|^{1/2} \prod_{\rho^a=1}{\eps(\psi,\chi\rho)}=
    -|k|^{1/2}\eps(\psi_a,\chi^a)
  \]
  by the Hasse--Davenport relation~(\ref{eqHD}).

  It remains to observe, in the case $\pm=1$, that the function
  \[
    u\mapsto -K_{a,b,c}(a^ab^bc^cu;\psi)
  \]
  is the multiplicative convolution of the functions
  \[
    -\sum_{x^a=u}\psi(ax),\quad\quad
    -\sum_{x^b=u}\psi(bx),\quad\quad
    -\sum_{x^c=u}\psi(cx),
  \]
  up to normalization (i.e., we have
  \[
    -\sum_{\alpha\beta\gamma=u}
    \sum_{x^a=\alpha}\psi(ax)\sum_{x^b=\beta}\psi(bx)
    \sum_{x^c=\gamma}\psi(cx)=-\sum_{y_1^ay_2^by_3^c=uf_+}
    \psi(y_1+y_2+y_3)    
  \]
  for all~$u$), whereas, by the convolution construction of
  hypergeometric complexes and sums, the sums
  $\eps_{a,b,c}(\psi)\Hyp_{a,b,c}(u;\psi)$ are similarly obtained as the
  convolution of
  \[
    \eps_a(\psi)\Kl(u;\bfrho[a];\psi),\quad\quad
    \eps_b(\psi)\Kl(u;\bfrho[b];\psi),\quad\quad
    \eps_a(\psi)\Kl(u;\bfrho[c];\psi),
  \]
  up to sign and normalization. It is straightforward that the
  normalizations and signs agree to obtain the stated formula (to check
  this, note that the  hypergeometric sums are unnormalized, to keep
  the notation consistent with~\cite{ESDE}).

  The same argument applies for~$\pm=-1$.
\end{proof}

\begin{corollary}\label{cor:hyperkloossheaf}
  With notation as above, 
  the complexes $\mathcal{T}_{a,b,\pm c}(\psi)$ and
  $[\times f_\pm ^{-1}]^{*}\HYP_{a,b,\pm c}(\psi)$ are geometrically
  isomorphic.
\end{corollary}

\begin{proof}
  We consider the case where the sign is~$+1$.  Since base change
  to~$\overline{\Ff}_{q}$ commutes with convolution, the comparison of
  the convolution definition~(\ref{eq-conv-ta})
  of~$\mathcal{T}_{a,b,c}(\psi)$ and that of $\HYP_{a,b,c}(\psi)$,
  namely
  \[ \HYP_{a,b,c}(\psi)= \bigstar_{\rho\in\bfrho[a,b,c]}
    (u^*\mcL_{\psi}\otimes\mcL_{\rho}[1])
  \]
  (see~\cite{ESDE}*{\S\,8.2}) implies that it suffices to prove that
  for any integer $a\geq 1$, the complexes
  \[
    [x\mapsto x^{a}]_*(u^*\mcL_{\psi})[1]\quad\text{ and }\quad
    \Hyp_{a}(\psi,!;\bfrho[a];\emptyset)
  \]
  are geometrically isomorphic. This follows from the special case
  of~\cite[Prop.\,5.6.2]{GKM} with data
  \[
    (N,n,\chi_1,\ldots,\chi_n)=(a,1,1,\ldots, 1).
  \]

  The case of the sign $\pm=-1$ is similar.
\end{proof}

\begin{remark}
  It would be easy to ``upgrade'' this statement to an arithmetic
  isomorphism with suitable twists, but this is not needed for our
  purposes.
  \end{remark}

\subsection{The positive case}

This section concludes the proof of Theorem~\ref{th-galant-sheaf} in
the case of the positive sign, i.e., for the sums $K_{a,b,c}(u;q,\psi)$.

According to results of Katz
(see~\cite{ESDE}*{Cor.\,8.4.10.1,\,Th.\,8.4.11,\,Th.\,8.4.13}), the
hypergeometric complex $\HYP_{a,b,c}(\psi)$ is of the form
\[
  \mcF_{a,b,c}(\psi)[1]\Bigl(-\frac{a+b+c}{2}\Bigr)
\]
for a constructible sheaf~$\mcF_{a,b,c}(\psi)$ which is lisse
on~$\Gg_m$, pure of weight~$0$ and of rank $a+b+c$. Its complexity is
bounded in terms of $a$, $b$, $c$ only (either by~\cite{qst} because it
is a convolution of $a+b+c$ sheaves with absolutely bounded complexity,
or by inspection of its local monodromy invariants and the
definition~\cite{FKM1}*{Def.\,1.13}).

By the results of the previous section, it follows that the complex
$\mathcal{T}_{a,b,c}$ (which is geometrically isomorphic to the previous
one) is also of the form
\[
  \mcK_{a,b,c}(\psi)[1](2)
\]
for a constructible sheaf~$\mcK_{a,b,c}(\psi)$ which is lisse
on~$\Gg_m$, pure of weight~$0$ and of rank $r=a+b+c$, and is in fact
geometrically isomorphic to $\mcF_{a,b,c}(\psi)$. The trace function of
$\mcK_{a,b,c}(\psi)$ coincides (up to sign) with the exponential sums
$K_{a,b,c}(u;q,\psi)$. 

The sheaf $\mcK_{a,b,c}(\psi)$ is {\good} if and only if
$\mcF_{a,b,c}(\psi)$ is, since this is a geometric condition.  If the
rank~$r=a+b+c$ is not in~$\{4,8,9\}$, then we deduce
that~$\mcK_{a,b,c}(\psi)$ is {\good} directly
from~\cite{FKMSBilinear}*{Th.\,9.1,\S\,9.5} (which itself relies on the
work of Katz and Tiep), once we remark that the multiset $\bfrho[a,b,c]$
of characters cannot be Kummer-induced (see loc. cit., or below, for the
definition) when $a$, $b$, $c$ are coprime.

Furthermore, the cases where~$r\in \{4,8,9\}$ are discussed
in~\cite{FKMSBilinear}*{Th.\,9.2}; in the positive case, only~$r=4$ is
in fact possible, and then loc. cit. implies that~$\mcK_{a,b,c}(\psi)$
is oxozonic, because the determinant of the sheaf is the Kummer sheaf
associated to the character
\[
  \prod_{\chi\in\bfrho[a,b,c]}\chi,
\]
which is the character of order two raised to the power~$3$.

\subsection{The negative case}

The goal of this section is to prove Theorem~\ref{th-galant-sheaf} when
the sign $\pm$ is negative. The reason this case is more involved is
that the hypergeometric complex is then not always a single sheaf in a
suitable degree.

More precisely, let $\bfrho$ and $\bftheta$ be two multisets of
multiplicative characters of $\Ff_q$. Let
$\uple{\eta}=\bfrho\sqcap\bftheta$ denote the intersection of $\bfrho$
and $\bftheta$, as multisets (i.e., the multiplicity of a character is
the smallest of the multiplicities in the two
multisets). We denote
\[
  \mcH(\bfrho,\bftheta,\psi)=\mathrm{Hyp}(!,\psi;\bfrho,\bftheta)
\]
in the notation of~\cite{ESDE}*{Ch.\,8}. 

\begin{theorem}[Katz]
  Let $\bfrho$ and $\bftheta$ be two multisets of
multiplicative characters of $\Ff_q$ of respective size 
  $\bfrho$ and $\bftheta$ and let $\uple{\eta}$ be their intersection as multisets.
     \begin{enumth}
  \item If $\uple{\eta}$ is empty, i.e., if $\bfrho$ and $\bftheta$ are
    disjoint, then $\mcH(\bfrho,\bftheta,\psi)$ is of the form
    $\mcF(\bfrho,\bftheta)[1]$ for a constructible
    sheaf~$\mcF(\bfrho,\bftheta)$ which is lisse on~$\Gg_m$, pure of
    weight~$r+t-1$, of rank $r+t$, and with complexity bounded in terms
    of $r$, $t$ only.
  \item The semisimplification of the hypergeometric complex
    $\mcH(\bfrho,\bftheta,\psi)$ is isomorphic to a complex of the form
    \[
      \mcH(\bfrho\setminus\uple{\eta},\bftheta\setminus\uple{\eta};\psi)\oplus
      \mcE(\bfrho,\bftheta;\psi)
    \]
    where
    the complex $\mcE(\bfrho,\bftheta;\psi)$ has complexity bounded in
    terms of $r$ and $t$ and is mixed of weights~$\leq r+t-2$.
  \end{enumth}
\end{theorem}

Except for the statement concerning the complexity, this follows again
from Katz's
results~\cite{ESDE}*{Th.\,8.4.10,\,Th.\,8.4.11,\,Th.8.4.13}. The
complexity is bounded either using the definition by iterated
convolution and the formal properties of the complexity
(see~\cite{qst}*{Th.\,6.8}) or by the fact that the rank, number of
singularities, and Swan conductors are understood
(see~\cite{ESDE}*{Th.\,8.14.1} and~\cite{FKM1}*{Def.\,1.13}).

In our application, we have $\bfrho=\bfrho[a,b]$ and
$\bftheta=\bfrho[c]$. We denote by $\mcK_{a,b,-c}(\psi)$ the sheaf given
by part (1) of this proposition applied to the disjoint tuples
\[
  \bfrho[a,b]\setminus (\bfrho[a,b]\sqcap \bfrho[c]),\quad\quad
  \bfrho[c]\setminus (\bfrho[a,b]\sqcap \bfrho[c]).
\]

Combining the two parts of the proposition with the results of the
preceding section, we deduce that $\mathcal{T}_{a,b,c}(\psi)$ is of the
form
\[
  \mcK_{a,b,-c}(\psi)[1](2)\oplus \mathcal{E}
\]
where $\mcK_{a,b,-c}(\psi)$ is geometrically isomorphic to the
hypergeometric sheaf
\[
  \mcF_{a,b,-c}(\psi)=\mcF( \bfrho[a,b]\setminus (\bfrho[a,b]\sqcap
  \bfrho[c]), \bfrho[c]\setminus (\bfrho[a,b]\sqcap \bfrho[c]))
\]
and the trace function of $\mathcal{E}$ is $O(q^{-1/2})$.

The sheaf~$\mcK_{a,b,-c}(\psi)$ satisfies all conditions of
Theorem~\ref{th-galant-sheaf}, except the last one, again by the
previous discussion. We are therefore reduced to proving this final
part, i.e., to determine when $\mcK_{a,b,-c}(\psi)$ is {\good} or oxozonic
or sulfatic.  Equivalently, we need to do this for the hypergeometric
sheaf $\mcF_{a,b,-c}(\psi)$. To achieve this goal, we first note the
following lemma.

\begin{lemma}
  \begin{enumth}
  \item The sheaf $\mcF_{a,b,-c}(\psi)$ coincides with the
    hypergeometric sheaf
    \[
      \mcH(\bfrho_{a,b;c},\bftheta_{a,b;c};\psi)=
      \mathcal{H}_{1}(!,\psi;\bfrho_{a,b;c},\bftheta_{a,b;c})
    \]
    where
    \begin{gather*}
      \bfrho_{a,b;c}=\{1\}\sqcup\{\rho,\ \rho^a=1,\ \rho^{(a,c)}\not=1\}
      \sqcup \{\rho,\ \rho^b=1,\ \rho^{(b,c)}\not=1\},\\
      \bftheta_{a,b;c}=\{\rho,\ \rho^c=1,\ \rho^{(a,c)}\not=1,\
      \rho^{(b,c)}\not=1\}
    \end{gather*}
  \item It is of rank~$n$, where
    \begin{equation*}
      n=\max(r,t)=\max(a+b,c)-(a,c)-(b,c)+1.
    \end{equation*}
    with
    \begin{gather*} r=|\bfrho_{a,b;c}|=a+b-(a,c)-(b,c)+1,\\
      t=|\bftheta_{a,b,c}|=c-(a,c)-(b,c)+1,
    \end{gather*}
    \item We have $r\geq 1$, and $t=0$ if and only if $c$ divides $a$
    or~$b$.
  \end{enumth}
\end{lemma}

\begin{proof}
  All this follows from the previous discussion (e.g., we have $r\geq 1$
  because $a\geq (a,c)$ and $b\geq (b,c)$).
\end{proof}

The next lemma classifies the triples $(a,b,c)$ for which the
invariants $n$ and $r-t$ have specific values.

\begin{lemma}\label{lemtripleclassification}
Let $a$, $b$ and $c$ be three globally coprime  positive integers such that $1\leq a \leq b$.  
  Define again
  \begin{gather*}
    r=a+b-(a,c)-(b,c)+1,\quad t=c-(a,c)-(b,c)+1,
    \\
    n=\max(r,t)=\max(a+b,c)-(a,c)-(b,c)+1.
  \end{gather*}

  \begin{enumth}
  \item\label{case1} If $a+b=c$, then $(a,b)=(a,c)=(b,c)=1$ and
    $n=c-1=a+b-1$.

  \item\label{case2} We have $n\geq 1$ and the triple $(a,b,c)$ for
    which $n=1$ are either $(a,b,c)=(1,1,2)$ or of the form
    $(a,b,c)=(1,c,c)$ with $c\geq 1$.

  \item\label{casen=2} The triples for which $a+b-c\equiv 0\mods 2$
    and $n=2$ are  equal to
    \[
      (1,2,1)\text{ or } (1,2,3)
    \]
    or of the shape $(a,b,c)=(2,c,c)$ with $c> 2$ odd.
    
  \item\label{case3} The triples for which $|r-t|=|a+b-c|\not=0$,
    $a+b-c\equiv 0\mods 2$ and $n=4$ are either equal to one of
    \begin{equation}\label{sixtriples}
      (1,2,5),\quad (1,4,1),\quad (1,4,3),\quad (1,6,3),\quad (2,3,1),
      \quad (3,4,3)
    \end{equation}
    or of the shape $(a,b,c)=(4,c,c)$ with $c> 4$ odd.
    
    
  \item\label{case8} The triples $(a,b,c)$ for which $|r-t|=|a+b-c|=6$
    and $n=8$ are one of
    \[
      (1, 2, 9),\quad (1, 8, 3),\quad (2, 7, 3),\quad
      (4, 5, 3).
    \]
    
  \item\label{case9} The triples $(a,b,c)$ for which $|r-t|=|a+b-c|=6$
    and $n=9$ are one of
    \[
      (1, 3, 10),\quad (1, 9, 4),\quad (3, 7, 4),\quad (5, 5, 4).
    \]

  \item\label{case4} For a given integer $n_0\geq 1$, the set of
    triples $(a,b,c)$ such that $n=n_0$ is the union of a finite set
    of solutions, depending on $n_0$, and of the triples
    $(a,b,c)=(n_0,c,c)$, where $c>n_0$ is coprime to $n_0$.
  \end{enumth}
\end{lemma}

\begin{proof}
  We recall that that $r\geq 1$, so that $n\geq 1$ in all cases.

  The case where $a+b=c$ is straightforward given the fact
  that $(a,b,c)$ are globally coprime. This proves~\ref{case1}.
  
  We now prove the general statement~\ref{case4}, finding the triples
  $(a,b,c)$ such that $n=n_0$, up to finitely many exceptions.  For
  this, we first note that any solution to the equation~$n=n_0$
  satisfies
  \begin{equation}
    \label{upperboundb}
    c\leq \max(a+b,c)=n_0+(a,c)+(b,c)-1\leq n_0+a+b-1\leq n_0+2b-1,
  \end{equation}
  so that any set of solutions where the coefficient $b$ is bounded is
  necessarily finite and effectively computable.

  Let $(a,b,c)$ be such that $n_0=n$. We distinguish two cases: 

  \begin{enumerate}
  \item We assume that $b\mid c$ and write $c=bc'$. From
    \eqref{upperboundb} we have
    \[
      (c'-2)b\leq n_0-1.
    \]
   If $c'\geq 3$ we obtain that $b$ satisfies the inequality
   \begin{equation}\label{n0-1}
     b\leq n_0-1,
   \end{equation}
   so by the preliminary remark above, there are only finite many
   solutions.

   It remains to deal with the two cases $c=b$ or $c=2b$. We consider
   these separately.
   \begin{enumerate}
   \item If $c=b$, then $\max(a+b,c)=a+b$ and the equation $n=n_0$
     becomes
     \[
       n_0=a+b-(a,c)-b+1,
      \]
      and is therefore equivalent to $a-(a,c)=n_0-1$. But since $b=c$,
      we must have $(a,c)=1$ to ensure that $a$, $b$ and $c$ are
      globally coprime. Thus we have $a=n_0$ in this case and we
      obtain the solutions $(n_0,b,b)$ with $b\geq n_0$ coprime to
      $n_0$.
      
    \item If $c=2b$, then $\max(a+b,c)=c=2b$, so the equation $n=n_0$
      becomes
      \[
        n_0=2b-(a,2b)-b+1=b-(a,2b)+1.
      \]
      
      The global coprimality condition implies that $(a,2b)=(a,2)$, so
      the equation is $n_0=b-(a,2)+1$, and therefore 
      \begin{equation}\label{n0+1}
        b\leq n_0+1
      \end{equation}
      and
      there are only finitely many solutions.
          \end{enumerate}
    
  \item Now assume that $b\nmid c$. Then $(b,c)<b$, so $(b,c)\leq
    b/2$. Thus, if $n=n_0$, then we get
    \[
      n_0=\max(a+b,c)-(a,c)-(b,c)+1\geq a+b-a-\frac{b}{2}+1=\frac{b}{2}+1,
    \]
    or equivalently 
    \begin{equation}\label{2(n0-1)}
      b\leq 2(n_0-1),
    \end{equation}
    and there are only only finitely many solutions in this case. 

  \end{enumerate}

  We now handle the specific cases where $n_0\in\{1,2,4, 8,9\}$.

  For $n_0=1$, the inequalities \eqref{n0-1} and \eqref{2(n0-1)}
  cannot hold with $b\geq 1$. The argument above shows that we
  necessarily have $c=b$ or $c=2b$. So in addition to the family
  $(1,b,b)$ with $b\geq 1$, we obtain the unique extra solution
  $(a,b,c)=(1,1,2)$.
 
  For $n_0=2$, except for the family $(2,c,c)$, with $c\geq 3$ and
  odd, any triple $(a,b,c)$ satisfies $b\leq 3$  (see \eqref{n0-1}, \eqref{n0+1} and \eqref{2(n0-1)}) and $c\leq 7$  by \eqref{upperboundb}. A
  direct inspection gives the triples $(1,2,1)$ and~$(1,2,3)$.
  
  For $n_0=4$, the argument above shows that, except for the
  family $(4,c,c)$ with $c>4$ and odd, any triple $(a,b,c)$
  satisfies
  \begin{equation}\label{leqtriple}
    b\leq \max \left(n_0-1, \, n_0+1, \, 2(n_0-1)\right)
    =\max\left( n_0+1, \, 2(n_0-1)\right), 
  \end{equation}
  (see \eqref{n0-1}, \eqref{n0+1}, \eqref{2(n0-1)}). With $n_0 =4$, we
  obtain $b\leq 6$ which implies $c \leq 15$ by \eqref{upperboundb}.
  A direct computation gives the list of solutions
  in \eqref{sixtriples}.

  For $n_0\in \{8,9\}$, we impose the additional condition
  $|r-t|=6$. Since $n_0\not=6$, $(a,b,c)$ is not of the shape
  \[
    (n_0,c,c)\text{ or } (c,n_0,c)\text{ with } c\geq 1,\ (c,n_0)=1.
  \]
  A straightforward application of \eqref{leqtriple} and of
  \eqref{upperboundb} leads to the inequalities
  \[
    1\leq a\leq b\leq 2(n_0-1),\ c\leq 5n_0-3,
  \]
  and another direct computation leads to the lists above.
\end{proof}


We now recall the definition of Kummer-induced and Belyi-induced
multisets of characters.

\begin{definition}\label{def-induced}
  Let~$k$ be a finite field. Let $(\bfrho,\bftheta)$ be a pair of
  non-empty multisets of characters of $\kt$. Let~$r=|\bfrho|$ and
  $t=|\bftheta|$, and let~$n=\max(r,t)$.

  Let~$d\geq 2$ be an integer. The multiset $\bfrho$ is
  $d$-\emph{Kummer-induced} if $d$ divides $r$ and if for one (or any)
  character $\eta$ of order $d$, one has $\eta\cdot
  \bfrho=\bfrho$. The multiset is Kummer-induced if it is
  $d$-Kummer-induced for some~$d$.
  
  The pair $(\bfrho,\bftheta)$ is \emph{Kummer-induced} if there
  exists an integer $d\geq 2$ dividing $r$ and $t$ such that $\bfrho$
  and $\bftheta$ are both $d$-Kummer induced.
 
  The pair $(\bfrho,\bftheta)$ is \emph{Belyi-induced} if $r=t=n$ and
  there exists a partition $n=d+e$ into positive integers and
  characters $\alpha$ and $\beta$ of $\kt$, with $\beta\not=1$, such
  that
  \begin{gather*}
    \bfrho=\{\text{all $d$-th roots of $\alpha$}\}\cup \{\text{all
      $e$-th roots of $\beta$}\},
    \\
    \bftheta=\{\text{all $r$-th roots of $\alpha\beta$}\}.
  \end{gather*}

  The pair $(\bfrho,\bftheta)$ is \emph{primitive} if it is neither
  Kummer-induced nor Belyi-induced.
\end{definition}

The following lemma classifies the triples $(a,b,-c)$ for which the
corresponding $(\bfrho_{a,b;c},\bftheta_{a,b;c})$ are either
Kummer-induced or Belyi-induced.

\begin{lemma}
  \begin{enumth}
  \item The pair $(\bfrho_{a,b;c}, \bftheta_{a,b;c})$ is
    Kummer-induced if and only if $c=a$ or $c=b$.
  \item The pairs $(\bfrho_{a,b;c},\bftheta_{a,b;c})$ and
    $(\bftheta_{a,b;c},\bfrho_{a,b;c})$ are not Belyi-induced.
  \item In particular, the pair $(\bfrho_{a,b;c}, \bftheta_{a,b;c})$
    is primitive unless $c=a$ or $c=b$.
  \end{enumth}
\end{lemma}

\begin{proof}
  (1) If $c$ divides neither~$a$ nor~$b$, then the pair is not
  $d$-Kummer-induced for any~$d$, since $d\geq 2$ would have to divide
  $a$, $b$ and $c$, contradicting the assumption that these are
  coprime.

  If $c\mid b$, say, then we obtain $(a,b)=1$ and
  \[
    \bfrho_{a,b;c}=\bfrho[a]\sqcup(\bfrho[b]\setminus\bfrho[c]),\quad\quad
    \bftheta_{a,b:c}=\emptyset.
  \]
  
  The multiset $\bfrho[b]\setminus\bfrho[c]$ is non-empty, except when
  $b=c$, in which case $\bfrho_{a,b;c}=\bfrho[a]$ is indeed
  Kummer-induced. Otherwise, $\bfrho_{a,b;c}$ is not Kummer-induced
  since $a$ and $b$ are coprime and any inducing $d\geq 2$ would have
  to divide $a$ and $b$.

  A similar argument applies to the case $c\mid a$.

  (2) If $(\bfrho_{a,b;c},\bftheta_{a,b;c})$ or
  $(\bftheta_{a,b;c},\bfrho_{a,b;c})$ is Belyi-induced, then by
  definition, we must have $r=t=n$. This translates to $a+b=c$. The
  coprimality condition of $a$, $b$, $c$ implies that $a$, $b$ and~$c$
  are in fact pairwise coprime. Thus $n=a+b-1=c-1$, and furthermore
  \[
    \bfrho_{a,b;c}=\bfrho[a,b]\setminus\{1\},
    \quad
    \bftheta_{a,b;c}=\bfrho[c]\setminus\{1\}.
  \]

  Suppose $(\bfrho_{a,b;c},\bftheta_{a,b;c})$ is Belyi-induced, and
  let $(\alpha,\beta,d,e)$ be the parameters which, in the notation of
  Definition~\ref{def-induced}, witness the fact that the pair is
  Belyi-induced.  Since $\bfrho_{a,b;c}$ contains the trivial
  character, we must have $\alpha=1$. By definition, any
  $\theta\in \bfrho[c]\setminus\{1\}$ satisfies
  \[
    \beta=\alpha\beta=\theta^{n}=\theta^{c-1}=\theta^{-1},
  \]
  which implies that $c=2$ and $\beta$ is of order~$2$, hence $a=b=1$,
  and we obtain a contradiction since $(\{1\},\{\chi\})$ is not
  Belyi-induced for~$\chi$ of order~$2$.

  Finally, that the pair $(\bftheta_{a,b;c},\bfrho_{a,b;c})$ is not
  Belyi-induced, since the elements of
  $\bfrho_{a,b;c}=\bfrho[a,b]\setminus\{1\}$ would be $n$-th roots of
  some non-trivial character $\beta$, which is impossible as this
  multiset contains the trivial character.

  (3) Follows immediately from the previous cases.
\end{proof}

This already gives a first preliminary result towards
Theorem~\ref{th-galant-sheaf} in the negative case. We note that by
the definition $r=t$ is equivalent to~$a+b=c$.

\begin{proposition}[Case I]\label{pr-case1}
  Suppose that $a+b\not=c$ and $c\notin \{a,b\}$, then
  $\mcK_{a,b,-c}(\psi)$ is {\good} for all $q$ large enough, unless
  $n=r\in \{4,8,9\}$.
\end{proposition}

\begin{proof}
  This follows from the previous lemma
  and~\cite{FKMSBilinear}*{Th.\,9.1}.
\end{proof}

The next case to consider is when $a+b\not=c$ and $n=r\in\{4,8,9\}$.

\begin{proposition}[Case II]\label{pr-case2}
  Suppose that $a+b\not=c$, that $c\notin \{a,b\}$, and that
  $n=r\in \{4,8,9\}$. Then $\mcK_{a,b,-c}(\psi)$ is {\good} for all $q$
  large enough if $a+b+c$ is odd.
\end{proposition}

\begin{proof}
  From the previous lemma, the pair $(\bfrho_{a,b;c},\bftheta_{a,b;c})$
  is primitive. Noting that $a+b+c$ has the same parity as $r+t$, the
  result follows directly from~\cite{FKMSBilinear}*{Th.\,9.2}.
\end{proof}

Thus, there remain two cases: the case $a+b+c$ is even and $a+b\not=c$ and the case 
$a+b=c$. These are handled in the following two propositions.

\begin{proposition}[Case III]\label{pr-case3}
  Suppose that $a+b\not=c$, that $c\notin \{a,b\}$, that
  $n=r\in \{4,8,9\}$, and that $a+b+c$ is even.

  Then $\mcK_{a,b,-c}(\psi)$ is {\good} for all $q$ large enough unless
  $(a,b,c)$ is one of the triples
  \[
    (1,2,5),\quad (1,4,3),\quad (1,6,3),\quad (2,3,1),
  \]
  in which case $n=r=4$ and, for~$q$ large, the sheaf is sulfatic in
  the case $(1,2,5)$ and oxozonic otherwise.
\end{proposition}

\begin{proof}
  We apply again~\cite{FKMSBilinear}*{Th.\,9.2}.

  If $r=4$, we obtain the conclusion from loc. cit., and we only need
  check that the listed triples are those which give $r=4$ (more
  precisely, one first checks that in all cases, the connected
  component of the geometric monodromy group is~$\SO_4$; one then
  computes the determinant, taking into account that in loc. cit., it
  is assumed that $r>t$, and hence inversion is needed to deal with
  cases where~$r<t$; for instance, in the case $(1,2,5)$, we have
  $r=2$, $t=4$, and the determinant is the product of the conjugate of
  the characters from~$\bftheta_{a,b;c}$, which is trivial).

  It remains to show that the case with $r=8$ or $r=9$ for which the
  geometric monodromy group is not {\good} can not occur in our situation.

  \textbf{Case 1.} This is the case where $\Ff_q$ contains either the fourth
  roots or cube roots of unity, $r=8$, $|r-t|=6$, and there exists a character $\eta$
  and $\chi$ such that $\chi$ is not of exact order~$4$, and 
\begin{enumerate}
	\item either  \[
    \eta\cdot \bfrho_{a,b;c}=\{\chi,\bar{\chi}\}\cup \{\text{cube roots
        of }\chi,\bar{\chi}\},\quad \eta\cdot
      \bftheta_{a,b;c}=\{\chi_{1/4},\chi_{3/4}\}
  \]
  where $\chi_{1/4}$ and $\chi_{3/4}$ are the characters of exact
  order~$4$ (or the same with $\bfrho_{a,b;c}$ and $\bftheta_{a,b;c}$
  exchanged). 
  \item or   \[
    \eta\cdot \bfrho_{a,b;c}=\{\chi,\bar{\chi}\}\cup \{\text{cube roots
        of }\chi,\bar{\chi}\},\quad \eta\cdot
      \bftheta_{a,b;c}=\{\chi_{1/3},\chi_{2/3}\}
  \]
  where $\chi_{1/3},\ \chi_{2/3}$ are the characters of exact
  order~$3$ (or the same with $\bfrho_{a,b;c}$ and $\bftheta_{a,b;c}$
  exchanged). 
\end{enumerate}

  Up to exchanging $a$ and $b$, the conditions $r=8$ and $|r-t|=6$ imply
  the $(a,b,c)$ is one of
  \[
    (1,2,9),\quad (1,8,3),\quad (2,7,3),\quad (4,5,3).
  \]

  By inspection, the corresponding $\bfrho_{a,b;c}$ and
  $\bftheta_{a,b;c}$ are not of the indicated types: indeed let $\chi_{1/2}$ be the unique character of exact order $2$, we have $$\bfrho_{1,2;9}=\{1,\chi_{1/2}\}\not= \{\chi_{1/3},\chi_{2/3}\},\ \{\chi_{1/4},\chi_{3/4}\}$$ while for $(1,8,3),\ (2,7,3),\ (4,5,3)$, we have indeed
  $$\bftheta_{a,b;c}=\{\chi_{1/3},\chi_{2/3}\}$$
  but  $\bfrho_{a,b;c}$ equal 
  $$\bfrho[8],\ \{\chi_{1/2}\}\sqcup \bfrho[7]\hbox{ or } 
  \{\chi_{1/4},\chi_{1/2},\chi_{3/4}\}\sqcup \bfrho[5]$$ and does  not have the requested shape.

  \textbf{Case 2.} This is the case where $\Ff_q$ contains the third
  roots of unity, $r=9$, $|r-t|=3$, and there exists characters $\eta$,
  $\chi$, $\rho$, $\xi$, where the last are not of order dividing~$3$
  and satisfy $\chi\rho\xi=1$, such that
  \[
    \eta\cdot \bfrho_{a,b;c}=\{\chi,\rho,\xi\}\cup \{\text{square roots
        of }\bar{\chi},\bar{\rho},\bar{\xi}\},\quad\quad \eta\cdot
      \bftheta_{a,b;c}=\{1,\chi_{1/3},\chi_{2/3}\}
  \]
  where $\chi_{1/3},\chi_{2/3}$ are the characters of exact order~$3$, or the same with
  $\bfrho_{a,b;c}$ and $\bftheta_{a,b;c}$ exchanged.

  Up to exchanging $a$ and $b$, the conditions $r=9$ and $|r-t|=6$ imply
  the $(a,b,c)$ is one of
  \[
    (1,3,10),\quad (1,9,4),\quad (3,7,4),\quad (5,5,4),
  \]
  and again by inspection we check that none of these satisfy the
  conditions above.
\end{proof}

Finally, we address the case where $a+b=c$.

\begin{proposition}[Case IV]\label{pr-case4} Suppose that
   $a+b=c$. Then for $q$ large
  enough, the geometric monodromy group of $\mcK_{a,b;-c}(\psi)$ is
  the finite symmetric group $S_{c}$ acting on its
  $(c-1)$-dimensional irreducible representation. In particular, this
  sheaf is {\good} if~$c\geq 5$.
\end{proposition}

\begin{proof} Since $a$, $b$ and $c$ are postive integers, $c$ is different from $a$ and $b$. 
  As recalled in~\cite{FKMSBilinear}*{\S\,9.7}, there is a criterion due
  to Beukers and Heckman to check that a hypergeometric sheaf has finite
  monodromy, and a classification of the corresponding monodromy group,
  and we will check that this applies in our case.

  First, since $a+b=c$, the coprimality assumption implies that $a$, $b$
  and $c$ are pairwise coprime, so that
  \[
    n=r=t=a+b-1=c-1.
  \]

  In particular, only the trivial character occurs in the intersection
  of $\bfrho[a,b]$ and $\bfrho[c]$. Thus
  \[
    \bfrho_{a,b;c}=\bfrho[a]\sqcup (\bfrho[b]\setminus \{1\}),\quad\quad
    \bftheta_{a,b;c}=\bfrho[c]\setminus \{1\}.
  \]

  To these tuples, we associate the corresponding sets
  \[
    \mmu_{a}\cup (\mmu_b\setminus\{1\}),\quad\quad \mmu_c\setminus \{1\}
  \]
  of roots of unity, following the recipe
  in~\cite{ESDE}*{8.17}. According to the Beukers--Heckman
  criterion~\cite{BH}*{Prop. 5.9}, the hypergeometric differential
  equation with corresponding parameters has its differential Galois group
  isomorphic to $S_{c}$ acting by its $(c-1)$-dimensional
  representation. By the results of Katz~\cite{ESDE}*{\S\,8.17.3}, the
  geometric monodromy group of the corresponding hypergeometric sheaf
  is the same for $q$ large.

  For~$c\geq 5$, the alternating group $A_{c}\subset S_{c}$ is simple
  and still acts irreducibly in the $(c-1)$-dimensional representation,
  which shows that the sheaf $\mcK_{a,b;-c}(\psi)$ is \goodp.
\end{proof}

\begin{remark}
  In almost all cases, one can be more precise and in fact determine
  exactly the geometric monodromy group of the sheaf
  $\mcK_{a,b,c}(\psi)$.

  In the cases where the sheaf has finite monodromy but it not
  \goodp, or in the induced case, one can also explicitly compute the
  corresponding exponential sums. 

  For instance, consider the Kummer-induced case $(a,c,-c)$, where $a$
  and~$c$ are positive and coprime.  Setting $z=\lambda y$, we get
  \begin{align*}
    K_{a,c,-c}(u;\psi)
    &=\frac{1}{q}
      \sumsum_\stacksum{x,y,z\in\Fqt}{x(y/z)^c=u}\psi(x+y+z)\\
    &=\frac{1}{q}
      \sum_{x^a=u\lambda^c}\psi(x)\sum_{y\in\Fqt}\psi(y(1+\lambda))=
      \frac{1}{q}
      \sumsum_\stacksum{x,\lambda\in\Fqt}{x^a=u\lambda^c}
      \psi(x)(q\delta_{\lambda=-1}-1)\\
    &=\sum_{x^a=u(-1)^c}\psi(x)+O(q^{-1/2})
  \end{align*}
  by bounding Gauss sums.

  In the ``finite'' case $(a,b,-c)=(1,k-1,-k)$ for $k\geq 2$, we have
  \[
    K_{1,k-1,-k}(u;\psi)=\frac{1}{q}
    \sum_\stacksum{x,y,z\in\Fqt}{xy^{k-1}=uz^{k}}\psi(x+y+z).
  \]

  Since $xy^{k-1}=uz^{k}$ if and only if $(y/x)^{k-1}=u(z/x)^{k}$, a
  change of variable leads to
  \begin{multline*}
    K_{1,k-1,-k}(u;\psi)
    =\frac{1}{q}\sum_\stacksum{x,y,z\in
      \Fqt}{y^{k-1}=uz^{k}}\psi(x(1+y+z))
    \\
    =
    |\{(y,z)\in\Fqt\times \Fqt,\ y^{k-1}=uz^{k},\ 1+y+z=0\}|-
    \frac{1}{q}|\{(y,z)\in{\Fqt\times \Fqt},\ y^{k-1}=uz^{k}\}|.
  \end{multline*}

  Since
  \[
    |\{(y,z)\in\Fqt\times \Fqt,\ y^{k-1}=uz^{k}\}|=|\{(t,y)\in\Fqt\times \Fqt,\
    t^k/y=u\}|=q-1,
  \]
  we obtain
  \[
    K_{1,k-1,-k}(u;\psi)=\nu(P_{k,u};\Fq)-1+\frac{1}{q},
  \]
  where~$\nu(f;\Fq)$ is the number of roots of a polynomial~$f$
  in~$\Ff_q$, with
  \[
    P_{k,u}=uX^k+(-1)^{k}(X+1)^{k-1}.
  \]
\end{remark} 

 \section{Analysis of the cubic mixed moment}\label{seccubicmoment}
 Our purpose now  is to initiate the proof of Theorem   \ref{mainthm}, which contains Theorem \ref{mainthmintro}
  as a particular case ($\xi =1$). 

\subsection{Reduction to the coprime case}\label{secreduction} Let $a,b,c$ be non-zero setwise coprime integers (for the moment we make no assumption on the signs of $a,b,c$) and  $d\geq 1$ be another integer. 

We have
$$(\chi^{ad},\chi^{bd},\chi^{cd})=(\rho^{a},\rho^{b},\rho^{c}),\hbox{ for } \rho:=\chi^d.$$
Consider the $d$-th power map
\begin{equation}
	\label{dpower}
	\chi\in\whFqt\mapsto \chi^d\in\whFqt.
\end{equation}
Since $\whFqt$ is cyclic of order $q-1$ its image is 
$(\whFqt)^{d'}$ where $d'=(d,q-1)$ and its kernel is $$\whFqt[{d'}]=\{\chi\in\whFqt,\ \chi^{d'}=1\}.$$
The characteristic function of its image is given by
$$\rho\in\whFqt\mapsto 
\frac{1}{d'}\sum_{\xi^{d'}=1}\ov\rho(\xi);$$
moreover $d'$ is also the size of the non-empty fibers of \eqref{dpower}. 

By the definition \eqref{triplemoment} it follows that 
\begin{align}
	M_{ad,bd,cd}(q)&=\frac{1}{q-1}\sum_{\rho^{d'}=1}d'L(1/2,\rho^{a})L(1/2,\rho^{b})L(1/2,\rho^{c})\nonumber 
	\\&=\frac{1}{q-1}\sum_{\rho\mods q}(\sum_{\xi^{d'}=1}\ov\rho(\xi))L(1/2,\rho^{a})L(1/2,\rho^{b})L(1/2,\rho^{c})\label{d'decomposition}\\
	&=\sum_{\xi^{d'}=1}M_{a,b,c}(\xi;q)\nonumber
	\end{align}
where for $\xi\in \Fqt$, we have set
\begin{equation}
	\label{triplemomentxi}
	M_{a,b,c}(\xi;q):=\frac{1}{q-1}\sum_{\chi\mods q}\ov\chi(\xi)L(1/2,\chi^a)L(1/2,\chi^b)L(1/2,\chi^c).
\end{equation}
We  adjust notations and use $d$ in place of $d'$; we are then reduced to evaluating \eqref{triplemomentxi} for $a,b,c$ some fixed non-zero integers setwise coprime, $q$ a prime $q\equiv 1\mods d$ and some $\xi\in\mu_d(\Fq)$.
 
We can further decompose this moment along the even/odd characters: 
\begin{equation}\label{alreadydone}M_{a,b,c}(\xi;q)=\frac12M^{e}_{a,b,c}(\xi;q)+\frac12M^{o}_{a,b,c}(\xi;q),
\end{equation} 
with 
\begin{equation}\label{Mevendef}
	M^{e}_{a,b,c}(\xi;q)=\frac{2}{q-1}\sum_\stacksum{\chi\mods q}{\chi(-1)=1}\ov\chi(\xi)L(1/2,\chi^a)L(1/2,\chi^b)L(1/2,\chi^c)
\end{equation}
and
\begin{equation}
	M^{o}_{a,b,c}(\xi;q)=
\frac{2}{q-1}\sum_\stacksum{\chi\mods q}{\chi(-1)=-1}\ov\chi(\xi)L(1/2,\chi^a)L(1/2,\chi^b)L(1/2,\chi^c).\label{Modddef}
\end{equation}
Moreover if $a,b,c$ are of the same sign, up to conjugation we may and will  assume that $a,b,c\geq 1$; likewise, if $a,b,c$ have different signs up to conjugation and up to permutation we may and will  assume that $a,b\geq 1$ and $c\leq -1$. Consequently in the sequel we will restrict our evaluation to the cubic moments $M_{a,b,c}(\xi;q)$ and $M_{a,b,-c}(\xi;q)$ and their even/odd variants  for $a,b,c\geq 1$ setwise coprime integers and $d|(q-1)$.

Another basic observation, is that, by the identity
  $$\overline{\ov\chi(\xi)L(1/2,\chi^a)L(1/2,\chi^b)L(1/2,\chi^c)}=\chi(\xi)L(1/2,\ov\chi^a)L(1/2,\ov\chi^b)L(1/2,\ov\chi^c)$$ and the change of variable $\chi \rightarrow \overline \chi$, we see that
  \eqref{triplemomentxi},  \eqref{Mevendef} and \eqref{Modddef} are  invariant by complex conjugation and therefore the moments $M_{a,b,c} (\xi, q)$,  $M_{a,b,c}^{e} (\xi, q)$ and  $M_{a,b,c}^{o} (\xi, q)$ are all real numbers. 
  
  In \S \ref{secproofmainthm} we will prove the following 

\begin{theorem}\label{mainthm} Let $a,b,c,d\geq 1$ be integers with $a,b,c$ setwise coprime. Let $q$ be a prime and $\xi\in\Fqt$  such that $\xi^d\equiv 1 \bmod q$.

	Assume that $(a,b,\pm c)$ is {\good} or oxozonic. Then there exists $\eta=\eta(a,b,c,d)>0$ and $D_{a,b,\pm c}\geq 1$ (with $D_{a,b,+c}=1$) such that
	\begin{equation}\label{2147}
	    M^e_{a,b,\pm c}(\xi;q)\geq \delta_{\xi^2=1}D_{a,b,\pm c}+O_{a,b,c,d}(q^{-\eta}),
	\end{equation}
	\begin{equation}
	    M_{a,b,\pm c}(\xi;q)\geq \delta_{\xi=1}D_{a,b,\pm c}+O_{a,b,c,d}(q^{-\eta}),
     \label{Mxilowerbound}
	\end{equation}
	and
	\begin{equation}
	    M_{ad,bd,\pm cd}(q)\geq D_{a,b,\pm c}+O_{a,b,c,d}(q^{-\eta}).
     \label{Mxilowerbound*}
	\end{equation}
		If in addition $a=b$ we have 
		\begin{equation}
	\label{momenta=basymp}
	M^e_{a,a,\pm c}(\xi;q)= \delta_{\xi^2=1}D_{a,a,\pm c}+O_{a,c,d}(q^{-\eta}).
\end{equation}
as well as
\begin{equation}
	\label{momenta=basympodd}
	M^o_{a,a,\pm c}(\xi;q)= \delta_{\xi=1}D_{a,a,\pm c}-\delta_{\xi=-1}D_{a,a,\pm c}+O_{a,c,d}(q^{-\eta}),
\end{equation}
and
\begin{equation}
	    M_{a,a,\pm c}(\xi;q)= \delta_{\xi=1}D_{a,a,\pm c}+O_{a,c,d}(q^{-\eta}),
     \label{Mxilowerbound**}
	\end{equation}
	and
	\begin{equation}
	    M_{ad,ad,\pm cd}(q)= D_{a,a,\pm c}+O_{a,c,d}(q^{-\eta}).
     \label{Mxilowerbound***}
	\end{equation}

\end{theorem}

\section{Approximate functional equations for triples of Dirichlet $L$-functions}\label{secAFE}
In this section we use ``approximate functional equation''
(abbreviated AFE below) techniques to represent the cubic moments
$M^e_{a,b,c}(\xi;q)$, $M^o_{a,b,c}(\xi;q)$ as sums of several rapidly
converging series. There are many possibilities and we use one which
is due to Rubinstein~\cite{Rub} and discussed at length by Belabas and
Cohen in \cite[Chap. 8 \& 9 Thm 9.3.3]{BelCoh}.

Let
$$\gamma_+(s)=\pi^{-s/2}\Gamma(s/2),\ \gamma_-(s)=\gamma_+(s+1).$$
For $r\geq 1$ and $\bfeta=(\eta_1,\cdots,\eta_r)\in\{\pm 1\}^r$ we define
\begin{equation}\label{defgammaiota}
\gamma_\bfeta(s)=\prod_{i=1}^r\gamma_{\eta_i}(s),\ \iota_\bfeta:=\prod_{i=1}^ri^{\frac{\eta_i-1}{2}}.
\end{equation}
and let
\begin{equation*}
	V_\bfeta(y)=\frac{1}{2\pi i}\int_{(3)}\frac{\gamma_\bfeta(\frac{1}2+u)}{\gamma_\bfeta(\frac{1}2)}y^{-u}\frac{du}u
\end{equation*}

\begin{lemma}\label{LemmaV} For $\bfeta=(\eta_1,\cdots,\eta_r)\in\{\pm 1\}^r$, the function $y\mapsto V_\bfeta(y)$ is smooth on $\Rr_{>0}$  and takes real values. Furthermore it   satisfies 
\begin{enumth}
	
\item \label{Vlimzero} For any $y\in(0,1]$ and any $\eps>0$
$$V_\bfeta(y)=1+O_\eps(y^{1/2-\eps}).$$
\item \label{Vliminf} For any $A\geq 1$ any $y\geq 1$,
$$V_\bfeta(y)=O_A(y^{-A}).$$
\item \label{Vlimder}  For any $y>0$, for any $i\geq 1$
$$y^iV_\bfeta^{(i)}(y)\ll_A (1+y)^{-A}.$$
\item \label{Vpositivity} For any $y>0$,
$$V_\bfeta(y)>0.$$
\end{enumth}	
	\label{Vprop}
\end{lemma}
\proof Smoothness along with properties \eqref{Vlimzero}, \eqref{Vliminf}, \eqref{Vlimder} are obtained by standard contour shifts along the line $\Re u=-1/2+\eps$ together with with Stirling's formula (see \cite[p.129]{FKMAA}).
For \eqref{Vpositivity} we will use classical properties of the Mellin's transform $\mathcal M (f)$  of a function $f$ satisfying regularity properties  that we do not recall  (see
\cite[p.\,332]{BelCoh} or  \cite[Appendice A]{BaDBaLaSa} for more details)~:
$$
[\mathcal M (f)](s) := \int_0^\infty f(x) x^s \frac {dx}x.
$$ By inversion of integrations one checks that 
the Mellin transform satisfies the convolution property \begin{equation}\label{Mellinconvol}
\mathcal M( f_1) \cdot \mathcal M (f_2) = \mathcal M(f_1 \star f_2),
\end{equation}
where 
$$
(f_1\star f_2)(x)  = \int_0^\infty f_1\left( \frac xy\right)f_2 (y)\, \frac{dy}y,
$$
is the multiplicative convolution of $f_1$ and $f_2$.

The inverse Mellin transform $\mathcal M^{-1} (\varphi)$ is given by an integral in the complex plane, along some vertical line 
$$
[ \mathcal M^{-1}(\varphi) ](x) =\frac 1{2\pi i} \int_{c-i\infty}^{c+i\infty} x^{-s} \varphi (s) \, ds.
$$
Applying $\mathcal M^{-1}$ to both sides of \eqref{Mellinconvol}, we deduce the equality
$$\mathcal M^{-1}\left[ \mathcal M( f_1) \cdot \mathcal M (f_2) \right]= f_1 \star f_2.
$$
Defining $f_1=\mathcal M^{-1}(f)$ and $f_2=\mathcal M^{-1} (g),$ we obtain 
\begin{equation}\label{positivestar}
\mathcal M^{-1}(f\cdot g)  = \mathcal  M^{-1}(f)\star \mathcal M^{-1} (g),
\end{equation}
As for positivity we observe that $V_\bfeta(y)$ is the inverse Mellin transform of the product of the functions
\begin{equation}
	\label{gammalist}
	\frac{\gamma_{\eta_1}(1/2+u)}{\gamma_{\eta_1}(1/2)},\frac{\gamma_{\eta_2}(1/2+u)}{\gamma_{\eta_2}(1/2)},\cdots, \frac{\gamma_{\eta_r}(1/2+u)}{\gamma_{\eta_r}(1/2)},\ \frac{1}u
\end{equation} 
Now $\gamma_\eta(1/2)>0$ and $\gamma_\eta(1/2+u)$ is the Mellin transform 
of
$y\mapsto 2e^{-\pi y^2}y^{(2-\eta)/2}$ and  $1/u$ is the Mellin transform of $y\mapsto 1_{[0, 1]}(y)$(see \cite[\S 17.43]{GraRy} for instance).
It follows that $V_\bfeta(y)$ is the multiplicative convolution on $\Rr_{>0}$ of the inverse Mellin transforms of the functions in \eqref{gammalist}; as these inverse Mellin transforms are all non-negative and non zero, $V_\bfeta$ is positive after an iterated application of \eqref{positivestar}.\qed 

\subsection{The AFE in the generic case}\label{5.1} In what follows (i.e. in  \S\ref{5.1}, \S\ref{5.2}, \S\ref{5.3} and \S\ref{5.4}), $a$, $b$ and $c$ are non--zero setwise coprime integers, of any sign.  Let $\chi$ be a character such that $\chi^a,\chi^b,\chi^c\not=\chi_0$ (we call such a character {\em generic}) and $$\bfeta=\bfeta_{a,b,c}(\chi):=(\chi^a(-1),\chi^b(-1),\chi^c(-1))\in\{\pm 1\}^3.$$
Recall (see \cite{IK})
 $$\Lambda(s,\chi^a,\chi^b,\chi^c):= \gamma_{\bfeta}(s)L(s,\chi^a)L(s,\chi^b)L(s,\chi^c)$$
satisfies the functional equation
$$\Lambda(s,\chi^a,\chi^b,\chi^c)=\iota_\bfeta\eps(\chi^a)\eps(\chi^b)\eps(\chi^c)\Lambda(1-s,\ov\chi^a,\ov\chi^b,\ov \chi^c)$$
Here 
$$\eps(\chi)=\frac{1}{\sqrt q}\sum_{x\in\Fqt}\chi(x)e_q(x)$$ denotes the normalized Gauss sum relative to
the usual non-trivial additive character of $\Fq$
\begin{equation}\label{defeq}
e_q(\bullet)=\exp(\frac{2\pi i\bullet}q).
\end{equation}
Let $X,Y\geq 1$ such that
\begin{equation}
	\label{XYrelation}
	XY=q^{3},
\end{equation}
by contour integration
we have the equality 
\begin{multline}\label{approxgeneric}
	L(1/2,\chi^a)L(1/2,\chi^b)L(1/2,\chi^c)=\\
	\sum_\stacksum{l,m,n\geq 1}{(lmn,q)=1}\frac{\chi(l^am^bn^c)}{(lmn)^{1/2}}V_\bfeta(\frac{lmn}X)+\iota_\bfeta\eps(\chi^a)\eps(\chi^b)\eps(\chi^c)
\sum_\stacksum{l,m,n\geq 1}{(lmn,q)=1}\frac{\ov \chi(l^am^bn^c)}{(lmn)^{1/2}}V_\bfeta(\frac{lmn}Y).
\end{multline}

\begin{remark}\label{remfcteqn} Notice that given three setwise coprime integers $a,b,c\in\Zz-\{0\}$, the triple $\bfeta=\bfeta_{a,b,c} (\chi)$, the function $V_\bfeta(\bullet)$ and the sign $\iota_\bfeta$ depend only on the parity of $\chi$; moreover for $\chi$ even we have $\bfeta=(1,1,1)$ and $\iota_\bfeta=1$. In the sequel we will write $$V_\bfeta(y)=:V_{e}(y)\hbox{ or }V_{o}(y)$$ depending on whether  $\chi$ is even or odd and likewise we write $\iota_\bfeta=\iota_e=1$ or $\iota_o\in i^\Zz$.  In particular we have
\begin{equation*}
\iota_e =1 \text{ and } \iota_o =
\begin{cases}
-i  & \text{ if }  \nu (a,b,c) =1,\\
-1 & \text{ if }  \nu (a,b,c) =2,\\
i & \text{ if }  \nu (a,b,c) =3,
\end{cases}
\end{equation*} 
where $\nu (a,b,c)  $ is the number of odd elements in the list of coprime integers $a,\, b, \, c$ (see \eqref{defgammaiota}).
\end{remark}

\subsection{The non-generic case: when one of $\chi^a,\chi^b,\chi^c$ equals $\chi_0$}\label{5.2}
In that case the equality \eqref{approxgeneric} does not hold. However this situation occurs for at most $\vert a\vert +\vert b\vert +\vert c\vert $  characters $\chi$'s modulo $q$ and we will  evaluate  both sides of \eqref{approxgeneric} separately.

Regarding the left hand side, the central $L$-values involving the trivial character are absolutely bounded (there is at least one of them) and for the ones associated with a non-trivial character (there are at most two), we use the convexity bound  (see \cite[Theorem 5.23]{IK}) and obtain 
\begin{equation}
	L(1/2,\chi^a)L(1/2,\chi^b)L(1/2,\chi^c)\ll q^{1/2}.\label{convexnongeneric}
\end{equation}

Regarding the right hand side, we use trivial estimates and find that for any such $\chi$
\begin{multline}
	\sum_{l,m,n\geq 1}\frac{\chi(l^am^bn^c)}{(lmn)^{1/2}}V_\bfeta(\frac{lmn}X)\\+\iota_\bfeta\eps(\chi^a)\eps(\chi^b)\eps(\chi^c)
\sum_{l,m,n\geq 1}\frac{\ov \chi(l^am^bn^c)}{(lmn)^{1/2}}V_\bfeta(\frac{lmn}Y)\ll q^{o(1)}( X^{1/2}+(Y/q)^{1/2}).\label{sumnongeneric}
\end{multline}
Notice that in order to bound the second sum, we have also used that for $\chi=\chi_0$ the trivial character the normalized Gauss sum is small: $$\eps(\chi_0)=-1/\sqrt q.$$
\begin{remark}\label{remtrivialchar} Observe that when $\chi=\chi_0$ is the trivial character, \eqref{sumnongeneric} can be evaluated more precisely: it is equal to
	$$
		\sum_{(lmn,q)= 1}\frac{1}{(lmn)^{1/2}}V_\bfeta(\frac{lmn}X)+O(q^{o(1)}(\frac{Y}{q^3})^{1/2}).
	$$
	Using the definition of $ V_\bfeta$ and shifting contours, we see that the first term above equals
	\begin{align}\nonumber
		\sum_{(k,q)= 1}\frac{d_3(k)}{k^{1/2}}V_\bfeta(\frac{k}X)&=\intc_{(3)}(\zeta^{(q)})^3(s)\frac{\gamma_\bfeta(s)}{\gamma_\bfeta(\frac{1}2)}X^{s-1/2}\frac{ds}{s-1/2}\\
		&=P_{\bfeta}(\log X)X^{1/2}+O((\log q)^2\frac{X^{1/2}}{q}+X^{1/4})\label{chi0sum2}
	\end{align}
where $P_{\bfeta}(T)$ denotes  a polynomial of degree $2$ with leading coefficient $2\frac{\gamma_\bfeta(1)}{\gamma_\bfeta(\frac{1}2)}>0$. In particular the upper bound \eqref{sumnongeneric} is sharp.
\end{remark}
\subsection{Summing the approximate functional equations}\label{AFEaverage}\label{5.3}
As pointed out in Remark \ref{remfcteqn} above, the general shape of the approximate functional equation \eqref{approxgeneric} depends only on the parity of $\chi$ (and for $\chi$ odd, on the parities of the entries of the given triple $(a,b,c)$). Recall \eqref{alreadydone}, where  we split the sum $M_{a,b,c}(\xi;q)$ into  its subsums over even and odd characters:
\begin{equation}
	\label{evenodddecomp}
	M_{a,b,c}(\xi;q)=\frac12M^{e}_{a,b,c}(\xi;q)+\frac12M^{o}_{a,b,c}(\xi;q)
\end{equation} 
where these two terms are defined in \eqref{Mevendef} and \eqref{Modddef}. 

Let us evaluate $M^{e}_{a,b,c}(\xi;q)$ (the treatment of $M^{o}_{a,b,c}(\xi;q)$ is identical). We have the decomposition
$$M^{e}_{a,b,c}(\xi;q)=M^{e,gen}_{a,b,c}(\xi;q)+M^{e,n-gen}_{a,b,c}(\xi;q)$$
where $M^{e,gen}_{a,b,c}(\xi;q)$ is the contribution of the generic characters (the characters $\chi$ such that $\chi^a,\chi^b,\chi^c\not=\chi_0$) and $M^{e,n-gen}_{a,b,c}(\xi;q)$ is the contribution of the non-generic characters. 

By \eqref{convexnongeneric} we have
$$M^{e,n-gen}_{a,b,c}(\xi;q)\ll_{a,b,c}q^{-1/2}.$$

As for the generic portion $M^{e,gen}_{a,b,c}(\xi;q)$, we use \eqref{approxgeneric}; we then complete the sum over the generic $\chi's$ by adding and substracting the right hand side of \eqref{approxgeneric} for the non-generic characters and obtain
\begin{equation}\label{Megenabc}
	M^{e,gen}_{a,b,c}(\xi;q)=\frac{2}{q-1}\sum_\stacksum{\chi\mods q}{\chi(-1)=1}S^e_{a,b,c}(\xi;\chi)\\
	-\frac{2}{q-1}\sum_\stacksum{\chi_0\in\{\chi^a,\chi^b,\chi^c\}}{\chi(-1)=1}S^e_{a,b,c}(\xi;\chi)
\end{equation}
where
\begin{multline*}
	S^e_{a,b,c}(\xi;\chi)=\sum_{l,m,n\geq 1}\frac{1}{\sqrt{lmn}}\chi(\ov\xi l^am^bn^c)V_e(\frac{lmn}{X})
	\\
	+\iota_e\sum_{l,m,n\geq 1}\frac{1}{\sqrt{lmn}}\eps(\chi^a)\eps(\chi^b)\eps(\chi^c)\ov \chi(\xi l^am^bn^c)V_e(\frac{lmn}{Y}).
\end{multline*}
By \eqref{sumnongeneric} and the formula 
\begin{equation*}
\delta_{\chi(-1)=\pm1}=\frac{1\pm  \chi(-1)}{2}
\end{equation*}
we obtain the equalities
\begin{multline}
	\label{Mevenabcbasicdecomp}
	M^{e}_{a,b,c}(\xi;q)=M^e_{1,a,b,c}(\xi;q)+ M^e_{1,a,b,c}(-\xi;q)+M^e_{2,a,b,c}(\xi;q)+ M^e_{2,a,b,c}(-\xi;q)\\+O(q^{-1/2+o(1)}
	+( {X }/{q^2})^{1/2}q^{o(1)}+( {Y }/q^{3})^{1/2}q^{o(1)})
\end{multline}
\begin{multline}
	\label{Moddabcbasicdecomp}
	M^{o}_{a,b,c}(\xi;q)=M^o_{1,a,b,c}(\xi;q)- M^o_{1,a,b,c}(-\xi;q)+M^o_{2,a,b,c}(\xi;q)- M^o_{2,a,b,c}(-\xi;q)\\+O(q^{-1/2+o(1)}
	+( {X }/{q^2})^{1/2}q^{o(1)}+( {Y }/q^{3})^{1/2}q^{o(1)})
\end{multline}
where for $\bullet\in\{e,o\}$
\begin{eqnarray}\nonumber
	M^\bullet_{1,a,b,c}(\xi;q)&=&\frac{1}{q-1}\sum_{\chi\mods q}\sum_{l,m,n\geq 1}\frac{1}{\sqrt{lmn}}\chi(\ov\xi l^am^bn^c)V_\bullet(\frac{lmn}{X})\\
	&=&\sum_\stacksum{(lmn,q)=1}{l^am^bn^c\equiv \xi\mods q}\frac{1}{\sqrt{lmn}} V_\bullet(\frac{lmn}{X}),\label{M1abcdef}
\end{eqnarray}
(where  the second equality follows from the orthogonality relations) and 
\begin{eqnarray}\nonumber
	M^\bullet_{2,a,b,c}(\xi;q)&=&\frac{\iota_\bullet}{q-1}\sum_{\chi\mods q}\sum_{l,m,n\geq 1}\frac{1}{\sqrt{lmn}}\eps(\chi^a)\eps(\chi^b)\eps(\chi^c)\ov \chi(\xi l^am^bn^c)V_\bullet(\frac{lmn}{Y})\\
	&=&\frac{\iota_\bullet}{q^{1/2}}\sum_\stacksum{l,m,n\geq 1}{(lmn,q)=1}\frac{1}{\sqrt{lmn}}K_{a,b,c}(\xi l^am^bn^c)V_\bullet(\frac{lmn}{Y})
\label{M2def}\end{eqnarray}
(where the second equality follows from opening the Gauss sums $\eps(\chi^a),\eps(\chi^b),\eps(\chi^c)$ and switching summations).

 
  \begin{remark}
  	For notational simplicity we have not displayed the  dependency in parameters $X$ and $Y$ in the sums $M_{1,a,b,c}^\bullet(\xi;q)$ and $M_{2,a,b,c}^\bullet(\xi;q)$. Similarly, $  K_{a,b,c}(u)$ stands for the sum $K_{a,b,c}(u; q,e_q)$ 
	as defined in \eqref{defK}.
  \end{remark}	
 \subsection{Choice of $X,Y$}\label{5.4}
 
  In this section we describe the shape of the parameters $X$ and $Y$ subject to the relation \eqref{XYrelation}. 
 
Recall that the contributions to $M^{e}_{a,b,c}(\xi;q)$ and $M^{o}_{a,b,c}(\xi;q)$ of the non-generic characters is bounded by
 $$ \ll q^{-1/2+o(1)}
	+( {X }/{q^2})^{1/2}q^{o(1)}+( {Y }/q^{3})^{1/2}q^{o(1)}.$$
	For this bound to be admissible, it is sufficient that
	\begin{equation}
		\label{1stXcontrol}
		q^\delta\leq X\leq q^{2-\delta}
	\end{equation}
	for some fixed $\delta>0$.
\begin{remark} As we have seen in Remark \ref{remtrivialchar} \eqref{chi0sum2}, the contribution to \eqref{Megenabc} of the trivial character $\chi_0$ is asymptotic to 
$- \frac{X^{1/2}(\log X)^2}q$, so our choice $X\leq q^{2-\delta}$ is  necessary to insure that this term is an $O( q^{-\eta})$ (for any $\eta\in(0,\delta/2)$).
\end{remark}
	
To sharpen our choice of $X$, we consider $M^\bullet_{2,a,b,c}(\xi;q)$ which we expect to  satisfy (at least under suitable conditions)
 \begin{equation}
 	\label{M2need}
 	M^\bullet_{2,a,b,c}(\xi;q)\ll q^{-\eta}
 \end{equation}
 for some $\eta>0$. We have seen in \eqref{supnomrboundintro} that the sum  $K_{a,b,c}(u)$ is bounded independently of $q$: we have
 \begin{equation}\label{universalforK}
 K_{a,b,c}(u)\ll_{a,b,c} 1.
 \end{equation}
 This bound together with \eqref{M2def} and Lemma \ref{Vprop} \eqref{Vliminf} gives
 \begin{equation}
 	\label{M2trivialbound}
 	M^\bullet_{2,a,b,c}(\xi;q)=\frac{\iota_\bullet}{q^{1/2}}\sum_\stacksum{l,m,n\geq 1}{(lmn,q)=1}\frac{1}{\sqrt{lmn}}K_{a,b,c}(\xi l^am^bn^c)V_\bullet(\frac{lmn}{Y})\ll_{a,b,c} \frac{Y^{1/2+o(1)}}{q^{1/2}}
 \end{equation}
  Such a bound would be enough if we could choose $Y\leq q^{1-\delta}$; unfortunately, from \eqref{XYrelation} and \eqref{1stXcontrol} the minimal possible value for $Y$ is $Y=q^{1+\delta}$. Although \eqref{M2trivialbound} is not quite enough to obtain \eqref{M2need}, it comes close (especially as $\delta$ gets smaller) and all we need is to be able to exploit the oscillations of the function 
 $$(l,m,n)\mapsto K_{a,b,c}(\xi l^am^bn^c)$$
 within the range $$lmn\asymp Y=q^{1+\delta}.$$
In Proposition \ref{trilinearabcK}, we  show that this is possible if $(a,b,c)$ is \good\ and $\delta$ is sufficiently small:  we will therefore take $X,Y$ of the shape
  \begin{equation}
  	\label{XYchoice}
  	X=q^{2-\delta},Y=q^3/X=q^{1+\delta}
  \end{equation}
for $\delta\in (0,1/2)$ sufficiently small possibly depending on $a,b,c,d$ but not on~$q$.

In the next sections we discuss the evaluation of all the terms present in \eqref{Mevenabcbasicdecomp} and \eqref{Moddabcbasicdecomp}.

We start with the fifth terms appearing on the right--hand side of 
\eqref{Mevenabcbasicdecomp} and \eqref{Moddabcbasicdecomp}: these are bounded by
\begin{equation*}
	O(q^{-1/2+o(1)}+( {X }/{q^2})^{1/2}q^{o(1)}+( {Y }/q^{3})^{1/2}q^{o(1)})=O( q^{-\delta/2+o(1)})
\end{equation*}
which is admissible.

\section{Analysis of $M^\bullet_{1,a,b,\pm c}(\xi;q)$}
\label{secM1}
Let $a,b,c\geq 1$ be setwise coprime integers ; let $d\geq 1$ be another integer  and let $q$  be a prime and $\xi \in \mu_d(\Fq)$ be a $d$-th root\footnote{The $d$ in the present section is equal to the $d$ of the previous section if $d$ is even or twice this value if $d$ is odd} of $1$ modulo $q$. 

Our aim in this section is to discuss the evaluation of the sums
\begin{equation*}
M^\bullet_{1,a,b,\pm c}(\xi;q):=\sum_\stacksum{(lmn,q)=1}{l^am^bn^{\pm c}\equiv  \xi \mods q}\frac{1}{\sqrt{lmn}} V_\bullet(\frac{lmn}{X})
\end{equation*}
for $\bullet=e$ or $o$ and for $X$ given by \eqref{XYchoice}.


\subsection{Isolating  the main term}\label{subsecmainterm}
We expect $M^\bullet_{1,a,b,\pm c}(\xi;q)$ to be small (ie. $O(q^{-\eta'})$ for some $\eta'>0$ depending at most on $a,b,c,d$) excepted when $\xi=1$. In that case, we expect a main term coming from the contribution of the $l,m,n$ satisfying one of the two equalities (in $\Zz$)
\begin{equation}\label{lmn=1lm=n}
l^am^bn^c=1\hbox{ or } l^am^b=n^c.
\end{equation}
For the first equality, the only possibility is $l=m=n=1$ which gives by Lemma \ref{Vprop}\eqref{Vlimzero} the equality
\begin{equation}\label{sharp}
M^{\bullet,\main}_{1,a,b,c}(1;q)=V_\bullet(\frac{1}{X})=1+O(X^{-1/2+\eps}).
\end{equation}
For the second equality of \eqref{lmn=1lm=n} we have
$$M^{\bullet,\main}_{1,a,b,-c}(1;q)=\sum_{l^am^b=n^c}\frac{1}{(lmn)^{1/2}}V_\bullet(\frac{lmn}{X}).$$
We define the error term as the complement:
\begin{equation} 
	\label{MTETdecomp}M^{\bullet}_{1,a,b,\pm c}(1;q) =M^{\bullet,\main}_{1,a,b,\pm c}(1;q)+ M^{\bullet,\err}_{1,a,b,\pm c}(1;q)
\end{equation}

\begin{proposition}\label{MTprop} Let $a$, $b$ and $c$ be three positive setwise coprime integers with $c\not\in\{a,b\}$. Then
for $\bullet\in\{e,o\}$ and    for $X$ tending to infinity, 
we have  the equality 
\begin{equation}
	\label{M1mainasymp}
	M^{\bullet,\main}_{1,a,b,- c}(1;q)=D_{a,b,- c}+O(X^{-1/6+o(1)}).
\end{equation}	
where $D_{a,b,- c}>1$.
\end{proposition}
\proof Consider the Dirichlet series (converging for $\Re s\gg 1$)
\begin{equation}\label{defDabc}D_{a,b,-c}(s)=\sum_{l^am^b=n^c}\frac{1}{(lmn)^s}=\prod_pD_{a,b,-c;p}(s),
\end{equation}
where
$$D_{a,b,-c;p}(s)=\sumsum_\stacksum{\alpha,\beta,\gamma\geq 0}{\alpha a+\beta b=\gamma c}\frac{1}{p^{(\alpha+\beta+\gamma)s}}=1+\sum_{k\geq 1}\frac{c_{a,b,-c}(k)}{p^{ks}}$$
say. Uniformly for $k \geq 1$, we have $0\leq c_{a,b,-c}(k)\ll_{a,b,c} k^2$; therefore $D_{a,b,-c}(s)$
 is absolutely converging for $\Re s> 1$. In addition
 given $k\geq 0$, the system of two equations in non-negative integers $(\alpha,\beta,\gamma)$
$$
\begin{cases}\alpha+\beta+\gamma=k,\\ \ \alpha a+\beta b-\gamma c=0,
\end{cases}
$$
has only one solution for $k=0$ namely $(0,0,0)$ ; no solution for $k=1$ and no solution for $k=2$ unless $a=c$ or $b=c$ in which case there is only one solution (either $(1,0,1)$ or $(0,1,1)$). This shows that for $c\not=a$ and $b$, the series  $D_{a,b,-c}(s)$ is absolutely converging for $\Re s>1/3.$ In particular we have the inequality
\begin{equation}\label{>1}
D_{a,b,-c} (1/2)>1,
\end{equation}
following from the positivity of the terms in definition \eqref{defDabc}.\\

Recall that
$$V_\bullet(y)=\frac{1}{2\pi i}\int_{(3)}\frac{\gamma_\bullet(\frac{1}2+u)}{\gamma_\bullet(\frac{1}2)}y^{-u}\frac{du}u$$
so that
\begin{equation}
	\label{Mmainint}
	M^{\bullet,\main}_{1,a,b,-c}(1;q)=\intc_{(3)}D_{a,b,-c}(u+1/2)\frac{\gamma_\bullet(\frac{1}2+u)}{\gamma_\bullet(\frac{1}2)}X^{u}\frac{du}u,
\end{equation}

Shifting the line of integration  in \eqref{Mmainint} to  the vertical line $\Re u =u_0$ where $u_0 \in (-1/6,0)$ is fixed,  we hit a simple pole at $u=0$,  and  obtain that for any $\eps>0$ 
\begin{align*}\nonumber M^{\bullet,\main}_{1,a,b,-c}(1;q)&=\res_{u=0}\left\{D_{a,b,-c}(1/2+u) \frac{\gamma_\bullet(\frac{1}2+u)}{\gamma_\bullet(\frac{1}2)}\frac{X^{u}}{u}\right\}+O_\eps(X^{-1/6+\eps})\\
&=D_{a,b,-c}(1/2)+O_\eps(X^{-1/6+\eps}).
\end{align*}
The definition $D_{a,b,-c} = D_{a,b,-c}(1/2) $ and the inequality \eqref{>1} complete the proof of  Proposition~\ref{MTprop}.
\qed

\begin{remark} If  $c=a$ or $c=b$ (say $c=a$) we have
 $$D_{a,b,-a;p}(s)=1+\frac{1}{p^{2s}}+ \sum_{k\geq 3}\frac{c_{a,b,-a}(k)}{p^{ks}}$$
and we see that 
\begin{equation*}
D_{a,b,-a}(s)=\zeta(2s)E_{a,b,-a}(s),
\end{equation*}
with $E_{a,b,-a}(s)$ absolutely converging for $\Re s>1/3$ and defining an holomorphic function there. In that case the term $D_{a,b,-c}$ has to be replaced by $P_{a,b,-a}(\log X)$ for $P_{a,b,-a}(T)$ a polynomial of degree $1$.
\end{remark}

\subsection{Analyzing the error terms}\label{subsecerrorterm}

Recall that
$$
	M^{\bullet,\err}_{1,a,b,\pm c}(1;q):=M^{\bullet}_{1,a,b,\pm c}(1;q)-M^{\bullet,\main}_{1,a,b,\pm c}(1;q)
$$
and for $\xi\not=1$ we set
 \begin{equation}\label{M1=M1}
 M^{\bullet,\err}_{1,a,b,\pm c}(\xi;q):=M^{\bullet}_{1,a,b,\pm c}(\xi;q).
 \end{equation}
 
 The treatment of these terms, after performing a dyadic partition of the variables, $l$, $m$ and $n$ rests on counting solutions to monomial congruence equations in small boxes.

 We recall the following conjecture from the introduction.
 \begin{conjectureP}  Let $a,b,c$ be non-zero setwise coprime integers and $d\geq 1$ be another integer. For $q$ a prime and $L,M,N\geq 1$  let
$$
	N_{a,b,c,d}(L,M,N;q):=\ |\{(l,m,n)\sim L\times M\times N,\ (l^am^bn^c)^d\equiv 1\mods q,\ l^am^bn^c\not=1\}|.
$$
 
There exists $\eta_0>0$ (depending on $a,b,c,d$ but not on $q$) such that   for any $\eps>0$, one has 
\begin{equation}
	\label{Nabcwishedbound}
	\frac{N_{a,b,c,d}(L,M,N;q)}{\sqrt{LMN}}\ll_{\eps,a,b,c,d} q^\eps\Bigl( \frac{\sqrt{LMN}}{q}+(LMN)^{-\eta_0}\Bigr).
\end{equation}
 \end{conjectureP}
 \begin{remark} The condition $l^am^bn^c\not=1$ is only necessary for  $(a,b,c)$ of the shape $(a,b,-a)$ (say) since then
 $$\frac{|\{(l,m,n)\sim L\times 1\times L,\ l^{a}m^bn^{-a}=1\}|}{\sqrt{L.1.L}}\geq \frac{|\{(l,1,l),\ l\sim L\}|}{L}=1.$$
 In any other case we may safely forget this condition.
 \end{remark}
  
 We  discuss conjecture $\msP(a,b,\pm c,d)$ further in the next subsections. For now, we show that it implies satisfactory bounds for $M^{\bullet,\err}_{1,a,b,\pm c}(\xi;q)$: 
 
 \begin{proposition}\label{M1errlemma}
 Notations and conventions being as in Theorem  \ref{mainthm}; let $X=q^{2-\delta}$ for some fixed $\delta\in (0,2).$ 
Suppose that Conjecture $\msP(a,b,\pm c,d)$ holds; for any $\xi\in \mu_{d}(\Fq)$ we have 
\begin{equation}
	\label{desiredboundforM1err}
	M^{\bullet,\err}_{1,a,b,\pm c}(\xi;q)\ll q^{-\Delta}
\end{equation}
for some $\Delta>0$ depending on $a,b,c,d,\delta$.
\end{proposition}

Before starting the proof of this proposition, we observe that Conjecture $\msP(a,b,\pm c,d)$ is automatic if $LMN$ is sufficiently small: 
 \begin{lemma}\label{triplepriercelowerbound} If
 \begin{equation}\label{LMN<1/6} LMN\leq \frac16\, q^{\frac{1}{2}\cdot \frac1{d\max(|a|,|b|,|c|)}}.
 \end{equation}
 we have
 $$N_{a,b,c,d}(L,M,N;q)=0$$

  \end{lemma}
  
 \proof
Without loss of generality we may assume that $a,b>0$. 

Assume that $N_{a,b,c,d}(L,M,N;q)>0$.

\subsubsection*{The case $c>0$.}
Given $(l,m,n)$ such that $(l^am^bn^c)^d\equiv 1\mods q$ and $l^am^bn^c\not=1$; then 
$$(l^am^bn^c)^d\geq q+1$$ 
 which implies our claim.

\subsubsection*{The case $c<0$.}
 Replacing $c$ by $-c$ we are looking at the solutions of the congruence
$$l^am^b=\xi.n^c\mods q$$
where $(lmn,q)=1$ and $\xi$ ranges over the group of $d$-th roots of unity $\mu_d(\Fq)$.

In other words, we are looking for any $\xi\in \mu_d(\Fq)$, at the integral points of the lattice
$$\Lambda_\xi=\{(x,y)\in \Zz^2,\ x-\xi.y\equiv 0\mods q\}$$
whose coordinates are of the shape $(x,y)=(l^am^b,n^c)$.

If $\xi\not=\pm 1$, its minimum satisfies
$$\lambda_\xi=\min(|x|+|y|,\ (x,y)\in \Lambda_\xi-\{(0,0)\})\geq q^{1/d}.$$
In particular we have
$$l^am^b+n^c\geq q^{1/d}\Longrightarrow \max(l^a,m^b,n^c)\geq \frac{1}{3}q^{1/2d}$$
which implies our claim.

If $\xi=-1$ we have
$$l^am^b+n^c\equiv 0\mods q$$ and since 
$l^am^b+n^c>0$ we have
$$l^am^b+n^c\geq q \Longrightarrow \max(l^a,m^b,n^c)\geq \frac{1}{3}q^{1/2}$$
which implies what we want.

If $\xi=1$ we have
$$l^am^b-n^c\equiv 0\mods q.$$
Since, by hypothesis, we have  $l^am^b-n^c\not=0$ we deduce
$$|l^am^b-n^c|\geq q$$ which again implies what we want. 
\qed 
 \proof (of Proposition  \ref{M1errlemma}). By a smooth partition of unity (see for instance \cite{FouvryCrelle}*{Lemme 2}), we have 

\begin{equation}\label{defMbulleterr}M^{\bullet,\err}_{1,a,b,\pm c}(\xi;q)=\sum_{L,M,N\geq 1} \sumsum_\stacksum{l^am^bn^{\pm c}\equiv \xi\mods q}{l^am^bn^{\pm c}\not= 1}\frac{1}{(lmn)^{1/2}}V_\bullet(\frac{lmn}{X})U(\frac{l}L)U(\frac{m}M)U(\frac{n}N),
\end{equation}
where $(L,M,N)$ ranges over $O(\log^3 q)$ triples of real numbers satisfying
$$L,M,N\geq 1,\ LMN\ll X=q^{2-\delta}$$
and $U(x)$ a smooth function, compactly supported in $[1/2,2]$ and satisfying $$x^{k}U^{(k)}(x)\ll_k 1$$
for any $k\geq 0$.

The $(L,M,N)$ that are contributing satisfy $LMN\ll X\leq q^{2-\delta}$ and we have
\begin{equation}\label{2519}
\sumsum_\stacksum{l^am^bn^{\pm c}\equiv \xi\mods q}{l^am^bn^{\pm c}\not= 1}\frac{1}{(lmn)^{1/2}}V_\bullet(\frac{lmn}{X})U(\frac{l}L)U(\frac{m}M)U(\frac{n}N)\ll \frac{N_{a,b,\pm c,d}(L,M,N;q)}{\sqrt{LMN}}
\end{equation}
since the congruence $l^am^bn^{\pm c}\equiv \xi\mods q$ implies that $ (l^am^bn^{\pm c})^d\equiv 1\mods q$.
However by Lemma \ref{triplepriercelowerbound} we need only to consider the $(L,M,N)$ that do not satisfy  \eqref{LMN<1/6}.  In that situation 
Conjecture $\msP(a,b,\pm c,d)$ gives the bound 
$$\frac{N_{a,b,\pm c,d}(L,M,N;q)}{\sqrt{LMN}}\ll \left( q^{-\delta/2}+q^{-\eta_0/(2d\max(\vert a \vert, \vert b\vert, \vert c\vert))}
\right) q^\varepsilon.$$
Summing over the corresponding $(L,M,N)$ gives \eqref{desiredboundforM1err} and completes the proof of Proposition \ref{M1errlemma}.
\qed

At the moment we are unable to prove Conjecture $\msP(a,b,\pm c,d)$ in full generality but at least we observe that Lemma \ref{Vprop} \eqref{Vpositivity}  implies  the lower bound
\begin{equation*}
	M^{\bullet,\err}_{1,a,b,\pm c}(1;q):=\sumsum_\stacksum{l^am^bn^{\pm c}\equiv \xi\mods q}{l^am^bn^{\pm c}\not= 1}\frac{1}{(lmn)^{1/2}}V_\bullet(\frac{lmn}{X})\geq 0.
\end{equation*}

 \begin{remark}\label{remdiffprop}In Proposition \ref{differencelemma} below, we establish an unconditional variant of Proposition \ref{M1errlemma}: we show that if $(a,b,c)$ is \good\ or oxozonic and $X=q^{2-\delta}$ for $\delta>0$ sufficiently small, one has 
  $$M^{e,\err}_{1,a,b,c}(\xi;q)-M^{o,\err}_{1,a,b,c}(\xi;q)\ll_{a,b,c,d,\delta}q^{-\Delta}$$
for some $\Delta=\Delta(\delta)>0$.  The proof uses the methods developed in the next section.

 \end{remark}
\subsection{Some comments on Conjecture $\msP(a,b,\pm c,d)$}

\subsubsection*{A random model}	Suppose that
intervals $[L,2L)$, $[M,2M)$ and $[N,2N)$ are replaced by independent random sets $\mcL,\mcM,\mcN\subset\Fqt$ of sizes $L,M$ and $N$ with $L+M+N=o(q)$. Using a standard Bernoulli variable interpretation, one finds that the expectation of  $$N_{a,b,c}(\mcL,\mcM,\mcN)=|\{l^am^bn^c\equiv 1\mods q,\ (l,m,n)\in\mcL\times\mcM\times\mcN\}|$$ is $$E_{a,b,c}(L,M,N)=(a,b,c,q-1)\frac{LMN}{q-1}$$ and its variance satisfies
$$V_{a,b,c}(L,M,N)=(a,b,c,q-1)\frac{LMN}{q-1}(1-\frac{(a,b,c,q-1)}{q-1}) \ll_{a,b,c}\frac{LMN}{q}$$
and therefore, given $\eta_0>0$, we have from Chebycheff inequality  
$$\frac{N_{a,b,c}(\mcL,\mcM,\mcN)}{\sqrt{LMN}}\ll_{a,b,c} \frac{\sqrt{LMN}}{q}+(LMN)^{-\eta_0}$$
with probability $\geq 1+O_{a,b,c}((LMN)^{2\eta_0}/q)$.

\subsubsection*{The conjecture holds for sufficiently long ranges}
 The bound \eqref{Nabcwishedbound} holds if $LMN$ is
   sufficiently large compared to $q$ (sometimes
   even with an asymptotic formula for~$N_{a,b,c,d}(L,M,N;q)$). For instance \eqref{Nabcwishedbound} holds whenever
   $$LMN\geq q^{2}.$$
   This follows\footnote{going somewhat backwards from what we do in the present paper} from expressing the congruence $(l^am^bn^c)^d\equiv 1\mods q$ using Dirichlet characters (after discarding the condition $l^am^bn^c=1$) and expressing $N_{a,b,c,d}(L,M,N;q)$ (after smoothing the variables $l,m,n$) as an integral of cubic mixed moments 
   $$\frac{1}{q-1}\sum_\stacksum{\chi\mods q}{\chi^{ad},\chi^{bd},\chi^{cd}\not=1}L(s,\chi^{ad})L(t,\chi^{bd})L(u,\chi^{cd})$$
   along the vertical lines $\mathrm{Re}(s)=\mathrm{Re}(t)=\mathrm{Re}(u)=1/2$.
  H\"older's inequality and standard bounds for fourth moments of Dirichlet $L$-functions (see \cite{Mont}*{Chap. 10}) allows to conclude. 
  
 For our present problem we unfortunately require Conjecture $\msP(a,b,c,d)$ for  $LMN\leq q^{2}$ where the term $\sqrt{LMN}/q$ on the right-hand side of \eqref{Nabcwishedbound} start decays as $LMN$ becomes smaller. The range 
 \begin{equation}\label{q2range}
 	LMN\asymp q^{2}
 \end{equation} 
 seems also to be the limit for the stratification methods of Fouvry and Katz \cite{FK}. It turns out that the techniques of
   \S~\ref{diffsection} allow us to pass this range and to
   obtain \eqref{Nabcwishedbound} when $(a,b,c)$ is \good\ or oxozonic
   and \begin{equation}
   	\label{subcriticalrange}
   	LMN\geq q^{2-\delta}
   \end{equation} with $\delta>0$ sufficiently small
   (possibly depending on $a,b,c,d$). However we apparently need  \eqref{Nabcwishedbound} for $LMN$ extremely small: up to sizes
   $q^{\frac{1}{2}\cdot \frac1{d\max(|a|,|b|,|c|)}}$.
   
   We hope to pursue the study of Conjecture $\msP(a,b,c,d)$ in a later work.

\subsection{Proof of the conjecture  when $a=b$}
\label{seca=b}

 \begin{theorem}\label{thmaac} Let $a,c$ be non-zero coprime integers and let $d\geq 1$ be another integer then Conjecture $\msP(a,a,c,d)$ holds.
 
 That is, there exists $\eta_0>0$ (depending on $a,c,d$) such that for  any prime $q$, $L,M,N\geq 1$ and any $\eps>0$, one has 
\begin{equation}
	\label{Naacwishedbound}
	\frac{N_{a,a,c,d}(L,M,N;q)}{\sqrt{LMN}}\ll_{\eps,a,c,d} q^\eps( \frac{\sqrt{LMN}}{q}+(LMN)^{-\eta_0}).
\end{equation}

 \end{theorem}
 
 \proof
Without loss of generality we may assume that $a(=b)\geq 1$.  

We observe again that
$$((lm)^an^c)^d\equiv 1\mods q\Longrightarrow \exists \xi\in \mu_d(\Fq),\ (lm)^an^c\equiv \xi \mods q.$$

The main idea is to group the two variables $lm$ into a single variable (which we still denote by $m$). By the divisor bound we have therefore
\begin{equation}
	\label{a=breduction}
	N_{a,a,c,d}(L,M,N;q)\ll_{\eps} (LM)^\eps d\max_{\xi\in\mu_d(\Fq)}N_{a,c}(\xi;LM,N)
\end{equation}
with 
$$
N_{a,c}(\xi;M,N;q)=|\{(m,n)\sim M\times N,\ m^an^c\equiv \xi\mods q,\ m^an^c\not=1\}|.$$
To simplify notation we write $M$ for $LM$.

Observe that if $MN\geq q^2$ then $\max(M,N)\geq q$ and
$$\frac{N_{a,c}(\xi;M,N;q)}{\sqrt{MN}}\ll_{a,c}\frac{\min(M,N)}{\sqrt{MN}}(\frac{\max(M,N)}q+1)\ll \frac{\sqrt{MN}}q. $$
From now on we assume that $$MN\leq q^2.$$

The case $a=c=1$ is easy. Indeed, appealing to the classical divisor function $d_2$ we have the bound
$$
N_{1,1} (\xi;M,N) \leq \sum_{1\leq t \leq 4MN\atop t\equiv \xi (\bmod q)} d_2 (t) \ll (MN)^\varepsilon (\frac{MN}q +1),
$$
which leads to \begin{equation*}
\frac{N_{1,1}(\xi; M, N)}{\sqrt{MN}} \ll (MN)^\varepsilon (   \frac {\sqrt{MN}} q+(MN)^{-1/2}).
\end{equation*}
This  together with \eqref{a=breduction}	gives Theorem \ref{thmaac} in the case $a =c=1$. \\ 
Thus for the rest of the proof we assume that the integers $a$ and $c$ satisfy
\begin{equation}\label{2568}
a\geq 1, \ c\neq 0, (a,c)\neq (1,1) \text{ and g.c.d.}(a,c) =1.
\end{equation}

By the same argument as in Proposition \ref{triplepriercelowerbound} we have $N_{a,c}(\xi;M,N)=0$ if 
$$MN\leq \frac 14 q^{\frac{1}{d\max(a,|c|)}}.$$
Therefore we assume that
$$\frac 14q^{\frac{1}{d\max(a,|c|)}}\leq MN\leq q^2.$$

\subsubsection*{Consequence of the trivial bound}
Write
\begin{equation}
	\label{gammadef}
	M=\sqrt{MN}^{\,1-\gamma},\ N=\sqrt{MN}^{\,1+\gamma},\ \gamma\in[-1,1]
\end{equation}
so that 
$$M/N
 =\sqrt{MN}^{\,-2\gamma},\ N/M=\sqrt{MN}^{\,2\gamma}.$$
We have the trivial bound obtained just by counting congruences 
$$\frac{N_{a,c}(\xi;M,N)}{\sqrt{MN}}\leq \frac{\min(M,N)}{\sqrt{MN}}(\frac{\max(M,N)}{q}+1)=\frac{\sqrt{MN}}q+\sqrt{MN}^{\, -|\gamma|}.$$
This bound is satisfactory as long as $|\gamma|\geq \eta$ for some $\eta>0$ to be chosen sufficiently small later.

We may therefore assume that 
\begin{equation}\label{gamma<}|\gamma|\leq \eta,
\end{equation} which means that
$M$ and $N$ are close to one another:
\begin{equation}
	\label{closebyMN}
	M/N,N/M\leq (MN)^{|\gamma|}\leq q^{2\eta}.
\end{equation}

The case $(a,c)=(1,-1)$ will require a different kind of argument based on the properties of the lattice $\Lambda_\xi$ (see \S\,  \ref{case11}). Here we deal first with the case $(a,c)\not=(1,-1)$.

\subsubsection{The case $MN$ large} 
If $MN$ is large by \eqref{closebyMN}, both $M$ and $N$ are large; we  then smoothen  the $m$ and $n$ variables (which is possible since the expression we want to evaluate is a sum of non-negative terms),  and  we apply the Poisson summation formula to both $m$ and $n$. Thus  we obtain
$$\frac{N_{a,c}(\xi;M,N)}{\sqrt{MN}}\ll \frac{\sqrt{MN}}{q^2}\sumsum_{m'\ ,\ n'}({1+|m'|M/q})^{-2}({1+|n'|N/q})^{-2}\bigl|\sumsum_\stacksum{x,y\mods q}{x^ay^{c}\equiv \xi\mods q}e(\frac{m'x+n'y}{q})\bigr|.$$

The contribution of the frequencies  $(m',n')=(0,0)$  is bounded by
$$\ll \frac{\sqrt{MN}}{q^2}q=\frac{\sqrt{MN}}q.$$

Since $(a,c)\not=(1,-1)$, the contribution  of the terms with $(m',n')\not=(0,0)$ is bounded using Weil type bounds for algebraic exponential sums in one variable (and sometimes just bounds for Gauss sums). More precisely, we have 
$$\sumsum_\stacksum{x,y\mods q}{x^ay^{c}\equiv \xi\mods q}e(\frac{m'x+n'y}{q})\ll_{a,c} q^{1/2}.$$
We obtain  
\begin{align*}
	\frac{N_{a,c}(\xi;M,N)}{\sqrt{MN} }&\ll \frac{\sqrt{MN}}{q}+\frac{\sqrt{MN}}{q^2}(1+\frac{q}M)(1+\frac{q}N)q^{1/2}\\
	&\ll \frac{\sqrt{MN}}{q}+(\frac{M/N+N/M}{q})^{1/2}+\frac{\sqrt{MN}}{q^{3/2}}+(\frac{q}{MN})^{1/2}\\
	&\ll \frac{\sqrt{MN}}{q}+(\frac{q}{MN})^{1/2}+q^{\eta-1/2}.
\end{align*} 
This inequality reduces to 
\begin{equation}\label{2234}
	\frac{N_{a,c}(\xi;M,N)}{\sqrt{MN} }\ll  \frac{\sqrt{MN}}{q}+(MN)^{-\eta/(1+2\eta)},
\end{equation}
as long as one assumes the inequality
$$MN\geq q^{1+2\eta}$$
for any $\eta\in(0,1/4)$.

\subsubsection{The case $MN$ small}
It remains to bound $N_{a,c}(\xi;M,N)$ for $M,N$ satisfying 
\begin{equation}
	\label{MNbound2}
	(1/4)\,q^{1/d\max(a,|c|)}\leq MN\leq q^{1+2\eta},\ M/N,N/M\leq q^{2\eta}.
\end{equation}
In particular we have the inequalities 
$$M,N\leq q^{1/2+2\eta}.$$
We now recall the following result of L. Pierce~\cite[Th.\,4]{Pierce} that was already used in \cite{FKMAA}.

\begin{theorem}[Pierce]\label{thmpiercebis}
  Given $a,c\geq 1$ be positive integers. Let $q$ be a prime, $\xi\in\Fqt$ and $k\geq 1$ be an integer. Let $M,N\geq 1$ be  such that
  \begin{equation*}
    M\leq \demi q^{\frac{k+1}{2k}},\ N\leq \tfrac{q}{4}, 
  \end{equation*}
then   we have
  $$
  N_{a,c}(\xi;M,N)\ll 
  M^{\frac{k}{k+1}}N^{\frac{1}{2k}}\log q
  $$
where the implied constant depends only
  on~$(a,c,k)$. Moreover if $a/c$ is not an integer, the same bound holds for  $
  N_{a,-c}(\xi;M,N)$.

\end{theorem} 
Actually, Pierce proved Theorem \ref{thmpiercebis} for $\xi=1$ but the proof is valid  for general $\xi$'s (see also \cite{moreira}*{Lemma 1.4}).

\begin{remark} Since in Theorem \ref{thmaac},  $a$ and $c$ are coprime $a/c$ is an integer only if $c=1$.
\end{remark}
With the notations \eqref{gammadef} we have
\begin{equation}\label{Mk}
\frac{M^{\frac{k}{k+1}}N^{\frac{1}{2k}}}{\sqrt{MN}}=(\sqrt{MN})^{-\mathcal{E}_{k,\gamma}},
\end{equation}
with
$$\mathcal{E}_{k,\gamma}=\frac{k-1+\gamma(2k^2-k-1)}{2k(k+1)}.$$
 If $\eta$ satisfies 
$$\eta\leq \frac{k-1}{8k^2}$$ the bound
$|\gamma|\leq \eta$  (see \eqref{gamma<}) implies that
\begin{equation}
	\label{Ekgammalower}\mathcal{E}_{k,\gamma}\geq \frac{k-1}{4k(k+1)}.
\end{equation}
In order to apply Theorem \ref{thmpiercebis} we need to check that
 $M\leq \demi q^{\frac{k+1}{2k}}.$
 For this it is sufficient that (cf. \eqref{gamma<} and \eqref{MNbound2})
$$M\leq \sqrt{MN}^{1+|\gamma|}\leq q^{\frac12(1+2\eta)^2}\leq \frac12q^{\frac{k+1}{2k}}$$
or
$$\eta < \frac{1}{2}((\frac{k+1}{k})^{1/2}-1).$$
Choose $k=3$ and  take 
$\eta= 1/36$, so that
$\mathcal{E}_{3,\gamma}\geq 1/24$ by \eqref{Ekgammalower}. By \eqref{Mk} we have  the upper bound 
\begin{equation}
	\label{desiredbound}
	\frac{N_{a,c}(\xi;M,N)}{\sqrt{MN}}\ll \sqrt{MN}^{\, -1/24+o(1)}.
\end{equation}

This reasoning is also valid for $N_{a,-c}(\xi;M,N)$ as long as $a/c$ is not an integer. However, by the restriction \eqref{2568},  either $a/c$ or $c/a$ is not an integer so up to switching $m,n$ \eqref{desiredbound} also holds with $N_{a,-c}(\xi;M,N)$ (notice that the lower bound \eqref{Ekgammalower} is independent of the sign of $\gamma$).

Combining   \eqref{2234} with \eqref{desiredbound} and recalling that $\eta = 1/36$,   we obtain that when $a$ and $c$ are coprime, different from $(1,-1)$,
$$\frac{N_{a,c}(\xi;M,N)}{\sqrt{MN} }\ll  \frac{\sqrt{MN}}{q}+(MN)^{-\eta_0+o(1)}$$
with $$\eta_0=\min (\eta/(1+2 \eta), 1/24) =1/38.$$ 
We then deduce Theorem \ref{thmaac} from \eqref{a=breduction}.

\subsubsection{Bounding $N_{1,-1}(\xi;M,N)$}\label{case11} The case $(a,c)=(1,-1)$ is a bit degenerate and requires slightly different arguments taken from \cite{FKMAA}*{\S 4}.

 As we have already observed, we have
$${N_{1,-1}(\xi;M,N)}=0$$
for $M,N\leq q/4$ since for $1\leq m,n<q/2$
$$m\equiv n\mods q \Longrightarrow m=n.$$
Likewise, we have $N_{1,-1}(-1;M,N)=0$ for $M,N\leq q/4$.

For $\max(M,N)\geq q/4$ the trivial bound gives
$$\frac{N_{1,-1}(\pm 1;M,N)}{\sqrt{MN}}\ll \frac{\sqrt{MN}}q.$$
This concludes the proof for $(a,c,\xi)=(1,-1,\pm 1)$. 

 We now appeal to the notations \eqref{gammadef}. For $\xi\in\mu_d(\Fq)-\{\pm 1\}$, we have by the trivial bound 
$$\frac{N_{1,-1}(\xi,M,N)}{\sqrt{MN}}\leq \frac{\sqrt{MN}}q+\sqrt{MN}^{\, -|\gamma|}$$
which is satisfactory as long as $|\gamma|\geq \eta$. 

If $|\gamma|\leq \eta$ we observe that
our problem amounts to counting the number of points in the intersection of the box $[M,2M]\times [N,2N]$ and of the lattice of covolume $q$
$$\Lambda_\xi=\{(m,n)\in\Zz^2,\ m-\xi n\equiv 0\mods q\}.$$

  By the Lipschitz principle (see \cite{FKMAA}*{Prop 4.5}) we have
$$\frac{N_{1,-1}(\xi;M,N)}{\sqrt{MN}}\ll \frac{\sqrt{MN}}{q}+\frac{1}{\lambda_\xi}\cdot \frac{M+N}{\sqrt{MN}}\ll \frac{\sqrt{MN}}{q}+\frac{1}{\lambda_\xi}(\frac{M}{N}+\frac{N}{M})^{1/2}$$
where $\lambda_\xi$ is the minimum of $\Lambda_\xi$. Since $\xi\in\mu_d(\Fq)-\{\pm 1\}$ we have
$$\lambda_\xi\gg q^{1/d}$$ and (since $MN\leq q^2$) we have
$$\frac{N_{1,-1}(\xi;M,N)}{\sqrt{MN}}\ll \frac{\sqrt{MN}}{q}+\frac{\sqrt{MN}^{|\gamma|}}{q^{1/d}}\ll \frac{\sqrt{MN}}{q}+{\sqrt{MN}^{\, \eta-1/d}}.$$
Taking $\eta=1/2d$ we obtain Theorem \ref{thmaac} for $(a,c)=(1,-1)$ with $\eta_0=1/2d$.\\
The proof of Theorem \ref{thmaac} is now complete.

\subsection{Proof of conjecture $\msP(a,b,c,d)$ on average}

\begin{proposition}
	\label{propfouvry} Given $a,b,c\in\Zz-\{0\}$ setwise coprime, $d\geq 1$ another integer and $Q\geq 2$; then  for any $\eps>0$, we have 
 	$$\frac{1}{|\msP\cap [Q,2Q)|}\sum_{q\in \msP\cap [Q,2Q)}\frac{N_{a,b,c,d}(L,M,N;q)}{\sqrt{LMN}}\ll_{a,b,c,d,\eps} Q^\eps\frac{\sqrt {LMN}}{Q}.$$
\end{proposition}

\proof We may assume that $a,b\geq 1$. 

We discuss the case $c\leq -1$; for $q\in \msP\cap [Q,2Q)$ and
$(lmn,q)=1$, we have the  equivalence
$$(l^am^bn^c)^d\equiv 1\mods q\Longleftrightarrow
(l^am^b)^d-n^{|c|d}=qk,\ k\in\Zz.$$

If $k=0$, then $l^am^b=n^{|c|}$ which is excluded from the count.

The sum we aim to evaluate is therefore bounded by
$$\ll \frac{\log Q}{Q}\sum_\stacksum{l,m,n\sim L\times M\times N}{(l^am^b)^d-n^{|c|d}\not=0}\frac{d_2((l^am^b)^d-n^{|c|d})}{(LMN)^{1/2}}\ll_{a,b,c,d,\eps} Q^\eps\frac{\sqrt{LMN}}{Q}$$
where $d_2$ denotes the divisor function.

For $c\geq 1$, it suffices to replace $(l^am^b)^d-n^{|c|d}$ by
$(l^am^bn^c)^d-1$ in the above argument.  \qed
 
 By gathering Proposition \ref{propfouvry}, the definition \eqref{defMbulleterr} and the inequality \eqref{2519} we obtain:
 \begin{corollary} \label{M1erraverage} Let $a,b,c\geq 1$ be setwise coprime integers and $d\geq 1$ be another integer; let $q\geq 3$ be a prime. 
 We adopt the notations of  Theorem  \ref{mainthm}.
 
Given $\delta\in (0,2)$ and $Q\geq 2$, let $X =Q^{2-\delta}$, then  for any $\eps>0$, we have 
\begin{equation*}
	\frac{1}{|\msP\cap[Q,2Q)|}\sum_{q\sim Q}\sum_{\xi\in \mu_d(\Fq)}|M^{\bullet,\err}_{1,a,b,\pm c}(\xi;q)|\ll_{a,b,c,d,\eps} Q^{-\delta/2+\eps}.
\end{equation*}
 \end{corollary}

\section{The sums  $M_{2,a,b,\pm c}(\xi;q)$}\label{secM2}
Let $a,b,c\geq 1$ be  integers, $a,b,c$ setwise coprime and let $q$ be a prime  and $\xi\in\Fqt$.

Let $M^\bullet_{2,a,b,\pm c}(\xi;q)$ be the sum defined in \eqref{M2def} with $Y=q^{1+\delta}$ for $\delta\in (0,1/2)$  suitably small but fixed. The main objective of this section is to prove the 

\begin{theorem}
	\label{M2abcthm} Notations as above and assume that $(a,b,\pm c)$ is {\good}  or oxozonic in the sense of Definition \ref{defgoodtriple}. 
	
	For any $\delta>0$ sufficiently small (depending on $a,b,c$), there exists $\eta=\eta(\delta)>0$ such that for any $\xi\in \Fqt$ we have
	\begin{equation}
	M^\bullet_{2,a,b,\pm c}(\xi;q)\ll q^{-\eta}.
	\label{M2boundgoal}
\end{equation}
Here the implicit constant depends only on $a,b,c$.
\end{theorem}

\begin{remark}
Here, contrary to the previous section, we do not need to assume that $a=b$ or that $\xi\in\mu_{d}(\Fq)$ for some fixed $d\geq 1$.
\end{remark}

\subsection{Dyadic decomposition of $M^\bullet_{2,a,b,\pm c}(\xi;q)$}\phantom{.}\label{M2decomp}

Recall the definition \eqref{XYchoice} of $Y$.
We perform a smooth dyadic partition on the $l,m,n$ sums in $M^\bullet_{2,a,b,\pm c}(\xi;q)$  and  we are reduced to bounding $O(\log^3 q)$ sums of the  shape:   
$$\Sigma_{a,b,\pm c}(L,M,N):=\sum_{l,m,n}U(\frac{l}L)U(\frac{m}M)U(\frac{n}N)V(\frac{lmn}Y)K_{a,b,\pm c}(\xi l^am^bn^{\pm c}),$$
for $\xi\in\mu_{2d}(\Fq)$, $V=V_\bullet$, $U$ some smooth function, compactly supported in $(1/2,2)$ and satisfying for any $j\geq 0$
\begin{equation}
\label{Ubounds}
	x^jU^{(j)}(x)\ll_j 1,
\end{equation}
 and $L,M,N\geq 1$ satisfy
\begin{equation}
	LMN\leq Y^{1+\delta^3} (\leq q^{1+\delta+\delta^2})
	\label{LMNM2bound}
\end{equation}
for some parameter $\delta>0$. Indeed by the universal bound \eqref{universalforK} and by the decay properties of $V$ (cf. Lemma \ref{LemmaV}) the contribution of the $l,m,n$ such that $lmn\gg Y^{1+\delta^3}$ is negligible.

In order to prove Theorem \ref{M2abcthm} it is sufficient to show that for all  $L,M,N$ satisfying \eqref{LMNM2bound}, we have
\begin{equation*}
	\frac{\Sigma_{a,b,\pm c}(L,M,N)}{q^{1/2}\sqrt{LMN}}\ll_{a,b,c,\delta} q^{-\eta}
	\end{equation*}
for some $\eta=\eta(\delta)>0$.

For {\good} triples this will be a consequence of Proposition \ref{trilinearabcK} below and of Theorem \ref{th-galant-sheaf} which shows that $\mcK_{a,b,\pm c}$ satisfies the assumption of this Proposition.

\subsection{Bounds for trilinear sums with monomials}\label{trilinearwithmonomials}

Given three non zero integers $a,b,c\in\Zz-\{0\}$ not necessarily positive nor coprime and $K:\Zz/q\Zz\to \Cc$ a $q$--periodic function, let
$$\Sigma_{a,b,c}(L,M,N):=\sum_{l,m,n}U(\frac{l}L)U(\frac{m}M)U(\frac{n}N)V(\frac{lmn}Y)K(l^am^bn^{c})$$
where $U,V$ are as above.

\begin{proposition}
	\label{trilinearabcK} There exists $\delta_0>0$ (absolute) such that for any  $\delta\in (0,\delta_0]$, there exists $\eta=\eta(\delta)>0$ such that the following holds.
	
	For any prime $q$, any $\ell$-adic  sheaf $\mcF$, mixed of weight $\leq 0$ with trace function $K$, complexity $c(\mcF)$ and such that  $\mcF$ is {\good} in the sense of Definition \ref{defgoodcomplexsheafintro}, one has for any $L,M,N\geq 1$ with$$LMN\leq Y^{1+\delta^3}$$
	the inequality
\begin{equation}
	\label{boundM2}
	\frac{\Sigma_{a,b,c}(L,M,N)}{q^{1/2}\sqrt{LMN}}\ll q^{-\eta}
	\end{equation}
where the implicit constant depends on $a,b,c,\delta$ and $c(\mcF)$.
\end{proposition}

The arguments in the proof of \eqref{boundM2}  depends on the relative sizes of $L,M,N$ and $q$. Let $0<\delta<1/2026$ be some constant to be chosen sufficiently small.

\subsubsection{The trivial range}\label{trivialrange}

Since 
$$K(x)=O_{c(\mcF)}(1)$$
a  trivial estimate gives \eqref{boundM2} (with $\eta=\delta/2$) as soon as
$$LMN\leq q^{1-\delta}.$$
From now on, we will assume that 
\begin{equation}\label{lowerupperM2}
	q^{1-\delta}\leq LMN\leq q^{1+\delta+\delta^2}.
\end{equation}
As we will see, our assumptions and our arguments are symmetric in the three variables $l,m,n$, we may and will then assume that
\begin{equation}\label{L<M<N}L\leq M\leq N.
\end{equation}
In particular $L\leq q^{(1+\delta+\delta^2)/3}$.

\subsubsection{One  large variable}\label{onelargerange}

The most basic non-trivial argument is the P\'olya-Vinogradov completion method. Since $\mcF$ is {\good}, for any $l,m,n\in\Fqt$, the pull-back sheaves associated to the trace functions
$$x\mapsto K(l^am^bx^c), K(l^ax^bn^c),\ K(x^am^bn^c)$$ 
are still {\good} and therefore {\em Fourier}. In particular their naive Fourier transforms satisfy 
\begin{equation}
	\label{trivialFourier}
	\what K(l^am^b\bullet^c )(y)=\frac{1}{q^{1/2}}\sum_{x\mods q}K(l^am^b x^c)e_q(-xy)=O_{c(\mcF),c}(1)
\end{equation}
and similarly for the second and third arguments.
 
Applying Poisson's summation formula on the $n$-variable, using Lemma \ref{LemmaV}, \eqref{Ubounds}  and  \eqref{trivialFourier}, we obtain 
\begin{equation*}
	\frac{\Sigma_{a,b,c}(L,M,N)}{q^{1/2}\sqrt{LMN}}\ll  \frac{LMq^{1/2}}{q^{1/2}\sqrt{LMN}}=(\frac{LM}{N})^{1/2}\leq \frac{q^{(1+\delta+\delta^2)/2}}{N}\ll q^{-\delta/4}
\end{equation*}
 as long as 
\begin{equation*}
	N\geq q^{1/2+\delta}.
\end{equation*}
We may therefore assume that 
$$L\leq M\leq N\leq q^{1/2+\delta}.$$
From this point on it is useful to separate the variables $l,m,n$. We use inverse Mellin transform:
\begin{multline*}
	\Sigma_{a,b,\pm c}(L,M,N)=\intc_{\Re u=\delta}\frac{\gamma_\bullet(1/2+u)}{\gamma_\bullet(1/2)}(\frac{LMN}{Y})^{-u}\\\times\sum_{l,m,n}U(\frac{l}L)U(\frac{m}M)U(\frac{n}N)(\frac{lmn}{LMN})^{-u}K_{a,b,\pm c}(\xi l^am^bn^{\pm c};q)\frac{du}u.
\end{multline*}

\subsubsection{No small variable}\label{nosmallvariable}
Suppose now that
\begin{equation}\label{newineq}
q^{2\delta}\leq L\leq M\leq N\leq q^{1/2+\delta}.
\end{equation}
From \eqref{lowerupperM2}  and  \eqref{L<M<N} we have
\begin{equation}\label{<MN<}q^{2/3-2\delta/3}\leq MN \leq q^{1-\delta +\delta^2}.
\end{equation}
Applying Theorem \ref{thmtriplesum} with 
$$\alpha_l=(\frac{l}{L})^uU(\frac{l}L),\ \beta_m=(\frac{m}{M})^uU(\frac{m}M),\ \gamma_n=(\frac{n}{N})^uU(\frac{n}N)$$
we obtain for $k=2[1/24\delta]\geq 4$ the inequality  
\begin{equation}
	\frac{\Sigma_{a,b,c}(L,M,N)}{q^{1/2}\sqrt{LMN}}\ll q^{2\delta^2}(\frac{LMN}{q})^{1/2}(\frac{q^{1/2}}{MN}+\frac{q}{L^kMN})^{1/(2k)}.\label{2896}
	 \end{equation}
To bound this expression we incorporate the inequalites
$$
\frac {LMN} q < q^{\delta +\delta^2},
$$
(consequence of \eqref{lowerupperM2})
$$
\frac {q^{1/2}}{MN} < q^{-1/6+2\delta/3},
$$
(consequence of \eqref{<MN<}) and
$$\frac q{L^kMN} \leq q^{1-(1-\delta) -2\delta (k-1)} = q^{-\delta (2k-3)}, $$
(consequence of \eqref{lowerupperM2} and \eqref{newineq}).Thus the right hand--side of \eqref{2896} is
\begin{equation}\label{2910}
\ll q^{2\delta^2} (q^{-1/(12k) +\delta (1/2+1/(3k) )+\delta^2/2} +q^{\delta (-1/2+3/(2k)) +\delta^2/2}) \ll q^{-\delta/3}
\end{equation}
 at least for $\delta>0$ sufficiently small: indeed by the definition of $k$, we have the equality 
\begin{equation}\label{1/k=}\frac{1}{k}=\frac{12\delta}{1+O(\delta)}={12\delta}+O(\delta^2),
\end{equation}
so that the exponents of  $q$ in the right--hand side of \eqref{2910}  are given up to terms of size $O(\delta^2)$ by
$$-\delta+\delta/2=-\delta/2.$$ 
Returning to \eqref{2896}, we proved \eqref{boundM2} is that case.

\subsubsection{One small variable}\label{onesmallvariable}
The remaining case is  $L\leq q^{2\delta}$; in which  case we have 
$$q^{1-3\delta}\leq MN\leq q^{1+\delta +\delta^2}/L,\ q^{1/2-4\delta}\leq M\leq N\leq q^{1/2+\delta},\ N\geq q^{1/2-3\delta/2}.$$
For $k=2[1/24\delta]\leq 1/(12\delta)$ we have
$3/(4k)\geq 9\delta$ so that
$M\leq N \leq q^{1/2+3/(4k)}$ hence $M$ and $N$ satisfy \eqref{assumptypeII} and Theorem \ref{thmTypeII}  gives 
\begin{align*}
	\frac{\Sigma_{a,b,c}(L,M,N)}{q^{1/2}\sqrt{LMN}}&\ll q^{o(1)}(\frac{LMN}{q})^{1/2}\bigl(\frac{1}{N^{1/2}}+(\frac{q^{\frac{3}4(1+1/k)}}{MN})^{1/(2k)})\\
	&\ll q^{o(1)+\frac12\delta(1+\delta)}(q^{-\frac14(1-3\delta)}+q^{\frac{3}{8k}(1+\frac{1}{k})-\frac{1-3\delta}{2k}}).
\end{align*}
Again appealing to \eqref{1/k=}
we see 
 that the exponents in $q$ in the above bound are (up to $o(\delta^2)$ terms) of the form 
$$\frac{1}2\delta-\frac{1}4+\frac{3}4\delta=-\frac{1}4+\frac 54\delta\hbox{ and }\frac{1}2\delta+(\frac{3}{8}-\frac{1}2)12\delta =-\delta.$$
Hence upon choosing a small enough $\delta>0$ we obtain \eqref{boundM2} with $\eta=\delta/4$.
\qed

\subsection{The oxozonic case}

In this section, we complete the discussion of Theorem \ref{M2abcthm} by studying  the specific case of $(a,b,\pm c)$  being an  oxozonic triple (which means that the geometric monodromy group  of $\mcF=\mcK_{a,b,\pm c}$ is $$G_{a,b,\pm c}\simeq \Ort_4$$ acting by the standard representation). We recall that up to permutation (see Definition~\ref{defgoodtriple})
we have $$(a,b,c)=(1,1,2), \ (1,4,-3),\ (1,6,-3) \hbox{ or }(2,3,-1).$$
Let
$$\Sigma_{a,b,\pm c}(L,M,N):=\sumsumsum_{l\sim L,m\sim M,n\sim N}U(\frac{l}L)U(\frac{m}M)U(\frac{n}N)V(\frac{lmn}{LMN})K_{a,b,\pm c}(l^am^bn^{\pm c}).$$

\begin{proposition}
	\label{trilinearoxozonic} There exists $\delta_0>0$ (absolute) such that for any  $\delta\in (0,\delta_0]$, there exists $\eta=\eta(\delta)>0$ such that the following holds.
	
	For any prime $q>2$, and $(a,b,\pm c)$ an oxozonic triple, one has for any $L,M,N\geq 1$ with
	$$LMN\leq q^{1+\delta}$$
	the inequality 
\begin{equation*}
	\frac{\Sigma_{a,b,c}(L,M,N)}{q^{1/2}\sqrt{LMN}}\ll q^{-\eta}
	\end{equation*}
where the implicit constant depends on $a,b,c$ and $\delta$. 
\end{proposition}

\proof Much of the arguments of \S \ref{secM2}  carries over and we only need some slight adaptation. For this we will use Theorem \ref{thmO4} in place of Theorem \ref{thmTypeII}  and Theorem \ref{thmtriplesum}, excepted, when the trace function we want apply the theorem to
$$x\mapsto K_{a,b,\pm c}(l^am^bx^c), K_{a,b,\pm c}(l^ax^bn^c),\hbox{ or } K_{a,b,\pm c}(x^am^bn^c)$$
has $c,b$ or $a$ even is even (for then the geometric monodromy group of the corresponding pull-back is $\SO_4$ instead of $\Ort_4$). In particular, this argument is no longer exactly symmetric in the three variables $l,m,n$. 

Still the arguments of \S \ref{trivialrange} and \ref{onelargerange} continue to hold (even if a pullback has smaller geometric monodromy group $\SO_4$ the sheaf is still Fourier) so we may still reduce to the case 
$$q^{1-\delta}\leq LMN\leq q^{1+\delta+\delta^2},\ q^{1/3+\delta/3}\leq L,M,N\leq q^{1/2+\delta}.$$
Likewise, the argument in \S \ref{nosmallvariable} continue to hold unchanged since \eqref{tripleort} holds independently of the values of $a,b,c$.

 The only issue that could possibly occur then is related to \S \ref{onesmallvariable}: 
suppose to fix ideas that $L\leq M, N$ and that
$$L\leq q^{2\delta},\ q^{(1-3\delta)/2}\leq M, N\leq q^{(1+4\delta)/2}.$$
If the exponent $a$ (corresponding to the first variable $l$) is even, then $b$ and $\pm c$ are odd and we can apply \eqref{BtypeII}. If $a$ is odd, then only one of the two exponents $b,c$ is odd but since both variables $m$ and $n$ are smooth we can choose the one with odd exponent for the $n$ variable in the application of  \eqref{BtypeII}. 
\qed

Obviously Proposition \ref{trilinearoxozonic} implies a version of
Theorem \ref{M2abcthm} where the hypothesis that the triple is \good\ is replaced by oxozonic.

 \section{A variant of Proposition \ref{M1errlemma}}\label{diffsection}

In this section we analyse the difference between the even and odd portions of our moment 
$$
M^{e,\err}_{1,a,b,\pm c}(\xi;q)-M^{o,\err}_{1,a,b,\pm c}(\xi;q)
$$ and show that it is negligible. 
\begin{proposition}
\label{differencelemma} Notations and conventions are  as in Theorem
\ref{mainthm}, and as  in  \eqref{MTETdecomp}. We assume that the
triple $(a,b,\pm c)$ is \good\ or oxozonic. There exists $\delta_0=\delta({a,b,c,d})\in (0,1)$ such that for $0< \delta\leq \delta_{0}$ and 
\begin{equation}\label{X=q}
X=q^{2-\delta},
\end{equation} we have
\begin{equation}
	\label{desiredboundforM1eoerr}
	M^{e,\err}_{1,a,b,\pm c}(\xi;q)-M^{o,\err}_{1,a,b,\pm c}(\xi;q)\ll q^{-\Delta}
\end{equation}
and
\begin{equation}
	\label{desiredboundforM1}
	M^{e}_{1,a,b,\pm c}(\xi;q)-M^{o}_{1,a,b,\pm c}(\xi;q)\ll q^{-\Delta}
\end{equation}
for some $\Delta=\Delta(\delta)>0$.
\end{proposition}  

\proof


First remark that \eqref{desiredboundforM1} is a direct consequence of \eqref{desiredboundforM1eoerr}, of the decomposition \eqref{MTETdecomp}, of the definition \eqref{M1=M1} and of the equalities \eqref{sharp} and \eqref{M1mainasymp} applied with $\bullet = e$ and $o$.

So we restrict ourselves to the proof of \eqref{desiredboundforM1eoerr}. 
By the definitions  \eqref{MTETdecomp} and \eqref{M1=M1} we have the equality
$$M^{e,\err}_{1,a,b,\pm c}(\xi;q)-M^{o,\err}_{1,a,b,\pm c}(\xi;q)=\sumsum_\stacksum{l^am^bn^{\pm c}\equiv \xi\mods q}{l^am^bn^{\pm c}\not= 1}\frac{1}{(lmn)^{1/2}}(V_e-V_o)(\frac{lmn}{X})$$
where we recall that $X=q^{2-\delta}$.

By Lemma \ref{LemmaV} (1) and (2), $V_e-V_o$ is  smooth function on 
$\Rr_{>0}$  satisfying  for $y>0$ and any $A\geq 1$
\begin{equation*}
(V_e-V_o)(y)\ll_{\eps,A}\frac{y^{1/2-1/A}}{(1+y)^A}.
\end{equation*}

Hence, applying this formula with a very large value of $A$ (in terms of $\eps$) and performing a smooth partition of unity we have for any $\eps>0$ the inequality 
$$M^{e,\err}_{1,a,b,\pm c}(\xi;q)-M^{o,\err}_{1,a,b,\pm c}(\xi;q)\ll_\eps  q^{-2026}+ q^\eps \sup_{(L,M,N)}\frac{|S_{a,b,c}(\xi;L,M,N)|}{X^{1/2}}$$ where
$$S_{a,b,c}(\xi;L,M,N)=\sumsum_\stacksum{l^am^bn^{\pm c}\equiv \xi\mods q}{l^am^bn^{\pm c}
\not= 1}W(\frac{lmn}{X})U(\frac{l}L)U(\frac{m}M)U(\frac{n}N)$$
where $U,W$ are smooth functions on $\Rr_{>0}$ satisfying
\begin{equation*}
	y^iU^{(i)}(y),\ y^iW^{(i)}(y)\ll_i 1, 
\end{equation*}
 $\supp U\subset (1/2,2)$ and the parameters $(L,M,N)$ satisfy
\begin{equation}
	\label{LMNrange}
	1\leq L,M,N,\ LMN\leq Xq^{\delta^2}
\end{equation}

Our objective is to prove that for such $(L,M,N)$ we have
\begin{equation}
	\label{whatweneed}
	S_{a,b,c}(\xi ;L,M,N)\ll X^{1/2}q^{-\Delta},
\end{equation}
for some $\Delta=\Delta(\delta)>0$. We will see that
$$\Delta:=\delta/3$$
works when $\delta$ is sufficiently small.
Next, we observe that for $\max(L,M,N)\geq q$, the trivial bound gives
$$
S_{a,b,c}(\xi;L,M,N)\ll \frac{LMN}{\max(L,M,N)}(\frac{\max(L,M,N)}q+1)\ll LMN/q\leq q^{1-\delta+\delta^2}\leq X^{1/2}q^{-\delta/3}.$$ This follows from \eqref{X=q} and  \eqref{LMNrange}. Thus \eqref{whatweneed} is proved in that case (as soon as  $\delta_0$ is picked sufficiently small).

From now on we assume that $$1\leq \max(L,M,N)\leq q.$$

As in \S \ref{onelargerange}, we separate the variables $l,m,n$ from one another (in the expression $W(\frac {lmn} X)$) via Mellin transform; we are reduced to proving \eqref{whatweneed}
 for the following modified  sums
 $$S_{a,b,c,u}(\xi;L,M,N):=\sumsum_\stacksum{l^am^bn^{\pm c}\equiv \xi\mods q}{l^am^bn^{\pm c}\not= 1}U_u(\frac{l}L)U_u(\frac{m}M)U_u(\frac{n}N)$$
 where
 $$U_u(x):=U(x)x^u\hbox{ for }\Re u=\delta^{10}.$$
 We then apply the Poisson summation formula to each of the variables $l,m,n$: we obtain the equality
 \begin{equation}\label{3083}
 S_{a,b,c,u}(\xi;L,M,N)=\frac{LMN}{q^2}\sumsum_{l,m,n}\what{U_u}(\frac{l}{L^*})\what{U_u}(\frac{m}{M^*})\what{U_u}(\frac{n}{N^*})K_{a,b,\pm c}(\xi, l,m,n;q),
 \end{equation}
 where
 $$L^*=q/L\geq 1,\ M^*=q/M\geq 1,\ N^*=q/N\geq 1$$
 and
 $$K_{a,b,\pm c}(\xi, l,m,n;q):=\frac{1}{q}\sum_{x^ay^bz^{\pm c}\equiv \xi\mods q}e_q(lx+my+nz).$$
 Also we observe that by \eqref{LMNrange} we have
\begin{equation*}
	L^*M^*N^*= \frac{q^3}{LMN}\geq q^{1+\delta-\delta^2}\geq q^{1+\delta/2}
\end{equation*}
 
 This sum resembles very much the sum $\Sigma_{a,b,c}(L^*,M^*,N^*)$ discussed in \S \ref{trilinearwithmonomials} excepted that $l,m,n$ could be zero or negative. 
 
 \subsection{The degenerate terms}
 We first handle the contribution of the $(l,m,n)$ such that
 $lmn\equiv 0\mods q$. To simplify some formulas, we make the convention that $C$ denotes some absolute positive constant, that is useless to specify. The value of $C$ may change 
 at different occurrences.
 
 \subsubsection{The case $l\equiv m\equiv n\equiv 0\mods q$}
 The contribution of such terms  to the right--hand side of  \eqref{3083} is trivially bounded by
 $$\ll_\delta q^{C\delta^{10}}\frac{LMN}{q^2}|K_{a,b,\pm c}(\xi, 0,0,0;q)|$$
 with
 $$K_{a,b,\pm c}(\xi, 0,0,0;q)=\frac 1q \cdot |\{(x,y,z)\in (\Fqt)^3,\ x^ay^bz^{\pm c}\equiv \xi\}|\ll_{a,b,c}q.$$ 
Hence the contribution of these terms to $S_{a,b,c,u}(\xi;L,M,N)$ (see \eqref{3083}) is bounded by
 $$\ll q^{C\delta^{10}}\frac{LMN}{q}\ll X^{1/2}q^{-\delta/2+\delta^2+C\delta^{10}}\leq X^{1/2}q^{-\delta/3}, $$
 which fits  \eqref{whatweneed}.  
 
 In the following two subsections, we will use the classical formula giving the number of solutions of a monomial congruence
 \begin{equation}\label{monomialequation}
 \vert \{x \bmod q, x^k \equiv t\}\vert = \sum_{\chi \bmod q \atop \chi^k =\chi_0} \chi (t),
 \end{equation}
 which is valid for any integer $k\geq 1$ and any integer $t$ not divisible by $q$. 
 \subsubsection{Exactly two of $l,m,n$ are $\equiv 0\mods q$} 
 If $l,m\equiv 0\mods q$ and $n\not\equiv 0\mods q$ (say) we have
$$K_{a,b,\pm c}(\xi, 0,0,n;q)=\frac{1}{q}\sum_{x,y\in\Fqt}\sum_{x^ay^bz^{\pm c}\equiv \xi}e_q(nz).$$
Let $a'=\gcd(a,q-1)$, $b'=\gcd(b,q-1)$. The map
$x\in\Fqt\ra x^a\in \Fqt$ has image $(\Fqt)^{a'}$ and the fibers of that map have size $a'$ and likewise for $x\ra x^b$ so the  above $K$--sum  equals
$$K_{a,b,\pm c}(\xi, 0,0,n;q)=\frac{1}{q}\sum_{x,y\in\Fqt}\sum_{x^{a'}y^{b'}z^{\pm c}\equiv \xi}e_q(nz).$$
Applying twice the equality \eqref{monomialequation} we deduce the relations
\begin{align}
K_{a,b,\pm c}(\xi, 0,0,n;q)&=\frac{1}{q}\sum_{\chi \bmod q \atop \chi^{a'} =\chi_0} 
\sum_{\psi \bmod q \atop \psi^{b'} =\chi_0} \sum_{u\in\Fqt}  \sum_{v\in\Fqt} \chi (u) \psi (v) \sum_{uvz^{\pm c}\equiv \xi\mods q}e_q(nz)\nonumber\\
& = \frac{1}{q}\sum_{\chi \bmod q \atop \chi^{a'} =\chi_0} \chi (\xi)\sum_{\psi \bmod q \atop \psi^{b'} =\chi_0} 
 (\sum_{z\in \Fqt} \chi^{\mp c}(z)e_q (nz) )  (\sum_{v\in \Fqt} (\overline{\chi}
\psi )(v)).\label{easy}
\end{align}
We observe that the innermost $v$-sum is zero unless $\psi=\chi$. So, in the above sum,  we are searching for pairs $(\chi, \psi)$
such that $\overline{\chi} \psi =\chi_0$.
Introduce  $e:=\gcd(a',b') $. The set of pairs $(\chi,\psi)\mods q$ with $\chi^{a'}=\psi^{b'}=\chi_0$ and satisfying ${\chi}= \psi $ is the set of pairs 
$$\{(\chi, \chi),\ \chi^e = \chi_0\}.$$ After this remark, we transform 
   \eqref{easy} into  the equality
$$
K_{a,b,\pm c}(\xi, 0,0,n;q) = \frac{q-1}{q}\sum_{\chi^e =\chi_0} \chi (\xi) 
 (\sum_{z\in \Fqt} \chi^{\mp c}(z)e_q (nz) ).
$$
Observe that $\gcd(e,c)=1$ (since $\gcd(a,b,c)=1$) so if we denote by $c'\in [1,e)$ the multiplicative inverse of $c\mods e$, we have
\begin{equation}\label{18}
K_{a,b,\pm c}(\xi, 0,0,n;q)=\frac{q-1}{q}\sum_{\chi^e =\chi_0} \chi (\xi') 
 (\sum_{z\in \Fqt} \chi(z)e_q (nz) ),
\end{equation}
where $\xi'=\xi^{\mp c'}$.
The $z$-sum equals
$$\sum_{z\in \Fqt} \chi(z)e_q (nz)=\begin{cases}
	-1&\hbox{ if }\chi=\chi_0,\\
	\ov\chi(n)G(\chi) &\hbox{ if }\chi\not=\chi_0.
\end{cases}$$
with
$$G(\chi)=\sum_{x\in \Fqt}\chi(x)e_q(x)$$
the un-normalized Gauss sum.

Return to \eqref{3083}.The contribution (via \eqref{18})   to $
S_{a,b,c,u}(\xi;L,M,N)$ of the character $\chi=\chi_0$ is bounded by
\begin{equation*}
	\ll q^{C\delta^{10}}\frac{LMN}{q^2}N^*\ll q^{C\delta^{10}}\frac{LM}{q}\leq q^{C\delta^{10}}\frac{LMN}{q}\ll X^{1/2}q^{-\delta/2+\delta^2+C\delta^{10}}\leq X^{1/2}q^{-\delta/3}.
\end{equation*}

The contribution of the remaining $e-1$ non-trivial  characters $\chi$ is treated by the P\'olya-Vinogradov bound, applied to the $\ov\chi(n)$-factor. It is bounded by
$$
\ll eq^{C\delta^{10}}\frac {LMN}{q^{2}}\min(N^*,q^{1/2})q^{1/2}\leq q^{C\delta^{10}}\frac {LMN}{q}\ll X^{1/2}q^{-\delta/3} 
$$
(if $\delta\leq \delta_0$ small enough)  which is \eqref{whatweneed}. 

  \subsubsection{Exactly one of $(l,m,n)$ is $\equiv 0\mods q$}
  If $l\equiv 0\mods q$ and $mn\not\equiv 0\mods q$ we have
$$K_{a,b,\pm c}(\xi, 0,m,n;q)=\frac{1}{q}\sum_{x\in\Fqt}\sum_{x^ay^bz^{\pm c}\equiv \xi}e_q(my+nz)=\frac{1}{q^{1/2}}\sum_{x\in\Fqt}K_{b,\pm c}(\xi x^{-a}m^bn^{\pm c})$$
where
$$K_{b,\pm c}(u)=\frac{1}{q^{1/2}}\sum_{y^bz^{\pm c}=u}e_q(y+z).$$
Let $a':=\gcd(a,q-1) $;
 we have
  $$K_{a,b,\pm c}(\xi,0,m,n;q)=\frac{a'}{q^{1/2}}\sum_{t\in(\Fqt)^{a'}}K_{b,\pm c}(\xi tm^bn^{\pm c})=\frac{1}{q^{1/2}}\sum_{\chi^{a'}=\chi_0}\sum_{t\in\Fqt}\chi(t)K_{b,\pm c}(\xi tm^bn^{ \pm c}).$$
  Writing $t=y^bz^{\pm c}/(\xi m^bn^{\pm c})$ we see that the innermost $t$-sum equal
  $$\frac{1}{q^{1/2}}\sum_{t\in\Fqt}\chi(t)K_{b,\pm c}(\xi tm^bn^{-c})=\ov\chi(\xi)\ov\chi^b(m)\ov\chi^{\pm c}(n)\frac{G(\chi^b)G(\chi^{\pm c})}{q}.$$
  Observe that if $\chi^b=\chi_0$ the factor $\ov\chi(m)^b$ does not oscillate but $G(\chi^b)=-1$ while for $\chi^b\not=\chi_0$, $|G(\chi^b)|=q^{1/2}$ but we can use the P\'olya-Vinogradov method on the factor $\ov\chi^b(m)$ to bound the $m$-sum by $q^{1/2+C\delta^{10}}$.
  
  From this, we conclude that the contribution to $S_{a,b,c,u}(\xi;L,M,N)$ (see \eqref{3083}) of this  kind of terms is bounded by a sum of four terms depending on whether $\chi^b$ and $\chi^c$ are trivial or not. More precisely this contribution is    bounded by
\begin{gather*}
	\ll q^{C\delta^{10}}\frac{LMN}{q^2}(\frac{q^2}{MN}\cdot \frac{1}{q}+\frac{q}{M}\cdot \frac{1}{q^{1/2}}\cdot q^{1/2}+\frac{q}{N}\cdot \frac{1}{q^{1/2}}\cdot q^{1/2}+q^{1/2}\cdot q^{1/2})\\
	\leq q^{C\delta^{10}}(\frac{L}{q}+\frac{LN}{q}+\frac{LM}{q}+\frac{LMN}{q})\ll X^{1/2}q^{-\delta/2+\delta^2+C\delta^{10}},
\end{gather*}
 which again implies \eqref{whatweneed}.
\subsection{The non-degenerate terms}
In this section we handle the contribution of the $(l,m,n)$ in \eqref{3083} such that $lmn\not\equiv 0\mods q$. This contribution is given by
 \begin{equation*}
 S^{ndeg}_{a,b,c,u}(\xi;L,M,N):=\frac{LMN}{q^2}\sumsum_{lmn\not\equiv 0\mods q}\what{U_u}(\frac{l}{L^*})\what{U_u}(\frac{m}{M^*})\what{U_u}(\frac{n}{N^*})K_{a,b,\pm c}(\xi l^am^bn^c;q),
 \end{equation*}
where $K_{a,b,\pm c}(u;q)$ is defined in \eqref{Kabcdefintro}.

We perform again a smooth dyadic partition of unity and decompose the sum into $O(\log^3q)$ dyadic sums of the shape
$$\frac{LMN}{q^2}\sumsum_\stacksum{l,m,n\sim L'\times M'\times N'}{lmn\not\equiv 0\mods q}V_u(\frac{l}{L'})V_u(\frac{m}{M'})V_u(\frac{n}{N'})K_{a,b,\pm c}(\xi l^am^bn^c;q)=\Sigma'_{a,b,c}(L',M',N')$$
with
$$L'\leq q^{1+\delta^2}/L, M'\leq q^{1+\delta^2}/M,\ N'\leq q^{1+\delta^2}/N.$$
The trivial bound gives 
\begin{equation}\label{trivialboundS'}
	\Sigma'_{a,b,c}(L',M',N')\ll q^{C\delta^{10}}\frac{LMN}{q^2}L'M'N'
\end{equation}
($C\geq 0$ an absolute constant). This is bounded by $O(X^{1/2}q^{-\Delta})$
as long as
$$q^{C\delta^{10}}L'M'N'\leq \frac{q^3}{LMN}q^{-\delta-\Delta}.$$
We may therefore assume that 
\begin{equation}
	\label{L'M'N'upperlower}
	q^{1-\Delta-\delta^2-C\delta^{10}}\leq \frac{q^{3}}{LMN}q^{-\delta-\Delta-C\delta^{10}}\leq L'M'N'\leq \frac{q^{3+3\delta^2}}{LMN},
\end{equation}
where the first inequality follows from \eqref{X=q} and \eqref{LMNrange}. 

We set 
\begin{equation}\label{lambdamunu}
q^\lambda=L',\ q^\mu=M',\ q^\nu=N'
\end{equation}
with $\lambda,\mu,\nu\geq 0$ (since $L,\, M,\, N \geq 1.$). By \eqref{L'M'N'upperlower} we have
\begin{equation}\label{rabbit}
1-\Delta-\delta^2-C\delta^{10}\leq \lambda+\mu+\nu\leq 3+3\delta^2.
\end{equation}
Suppose that within the range \eqref{L'M'N'upperlower}, we can improve  the trivial bound \eqref{trivialboundS'} by a factor $q^{-\Delta'}$, we obtain 
\begin{equation*}
\Sigma'_{a,b,c}(L',M',N')\ll q^{C\delta^{10}}\frac{LMN}{q^2}L'M'N'q^{-\Delta'}\leq q^{1+4\delta^2-\Delta'}=X^{1/2}q^{\delta/2+4\delta^2-\Delta'}
\end{equation*}
which is $O(X^{1/2}q^{-\Delta})$ as long as
\begin{equation}
	\label{Delta'lowerbound}
	\Delta'\geq \delta/2+ 4\delta^2+\Delta.
\end{equation}
We will see that $\Delta'$ can be taken to be an absolute constant.

Let $k\geq 32$ be some fixed integer and suppose that
$$\max(\lambda,\mu,\nu) \geq {1/2+3/(4k)}.$$

Since the trace functions $$x\mapsto K_{a,b,\pm c}(\xi x^am^bn^c;q),\ x\mapsto K_{a,b,\pm c}(\xi l^ax^bn^c;q),\ x\mapsto K_{a,b,\pm c}(\xi l^am^bx^c;q)$$ are attached to Fourier sheaves, the P\'olya-Vinogradov method saves a factor $$\gg q^{-3/(4k)+C\delta^{10}}$$ over the trivial bound, so that \eqref{Delta'lowerbound} holds for any $\delta\leq \delta_0$ chosen sufficiently small (depending on $k$).
 
We may therefore assume that
\begin{equation}\label{bingo}
\lambda,\mu,\nu< 1/2+3/(4k).
\end{equation}

\subsubsection{When $(a,b,\pm c)$ is \good}\label{firstcase}
To fix ideas, we  assume that $$\lambda\leq\mu\leq \nu$$ (as we will see the forthcoming argument is symmetric in $l,m,n$).

Suppose that
\begin{equation}
	\label{typeIIcond}
	3/(2k)\leq\mu,\ \mu+\nu\geq 3/4+1/k
\end{equation}
We appeal to  Theorem \ref{thmTypeII}  combined with a trivial summation over $l$. Doing that way we save  a factor
$$
\gg \max\{ q^{-3/(4k)}, q^{-1/(8k^2)}\} \, q^{C\delta^2}= q^{-1/(8k^2) +C\delta^2},
$$
over the trivial bound which is again satisfactory for any $\delta\leq \delta_0$ chosen sufficiently small (depending on $k$).

If \eqref{typeIIcond} does not hold then we have either $\mu+\nu\leq 3/4+1/k$ or $\mu\leq 3/(2k)$. \\
The later cannot occur: otherwise we would have $\lambda+\mu\leq 3/k$ (by \eqref{bingo}) and then (by \eqref{rabbit} and the value of $k$)  $$\nu\geq 1-\Delta-\delta^2-C\delta^{10}-3/k\geq 1/2+3/(4k)$$ (if $\delta$ is sufficiently small) which we have previously excluded by \eqref{bingo}.

Thus we are in the situation where
$$
\mu +\nu \leq 3/4 +1/k \text{ and } \mu \geq 3/(2k).$$

By \eqref{rabbit} and the value of $k$, it follows that
\begin{equation}\label{board1}
\lambda\geq \frac14-(\Delta+\delta^2)-C\delta^{10}-\frac{1}{k}\geq \frac{7}{32}-\Delta-\delta^2-C\delta^{10},
\end{equation}
and
\begin{align}
	\mu+\nu&\geq \frac{1}3(1-\Delta-\delta^2-C\delta^{10})+\frac14-(\Delta+\delta^2)-C\delta^{10}-\frac 1k\nonumber\\
	&=\frac7{12}-\frac1k-\frac{4}3(\Delta+\delta^2)-\frac 43 C\delta^{10}\geq\frac{1}2+\frac{5}{96}-\frac{4}3(\Delta+\delta^2)-C\delta^{10}.\label{board2}
\end{align}

For $\delta$ sufficiently small, we  can then use Theorem \ref{thmtriplesum} with some parameter $k'$ (instead of $k$) fixed but sufficiently large so that 
$$k'\lambda+\mu+\nu\geq 11/10$$
(say) and we obtain a saving by a factor
$$q^{-\Delta'+C\delta^{10}}$$
for  $\Delta'=\Delta'(k,k')>0$.

This establishes Proposition \ref{differencelemma} when $(a,b,\pm c)$ is \good, since we investigated all the sorts of terms participating in $ S_{a,b,c,u}(\xi;L,M,N)$ (see \eqref{3083}).
\subsubsection{When $(a,b,\pm c)$ is oxozonic.}
 
If $(a,b,\pm c)$ is oxozonic we have to use Theorem \ref{thmO4} in place of Theorems  \ref{thmTypeII} and \ref{thmtriplesum}. In particular our argument is no longer exactly symmetric in the three variables (because of the requirement that the exponent $c$  be odd to obtain the bound \eqref{BtypeIIort} in Theorem \ref{thmO4}).  
What is making it possible to deal with  the oxozonic case, is that up to exchanging  $a$ and $b$, we have $(a,b,c) =(1,1,2)$ or $(a,b,-c)=(1,4,-3),\ (1,6,-3)$ or $(2,3,-1)$; we then observe that  exactly  one of these three exponents is even (see Definition \ref{defgoodtriple}).

To set-up the discussion we increasingly reorder the exponents  (see \eqref{lambdamunu}) 
$$\{\lambda,\mu,\nu\}=\{\lambda', \mu', \nu'\}.$$
 with $$\lambda'\leq \mu'\leq \nu'.$$
We now discuss with  which exponent is associated the unique even  element in the oxozonic triple $(a,b,\pm c)$.\\
$\bullet$\  Suppose that the two exponents (ie. $a,b$ or $\pm c$) corresponding to the largest ranges $\mu', \nu'$ are both {\em odd}, then the argument given above goes through: if $(\mu',\nu')$ satisfy \eqref{typeIIcond} we apply the bound \eqref{BtypeIIort} on the corresponding variables (averaging trivially over the remaining variable) and the bound \eqref{tripleort} if they don't. For instance this is the case when $(a,b,-c)=(2,3,-1)$ and $\lambda\leq \mu\leq \nu$. \\
$\bullet $\ In fact  this discussion works if we only assume that   the exponent corresponding to $\nu'$ is  odd. For instance when $(a,b,-c)=(1,6,-3)$ and $\lambda\leq \mu\leq \nu$.\\
$\bullet$ \ Suppose now that the exponent corresponding to the largest range $\nu'$ is even: for instance if $(a,b,-c)=(1,4,-3)$ and $L'\leq N'\leq M'$.

-- If $(\mu',\nu')$ satisfies  \eqref{typeIIcond} we have (since $\mu'\leq \nu'$)
\begin{equation}
	\label{typeIIcondoxo}
	3/(2k)\leq \nu',\ \mu'+\nu'\geq 3/4+1/k,
\end{equation}
and the bound \eqref{BtypeIIort} is valid for the "$m$" variable corresponding to that of the range $\sim q^{\nu'}$ and the "$n$" variable  corresponding to that of the range $q^{\mu'}$.

-- On the other hand, if \eqref{typeIIcondoxo} does not hold, we then have (see the inequalities \eqref{board1} and \eqref{board2} in the discussion in \S \ref{firstcase})
$$\lambda'\geq \frac{7}{32}-\Delta-\delta^2-C\delta^{10},$$
and
$$
	\mu'+\nu'\geq \frac{1}2+\frac{5}{96}-\frac{4}3(\Delta+\delta^2)-C\delta^{10};
$$
We can then use the bound \eqref{tripleort} which has no requirement on the parity of  $a,b,c$.

This completes the proof of Proposition 
\ref{differencelemma} in the oxozonic case. 
\qed

\section{Completion of the proofs } 
We now put all these estimates together in order to prove Theorems \ref{mainthm}, \ref{mainthmconj} and \ref{mainthmaverage} and Corollary \ref{Ldonotvanish}.

\subsection{Proof of Theorem \ref{mainthm}}

\label{secproofmainthm}

We   summarize the strategy we worked out to deal with the triple moment $M_{ad,bd,\pm cd} (q)$ of values of functions $L$ at the central point $s=1/2$ (see \eqref{triplemoment}).
Here $(a,b,c)$ is a triple of setwise positive   integers, supposed to
be \good\ or oxozonic and $ d\geq 1$ is an integer. \\
In order to come down to coprime exponents, we decompose $M_{ad,bd,\pm cd} (q)$ in a sum of the twisted moments $M_{a,b,\pm c}(\xi ;q)$ over $\xi\in \mu_{(d,q-1)}(\Fq)$ (see \eqref{d'decomposition} and \eqref{triplemomentxi}). The set of such $\xi$ always contains $\xi =1$. \\
We decompose $M_{a,b,\pm c} (\xi; q)$ into the subsum over even and odd characters $\chi$, respectively called $M_{a,b,\pm c}^\bullet (\xi; q)$ with $\bullet = e$ or $o$ (see \eqref{alreadydone}).\\
We apply the AFE, with the choice of parameters 
$$X=q^{2-\delta},\ Y=q^{1+\delta}$$ for some fixed $\delta$ sufficiently small (see \eqref{XYchoice}).  We obtain formulas \eqref{Mevenabcbasicdecomp} and \eqref{Moddabcbasicdecomp} which have similar shapes. In particular the fifth terms of these identities are both of size $O(q^{-\eta})$ (with $\eta =\eta (\delta) >0$ hence is negligible) with the above choice of $X$ and $Y$.\\

Likewise, using Theorem \ref{M2abcthm} and Proposition \ref{trilinearoxozonic} we see that the four terms  of type $$M_{2,a,b,\pm c}^{\bullet} (\eps.  \xi;q),\ \bullet\in \{e, o\},\ \eps\in\{1,-1\}$$ present in \eqref{Mevenabcbasicdecomp} and \eqref{Moddabcbasicdecomp} are  also bounded by $O( q^{-\eta})$ hence are negligible.\\
Thus we  have produced  the equalities
\begin{equation}\label{3354}	
\begin{aligned}
	M^{e}_{a,b,\pm c}(\xi;q)&=M^e_{1,a,b,\pm c}(\xi;q)+ M^e_{1,a,b,\pm c}(-\xi;q)+O(q^{-\eta}) \\
	M^{o}_{a,b,\pm c}(\xi;q)&=M^o_{1,a,b,\pm c}(\xi;q)- M^o_{1,a,b,\pm c}(-\xi;q)+O(q^{-\eta}).
\end{aligned} 
\end{equation}
The terms $M^{\bullet}_{1,a,b,\pm c}(\pm \xi;q)$ are of a special nature : they count (with a smooth positive  weight) the number of solutions of some congruence equation modulo $q$.  In particular these terms are sums of positive numbers (see 
\eqref{M1abcdef}). By dropping any of these terms  we obtain a lower bound for  $M_{1,a,b,\pm c}^{\bullet} (\pm\xi;q)$. 
For instance, for any $\xi$,  we have the inequality $M_{1,a,b,\pm c}^{\bullet} (\pm \xi;q)\geq 0$ and  the first line of \eqref{3354} leads to the inequality
\begin{equation}\label{zerolowerbound}
	M_{a,b,\pm c}^{e} (\pm \xi;q)\geq O (q^{-\eta}).
\end{equation}

This is exactly \eqref{2147} in the case $\xi^2\neq 1$.

\begin{remark}
	This lower bound does not apply to $M_{a,b,\pm c}^{o} (\xi;q)$ because of the presence of the minus sign on the second line of \eqref{3354}.
\end{remark}	 

Consider the case $\xi =1$; in $M^{\bullet}_{1,a,b,c}(1;q)$, we keep only the contribution of the triples $(l,m,n)$ satisfying term $l^am^bn^{\pm c}=1$ which we call $M^{\bullet, \main}_{1,a,b,\pm c}(1;q)$; by positivity we obtain the inequality $$
M^{\bullet}_{1,a,b,\pm c}(1;q)\geq M^{\bullet, \main}_{1,a,b,\pm c}(1;q).$$
Equation \eqref{sharp} and Proposition \ref{MTprop} produce the asymptotic formula
$$
M^{\bullet, \main}_{1,a,b,\pm c}(1;q)=D_{a,b,\pm c} +O( q^{-\eta}).
$$
We combine this relation and \eqref{zerolowerbound}  into the following inequality: for $\bullet=e\hbox{ or }o$ we have
\begin{equation}\label{sunny}
M^{\bullet}_{1,a,b,\pm c}(\xi;q)\geq
\begin{cases}
D_{a,b,\pm c}+ O (q^{-\eta}) &\text{if } \xi =1,\\
0&\text{ otherwise.}
\end{cases}
\end{equation}
 Returning to the first line of \eqref{3354} and using the non-negativity of $M_{1,a,b,\pm c}^{e} (-1;q)$ we obtain   the inequality
 \eqref{2147} for $\xi =1$ or $-1$. This completes the proof of \eqref{2147}. 
 
 We now turn to the proof of  \eqref{Mxilowerbound}. Returning to the decomposition \eqref{evenodddecomp} and summing the two lines of \eqref{3354}, we obtain the equality
 $$
 M_{a,b,\pm c}(\xi;q) =\frac 12 (  M^e_{1,a,b,\pm c}(\xi;q)+ M^o_{1,a,b,\pm c}(\xi;q) ) + \frac 12( M^e_{1,a,b,\pm c}(-\xi;q)- M^o_{1,a,b,\pm c}(-\xi;q)) +O (q^{-\eta}).
 $$
Using the inequality \eqref{desiredboundforM1} of Proposition \ref{differencelemma}  this simplifies to
\begin{equation*}
 M_{a,b,\pm c}(\xi;q) =\frac 12 (  M^e_{1,a,b,\pm c}(\xi;q)+ M^o_{1,a,b,\pm c}(\xi;q) )   + O (q^{-\eta}).
 \end{equation*}
 The lower bound \eqref{sunny} completes the proof of \eqref{Mxilowerbound}.
 \vskip .3cm
The lower bound  \eqref{Mxilowerbound*} is a direct consequence of \eqref{d'decomposition} and of  \eqref{Mxilowerbound} since the sum contains the term corresponding to $\xi =1$ and negligible terms. 
\vskip .3cm
The proofs of the four last relations \eqref{momenta=basymp},  \eqref{momenta=basympodd}, \eqref{Mxilowerbound**} and \eqref{Mxilowerbound***} are simple variations of the above using the further assumption $a=b$. Indeed by Theorem  \ref{thmaac} we know that Conjecture $\msP(a,a,\pm c,d)$ holds and Proposition
\ref{M1errlemma}  allows us to replace \eqref{sunny} by the more precise equalities
\begin{equation}\label{sunny*}
M^{\bullet}_{1,a,a,\pm c}(\xi;q)=
\begin{cases}
D_{a,a,\pm c}+ O (q^{-\eta}) &\text{if } \xi =1,\\
 O (q^{-\eta})&\text{ otherwise,}
\end{cases}
\end{equation}
 for $\bullet = e$ or $o$.  Inserting \eqref{sunny*} in \eqref{3354} we complete the proof of \eqref{momenta=basymp} and 
 \eqref{momenta=basympodd}.  When summing  \eqref{momenta=basymp} and 
 \eqref{momenta=basympodd}, we obtain \eqref{Mxilowerbound**} thanks to \eqref{alreadydone}. \\
 Finally \eqref{Mxilowerbound***}  follows from \eqref{Mxilowerbound**} and the summation formula
 \eqref{d'decomposition}. \\ The proof of  Theorem \ref{mainthm} is complete.

 \subsection{Proof of  Theorem \ref{mainthmconj}}  Proposition \ref{M1errlemma} gives a suitable upper bound for the error term
 $M_{1,a,b,\pm c}^{\bullet, \err}(\xi; q)$   as soon as Conjecture $\msP(a,b,\pm c,d )$  holds and  $\xi$ satisfies   $\xi^d =1$. 
 When expanding $M_{ad,bd,\pm cd} (q)$ we exhibit linear combination of terms $M_{a,b,\pm c}^{\bullet} (\xi;q)$ and $M_{a,b,\pm c}^{\bullet}(-\xi;q)$, with $\xi^{(d, q-1)} =1.$ In order to control these terms, we require the truth of Conjecture $\msP(a,b,\pm c,d')$
 for any  $d'$ such that $\xi^{d'}=1$ and $(-\xi)^{d'}=1$. The choice $d'=d_q$ satisfies this requirement. And the proof is the same as above in the case of \eqref{Mxilowerbound***}.
 \subsection{Proof of Theorem \ref{mainthmaverage}} We follow the same steps as in the proof of   Theorem \ref{mainthmconj} using Corollary \ref{M1erraverage} (in place of Proposition
\ref{M1errlemma}) for the tuple $(a,b,c,2d)$ (a consequence on the fact that Conjecture $\msP(a,b,c,2d)$ holds on average over $q\in[Q,2Q)$).  

\subsection{Proof of Corollary \ref{Ldonotvanish}}\label{16h50}
We start from a triple $(a,b,c)$ of non zero  integers.  We now make some reductions 
on the parameters $a$, $b$ and $c$. Since $L(1/2, \chi^{-a}) \neq 0 \iff L(1/2, \chi^a)\neq 0$, we may assume that $a$, $b$ and $c$ are positive and distinct (indeed if $a=b$, the non--vanishing of the triple product  reduces to the non--vanishing of the double product $L(1/2,\chi^a) L(1/2, \chi^c)$ and is covered by
\cite{FKMAA}*{p. 111}).

So we are led to study the case where $a>b>c \geq 1.$  Let $d$ be the g.c.d. of  these three integers. Write
$$
a=da_1, \ b= db_1, \, c= dc_1;$$
the triple $(a_1,b_1,c_1)$ is made of coprime integers, satisfying $a_1>b_1>c_1\geq 1.$ By Definition \ref{defgoodtriple}, the triple $(a_1,b_1,c_1)$ is  necessarily {\good} and Theorem \ref{mainthm}  \eqref{Mxilowerbound*} applies with $D_{a_1,b_1,c_1}=1$: we have
$$\frac{1}{q-1}\sum_{\chi\mods q}L(1/2,\chi^a)L(1/2,\chi^b)L(1/2,\chi^c)\geq 1+O_{a,b,c}(q^{-\eta}).$$
  Set
$$\Xi_{a,b,c}(q)=\{\chi\mods q,\  L(1/2,\chi^a)L(1/2,\chi^b)L(1/2,\chi^c)\not=0\}$$
Using \eqref{Mxilowerbound*}    and H\"older's inequality we have
\begin{align*}
    1+O_{a,b,c}(q^{-\eta})&\leq \frac{1}{q-1}\sum_{\chi \mods q} L(1/2,\chi^a)L(1/2,\chi^b)L(1/2,\chi^c)\\
    &\leq (\frac{1}{q-1}\Xi_{a,b,c}(q))^{1/4}\prod_{h\in\{a,b,c\}}(\frac{1}{q-1}\sum_{\chi\mods q}|L(1/2,\chi^h)|^4)^{1/4}\\
    &\ll_{a,b,c} (\log q)^{3}(\frac{1}{q-1}\Xi_{a,b,c}(q))^{1/4}.
\end{align*}
In the last step we have used that, for $h\neq 0$, one has 
$$\frac{1}{q-1}\sum_{\chi\mods q}|L(1/2,\chi^h)|^4\leq \frac{(h,q-1)}{q-1}\sum_{\chi\mods q}|L(1/2,\chi)|^4\ll (\log q)^4$$
(see \cite{young} for an asymptotic formula for the fourth moment with a power saving error term). 
We deduce the lower bound $\vert \Xi_{a,b,c}(q)\vert \gg q(\log q)^{-12}$. 
\qed

\end{document}